\pdfoutput=1
\RequirePackage{ifpdf}
\ifpdf 
\documentclass[pdftex]{sigma}
\else
\documentclass{sigma}
\fi

\usepackage{cases,mathdots,fouridx}
\usepackage{tikz}

\numberwithin{equation}{section}

\newtheorem{Theorem}{Theorem}[section]
\newtheorem{Corollary}[Theorem]{Corollary}
\newtheorem{Lemma}[Theorem]{Lemma}
\newtheorem{Proposition}[Theorem]{Proposition}
\newtheorem{Conjecture}[Theorem]{Conjecture}

\renewcommand{\sc}{\scriptstyle}

\newcommand{\ip}[2]{\langle#1,#2\rangle}
\newcommand{\qhyp}[5]{\fourIdx{}{#1}{}{#2}\phi
\bigg[\genfrac{}{}{0pt}{}{#3}{#4};#5\bigg]}

\renewcommand{\leq}{\leqslant}
\renewcommand{\geq}{\geqslant}

\DeclareMathOperator{\Par}{Par}

\DeclareMathOperator{\SSYT}{SSYT}
\DeclareMathOperator{\CSSYT}{CSSYT}
\DeclareMathOperator{\symp}{sp}
\DeclareMathOperator{\so}{so}
\DeclareMathOperator{\oeven}{o}

\DeclareMathOperator{\mult}{mult}
\DeclareMathOperator{\lev}{lev}
\DeclareMathOperator{\eup}{e}

\DeclareMathOperator{\ch}{ch}
\DeclareMathOperator{\sgn}{sgn}
\newcommand{\odd}{\textup{o}}
\newcommand{\ceil}[1]{\lceil#1\rceil}
\newcommand{\floor}[1]{\lfloor#1\rfloor}
\newcommand{\qbin}[2]{\genfrac{[}{]}{0pt}{}{#1}{#2}}
\newcommand{\qBin}[2]{\Bigg[\genfrac{}{}{0pt}{}{#1}{#2}\Bigg]}
\newcommand{\abs}[1]{\vert#1\vert}
\newcommand{\la}{\lambda}
\newcommand{\La}{\Lambda}

\begin{document}

\allowdisplaybreaks

\newcommand{\arXivNumber}{2511.17034}

\renewcommand{\thefootnote}{}

\renewcommand{\PaperNumber}{062}

\FirstPageHeading

\ShortArticleName{Affine Jacobi--Trudi Identities and $q,t$-Rogers--Ramanujan Identities}

\ArticleName{Affine Jacobi--Trudi Identities\\ and $\boldsymbol{q,t}$-Rogers--Ramanujan Identities\footnote{This paper is a~contribution to the Special Issue on Recent Advances in Vertex Operator Algebras in honor of James Lepowsky. The~full collection is available at \href{https://sigma-journal.com/Lepowsky.html}{https://sigma-journal.com/Lepowsky.html}}}

\Author{S. Ole WARNAAR}

\AuthorNameForHeading{S.O.~Warnaar}

\Address{School of Mathematics and Physics, The University of Queensland, Brisbane, Australia}
\Email{\mail{o.warnaar@maths.uq.edu.au}}
\URLaddress{\url{https://people.smp.uq.edu.au/OleWarnaar/}}

\ArticleDates{Received November 24, 2025, in final form June 05, 2026; Published online June 25, 2026}

\Abstract{We conjecture affine or Hall--Littlewood analogues of the dual Jacobi--Trudi identities for orthogonal and symplectic Schur functions indexed by rectangular partitions of maximal height. These conjectures are then used to derive $t$-analogues of many known Rogers--Ramanujan identities for the characters of standard modules of affine Lie algebras.
This includes $t$-analogues of the classical Rogers--Ramanujan identities, (some of) the Andrews--Gordon identities and the
\smash{$\mathrm{C}_n^{(1)}$}, \smash{$\mathrm{A}_{2n}^{(2)}$} and \smash{$\mathrm{D}_{n+2}^{(2)}$} GOW identities. We also prove an affine analogue of the dual Jacobi--Trudi identity for Schur functions indexed by rectangular partitions of arbitrary height.}


\Classification{05E05; 05E10; 11P84; 17B67; 33D15; 33D52}

\begin{flushright}
\begin{minipage}{85mm}
\it This paper is dedicated to Jim Lepowsky, a pioneer in the study
of Rogers--Ramanujan-type identities from the point of view of
representation theory.
\end{minipage}
\end{flushright}

\renewcommand{\thefootnote}{\arabic{footnote}}
\setcounter{footnote}{0}

\section{Introduction}

\subsection{The Jacobi--Trudi and dual Jacobi--Trudi identities}

A partition $\la$ of length $l(\la)$ equal to $\ell$ is a weakly
decreasing sequence $(\la_1,\la_2,\dots)$ of nonnegative integers such
that $\la_i>0$ for $i\leq\ell$ and $\la_i=0$ for $i>\ell$.
For $x=(x_1,\dots,x_n)$ and $\la$ a partition, the Schur function
$s_{\la}(x)$ is defined as
\begin{equation*}
s_{\la}(x)=\frac{\det_{1\leq i,j\leq n}\bigl(x_i^{\la_j+n-j}\bigr)}
{\prod_{1\leq i<j\leq n}(x_i-x_j)}
\end{equation*}
if $l(\la)\leq n$, while $s_{\la}(x)=0$ if $l(\la)>n$.
It is a standard result in the theory of symmetric functions
that $s_{\la}(x)$ for $l(\la)\leq n$
corresponds to the character of the polynomial representation of
$\mathrm{GL}(n,\mathbb{C})$ indexed by $\la$ and admits the
combinatorial description
\begin{equation}\label{Eq_sT}
s_{\la}(x)=\sum_{T\in\SSYT_n(\la)} x^T.
\end{equation}
Here $\SSYT_n(\la)$ is the set of semistandard Young tableaux
of shape $\la$ on $\{1,\dots,n\}$, and
${x^T:=x_1^{\alpha_1}\cdots x_n^{\alpha_n}}$, where
$\alpha_i$ is the number of boxes or squares of $T$ with filling $i$.

The complete and elementary symmetric functions $h_r$ and $e_r$
are the Schur functions for partitions whose Young diagram consists of
a single row or column of $r$ boxes, respectively.
That is,
\[
h_r=s_{(r)} \qquad \text{and}\qquad
e_r=s_{(\!\underbrace{\scriptstyle{1,\dots,1}}_{r \text{ times}}\!)}=s_{(1^r)},
\]
or, more explicitly,
$h_0=e_0=1$ and
\begin{align*}
 h_r(x)=\sum_{1\leq i_1\leq i_2\leq \cdots \leq i_r\leq n}
x_{i_1}\cdots x_{i_r},\qquad
 e_r(x)=\sum_{1<i_1<i_2<\dots<i_r\leq n}
x_{i_1}\cdots x_{i_r}
\end{align*}
for $r$ a positive integer.
Hence $e_r(x)=0$ for $r>n$.
Also defining $e_r=h_r=0$ for $r<0$, two classical results in
the theory of symmetric functions, known as the Jacobi--Trudi
identity and dual Jacobi--Trudi (or N\"{a}gelsbach--Kostka) identity,
express the Schur functions in terms of determinants
with entries given by the complete and elementary symmetric functions,
respectively,
\begin{align}
s_{\la}(x)&=\det_{1\leq i,j\leq n} (h_{\la_i-i+j}(x) )
 =\det_{1\leq i,j\leq k} \bigl(e_{\la'_i-i+j}(x)\bigr).
\label{Eq_JTe}
\end{align}
In \eqref{Eq_JTe}, the partition $\la'$ is the conjugate of $\la$ and
$k$ is an arbitrary integer such that $\la'_1\leq k$.

Surprisingly little appears to be known about Jacobi--Trudi identities
for important generalisations of the Schur functions such as the
Hall--Littlewood, Jack and Macdonald polynomials.
Matsumoto \cite[Theorem 5.1]{Matsumoto08} discovered
Jacobi--Trudi-like formulas for the Jack polynomials
for partitions of rectangular shape, provided the Jack parameter
$\alpha$ is a positive integer or the reciprocal
of a positive integer.
In Matsumoto's formulas, determinants are replaced by hyperdeterminants
of (even) order depending on the value of $\alpha$.
This was subsequently generalised to partitions of near rectangular
shape by Belbachir, Boussicault and Luque \cite{BBL08}.
In this paper, we add to these results by proving a
dual Jacobi--Trudi formula for ordinary (or $\mathrm{GL}(n,\mathbb{C})$)
Hall--Littlewood polynomials indexed by rectangular shapes
and by conjecturing dual Jacobi--Trudi formulas
for $\mathrm{B}_n$, $\mathrm{C}_n$ and $\mathrm{BC}_n$
Hall--Littlewood polynomials indexed by rectangular
partitions of length $n$.


\subsection{Main results and conjectures}

For $x=(x_1,\dots,x_n)$ and $\la$ a partition,
let $P_{\la}(x;t)$ be the Hall--Littlewood polynomial indexed by~$\la$,
see Section~\ref{Sec_HL-Gln} for details.
In particular, $P_{\la}(x;0)=s_{\la}(x)$ and $P_{\la}(x;1)=m_{\la}(x)$,
where~$m_{\la}(x)$ is the monomial symmetric function.

Our first result is an affine (dual) Jacobi--Trudi formula for
Hall--Littlewood polynomials indexed by partitions of rectangular
shape
\[
(k^r)=(\underbrace{k,k,\dots,k}_{r \text{ times}}).
\]

\begin{Theorem}\label{Thm_GLn}
Let $k$ be a positive integer, $r$ a nonnegative integer and
$x=(x_1,\dots,x_n)$.
Then
\begin{equation}\label{Eq_Pe}
P_{(k^r)}(x;t)=
\sum_{\substack{y_1,\dots,y_k\in\mathbb{Z}\\[1pt] y_1+\dots+y_k=0}}\:
\det_{1\leq i,j\leq k}
\bigl(t^{k\binom{y_i}{2}+iy_i} e_{r-i+j-ky_i}(x) \bigr).
\end{equation}
\end{Theorem}

For $y=(y_1,\dots,y_k)\in\mathbb{Z}^k$, let $\abs{y}:=y_1+\dots+y_k$.
Then the right-hand side of \eqref{Eq_Pe} may also be stated as
\[
P_{(k^r)}(x;t)=
\sum_{y\in Q} \sum_{\sigma\in S_k} \sgn(\sigma) \prod_{i=1}^k
t^{k\binom{y_i}{2}+iy_i} e_{r-i+\sigma_i-ky_i}(x),
\]
where $Q=\big\{y\in\mathbb{Z}^k \mid \abs{y}=0\big\}$ and
\smash{$S_k\ltimes Q\cong W\bigl(\mathrm{A}_{k-1}^{(1)}\bigr)$},
the Weyl group of the affine root system~\smash{$\mathrm{A}_{k-1}^{(1)}$}
\cite{Kac90,Macdonald72}.
For integers $i$, $k$, $u$ such that $1\leq i\leq k$, let
${s_{i,k}(u):=k\binom{u}{2}+iu}$.
Then~(1)~$s_{i,k}(u)\geq 0$, (2) for $i<k$, $s_{i,k}(u)=0$
if and only if $u=0$, and (3) $s_{k,k}(u)=0$ if and only if
$u=0$ or $u=-1$.
These facts imply that for $y\in Q$,
\smash{$\sum_{i=1}^k \bigl(k\binom{y_i}{2}+iy_i\bigr)=\sum_{i=1}^k s_{i,k}(y_i)$}
is strictly positive unless $y=(0,\dots,0)$.
Consequently, for $t=0$ the only nonzero contribution to the sum over
$y$ in \eqref{Eq_Pe} comes from the zero vector.
Theorem~\ref{Thm_GLn} thus generalises the $\la=(k^r)$ case of the
dual Jacobi--Trudi identity \eqref{Eq_JTe}.

For $x=(x_1,\dots,x_n)$ and $\la=(\la_1,\dots,\la_n)$ a partition,
the odd orthogonal, even orthogonal and symplectic Schur functions
$\so_{2n+1,\la}(x)$, $\oeven_{2n,\la}(x)$ and $\symp_{2n,\la}(x)$
are defined as \cite{Littlewood50}\footnote{The full set of
odd-orthogonal Schur functions includes \eqref{Eq_SchurB} for
half-partitions $\la$, i.e., weakly decreasing $n$-tuples
$(\la_1,\dots,\la_n)$ such that $\la_i\in\mathbb{Z}+1/2$ and
$\la_n\geq\frac{1}{2}$.}
\begin{align}\label{Eq_SchurB}
\begin{split}
&\so_{2n+1,\la}(x)=\frac{\det_{1\leq i,j\leq n}
\bigl( x_i^{-\la_j+j-1}-x_i^{\la_j+2n-j} \bigr)}
{\prod_{i=1}^n (1-x_i)\prod_{1\leq i<j\leq n}(x_i-x_j)(x_ix_j-1)}, \\
&\oeven_{2n,\la}(x)=f_{\la} \frac{\det_{1\leq i,j\leq n}
\bigl( x_i^{-\la_j+j-1}+x_i^{\la_j+2n-j-1} \bigr)}
{\prod_{i=1}^n \prod_{1\leq i<j\leq n}(x_i-x_j)(x_ix_j-1)}, \\
&\symp_{2n,\la}(x)=\frac{\det_{1\leq i,j\leq n}
\bigl( x_i^{-\la_j+j-1}-x_i^{\la_j+2n-j+1} \bigr)}
{\prod_{i=1}^n (1-x_i^2)\prod_{1\leq i<j\leq n}(x_i-x_j)(x_ix_j-1)},
\end{split}
\end{align}
where $f_{\la}=1$ if $l(\la)=n$ and $f_{\la}=1/2$ if $l(\la)<n$.
In \eqref{Eq_SchurB}, we have followed the convention of writing
$\so_{2n+1,\la}(x)$ instead of $\oeven_{2n+1,\la}(x)$, which stems
from the fact that the characters of the irreducible polynomial
representations of $\mathrm{O}(2n+1,\mathbb{C})$ and
$\mathrm{SO}(2n+1,\mathbb{C})$ coincide.
In the even case, a similar coincidence between the characters of
$\mathrm{O}(2n,\mathbb{C})$ and $\mathrm{SO}(2n,\mathbb{C})$
only occurs for $l(\la)<n$.
From the definition, it follows that $\so_{2n+1,\la}(x)$,
$\oeven_{2n,\la}(x)$ and $\symp_{2n,\la}(x)$ are
$\mathrm{BC}_n$-symmetric Laurent polynomials
in the sense of \cite{RW21} in the variables $x_1,\dots,x_n$ with
integer coefficients, whose top-degree homogeneous components are
given by the Schur function $s_{\la}(x)$.
Assuming $\la\subseteq (k^n)$ , the symplectic and
odd-orthogonal Schur functions admit analogues of the dual
Jacobi--Trudi identity~\eqref{Eq_JTe} as follows
(see, e.g., \cite[p.~123]{Stembridge91} or
\cite[equations~(3.10) and~(3.27)]{FK97}):
\begin{align*}
&\so_{2n+1,\la}(x)=\det_{1\leq i,j\leq k}\bigl(
\dot{e}_{\la'_i-i+j}(x)+\dot{e}_{\la'_i-i-j+1}(x)\bigr), \\
&\oeven_{2n,\la}(x)=\frac{1}{2} \det_{1\leq i,j\leq k}\bigl(
\dot{e}_{\la'_i-i+j}(x)+\dot{e}_{\la'_i-i-j+2}(x)\bigr), \\
&\symp_{2n,\la}(x)=\det_{1\leq i,j\leq k}\bigl(
\dot{e}_{\la'_i-i+j}(x)-\dot{e}_{\la'_i-i-j}(x)\bigr),
\end{align*}
where $\dot{e}_r(x)$ is shorthand for
\smash{$e_r\bigl(x^{\pm}\bigr)=e_r\bigl(x_1,x_1^{-1},\dots,x_n,x_n^{-1}\bigr)$}.
For the rectangular partition of maximal height, i.e., for $\la=(k^n)$,
this yields
\begin{subequations}
\begin{align}
&\so_{2n+1,(k^n)}(x)=\det_{1\leq i,j\leq k} (
\dot{e}_{n-i+j}(x)+\dot{e}_{n+i+j-1}(x) ), \label{Eq_JTB} \\
&\oeven_{2n,(k^n)}(x)=\frac{1}{2}\det_{1\leq i,j\leq k} (
\dot{e}_{n-i+j}(x)+\dot{e}_{n+i+j-2}(x) ),
\label{Eq_JTO} \\
&\symp_{2n,(k^n)}(x)=\det_{1\leq i,j\leq k} (
\dot{e}_{n-i+j}(x)-\dot{e}_{n+i+j}(x) ), \label{Eq_JTC}
\end{align}
\end{subequations}
where we have also used the symmetry \smash{$\dot{e}_r(x)=\dot{e}_{2n-r}(x)$.}

For $\la\subseteq (k^n)$ and $x$ as above, let
\smash{$P_{\la}^{\mathrm{B}_n}(x;t,s)$} denote the $\mathrm{B}_n$ Hall--Littlewood
polynomial indexed by $\la$, see Section~\ref{Sec_HL-BCn}.
Then \smash{$P_{\la}^{\mathrm{B}_n}(x;0,0)=\so_{2n+1,\la}(x)$} and
\smash{$P_{\la}^{\mathrm{B}_n}(x;0,1)=\oeven_{2n,\la}(x)$}.
For the $\mathrm{B}_n$ Hall--Littlewood polynomial, we conjecture two
affine analogues of \eqref{Eq_JTB}, the right-hand sides of which may
be identified with the affine root systems \smash{$\mathrm{A}_{2k}^{(2)}$}
and \smash{$\mathrm{D}_{k+1}^{(2)}$}, respectively.
We~also have one analogue of \eqref{Eq_JTO}, for which we do not have
an identification in terms of affine root systems.

\begin{Conjecture}\label{Con_JTB}
Let $k$ be a nonnegative integer, $x=(x_1,\dots,x_n)$
and \smash{$\dot{e}_r(x)=e_r\bigl(x^{\pm}\bigr)$}.
Then
\begin{subequations}
\begin{gather}
P_{(k^n)}^{\mathrm{B}_n}(x;t,0)
=\sum_{y\in\mathbb{Z}^k}
\det_{1\leq i,j\leq k}\bigl((-1)^{y_i}
t^{\frac{1}{2}Ky_i^2-(j-\frac{1}{2})y_i}
 (\dot{e}_{n-i+j-Ky_i}(x)+\dot{e}_{n+i+j-Ky_i-1}(x) )\bigr),\!\!\!\!\label{Eq_PBn}
\end{gather}
where $K:=2k+1$, and
\begin{gather}
P_{(k^n)}^{\mathrm{B}_n}\bigl(x;t,-t^{1/2}\bigr)
=\sum_{y\in\mathbb{Z}^k}
\det_{1\leq i,j\leq k}\bigl(
t^{\frac{1}{2}Ky_i^2-(j-\frac{1}{2})y_i}
 (\dot{e}_{n-i+j-Ky_i}(x)+\dot{e}_{n+i+j-Ky_i-1}(x) )\bigr),\label{Eq_PBn2}
\\
P_{(k^n)}^{\mathrm{B}_n}(x;t,1)
=\frac{1}{2}\!\sum_{y\in\mathbb{Z}^k}\!
\det_{1\leq i,j\leq k}\bigl((-1)^{y_i}
t^{\frac{1}{2}Ky_i^2{-}(j{-}1)y_i}
 (\dot{e}_{n{-}i{+}j{-}Ky_i}(x){+}\dot{e}_{n{+}i{+}j{-}Ky_i{-}2}(x) )\bigr),\!\!\!\!\label{Eq_PBn3}
\end{gather}
\end{subequations}
where $K:=2k$.
\end{Conjecture}

For $\la\subseteq (k^n)$ and $x$ as above, let
\smash{$P_{\la}^{\mathrm{C}_n}(x;t,s)$} denote the $\mathrm{C}_n$ Hall--Littlewood
polynomial indexed by $\la$.
Then \smash{$P_{\la}^{\mathrm{C}_n}(x;0,0)=\symp_{2n,\la}(x)$}.
This time we have two conjectural analogues of~\eqref{Eq_JTC},
corresponding to the affine root systems
\smash{$\mathrm{C}_k^{(1)}$} and \smash{$\mathrm{A}_{2k-1}^{(2)}$}, respectively.

\begin{Conjecture}\label{Con_JTC}
Let $k$ be a nonnegative integer, $x=(x_1,\dots,x_n)$
and $\dot{e}_r(x)=e_r\bigl(x^{\pm}\bigr)$.
Then
\begin{subequations}
\begin{equation}\label{Eq_PCn}
P_{(k^n)}^{\mathrm{C}_n}(x;t,0)
=\sum_{y\in\mathbb{Z}^k} \det_{1\leq i,j\leq k}\bigl(
t^{\frac{1}{2}Ky_i^2-j y_i} (
\dot{e}_{n-i+j-Ky_i}(x)-\dot{e}_{n+i+j-Ky_i}(x) )\bigr),
\end{equation}
where $K=2k+2$, and
\begin{equation}\label{Eq_PCn2}
P_{(k^n)}^{\mathrm{C}_n}(x;t,t)
=\sum_{\substack{y\in\mathbb{Z}^k\\[1pt] \abs{y}\textup{ even}}}
\det_{1\leq i,j\leq k}\bigl(
t^{\frac{1}{2}Ky_i^2-j y_i} (
\dot{e}_{n-i+j-Ky_i}(x)-\dot{e}_{n+i+j-Ky_i}(x) )\bigr),
\end{equation}
\end{subequations}
where $K=2k$.
\end{Conjecture}

Finally, let \smash{$P_{\la}^{\mathrm{BC}_n}(x;t,s_1,s_2)$} denote the
$\mathrm{BC}_n$ Hall--Littlewood polynomial indexed by $\la$.
Then~\smash{$P_{\la}^{\mathrm{BC}_n}(x;0,0,0)=\symp_{2n,\la}(x)$} and
\smash{$P_{\la}^{\mathrm{BC}_n}(x;0,1,-1)=\oeven_{2n,\la}(x)$}.
Conjecturally, we have \linebreak one~more generalisation of
\eqref{Eq_JTC} corresponding to the affine root system
\smash{$\mathrm{A}_{2k}^{(2)}$}.

\begin{Conjecture}\label{Con_JTBC}
Let $k$ be a nonnegative integer, $x=(x_1,\dots,x_n)$
and \smash{$\dot{e}_r(x)=e_r\bigl(x^{\pm}\bigr)$}.
Then
\begin{gather}
P_{(k^n)}^{\mathrm{BC}_n}\bigl(x;t,-t^{1/2},0\bigr)
 =\sum_{y\in\mathbb{Z}^k}
\det_{1\leq i,j\leq k}\bigl((-1)^{y_i}
t^{\frac{1}{2}Ky_i^2{-}jy_i}
 (\dot{e}_{n{-}i{+}j{-}Ky_i}(x){-}\dot{e}_{n{+}i{+}j{-}Ky_i}(x) )\bigr),\!\!\!\label{Eq_PBCn}
\end{gather}
where $K=2k+1$.
\end{Conjecture}

Once again this gives \eqref{Eq_JTC} for $t=0$.

In Section~\ref{Sec_JTBC}, Conjectures~\ref{Con_JTB}--\ref{Con_JTBC}
will be proved for $k=1$ using bounded Littlewood identities.
In the same section we also present proofs of \eqref{Eq_PBn} and
\eqref{Eq_PBCn} for $t=1$ based on cylindric Schur functions.

Conjectures~\ref{Con_JTB}--\ref{Con_JTBC} have some remarkable
consequences in that (with the exception of \eqref{Eq_PBn3}) they
imply new Rogers--Ramanujan type identities which contain not just
the standard parameter $q$ but also a Hall--Littlewood parameter $t$,
while still admitting a product form.
These $q,t$\nobreakdash-Rogers--Ramanujan identities generalise many of the classical
Rogers--Ramanujan identities, such as the two original identities of Rogers
and Ramanujan \cite{Rogers94,Rogers17,RR19}, (a subset of) the
Andrews--Gordon identities \cite{Andrews74}, their even modulus
analogue of Bressoud \cite{Bressoud80,Bressoud80b}, and
the \smash{$\mathrm{C}_n^{(1)}$},~\smash{$\mathrm{A}_{2n}^{(2)}$} and~\smash{$\mathrm{D}_{n+1}^{(2)}$}
GOW identities of Griffin, Ono and the author~\cite{GOW16}.

Let \smash{$(a;q)_{\infty}=\prod_{i\geq 0}\bigl(1-aq^i\bigr)$} be a $q$-shifted factorial,
\smash{$\theta(a;p)=(a;p)_{\infty}(p/a;p)_{\infty}$} a modified Jacobi theta
function and $\theta(a_1,\dots,a_k;p)=\prod_{i=1}^k \theta(a_i;p)$.
For $\la$ a partition, $\la$ is said to be even if all of its parts
are even.

\begin{Theorem}[$q,t$-Rogers--Ramanujan-type identity
for \smash{$\mathrm{C}_k^{(1)}$}]\label{Thm_qtRR}
For $k$ a positive integer, let $p=tq^{2k+2}$.
Then
\begin{equation}\label{Eq_qt-one}
\sum_{\substack{\la\textup{ even} \\[1pt] \la_1\leq 2k}}
q^{(\sigma+1)\abs{\la}/2} P_{\la}\bigl(1,q,q^2,\dots;t\bigr)
=\frac{(p;p)_{\infty}^k}
{(q;q)_{\infty}^k}
\prod_{i=1}^k \theta\bigl(q^{(2-\sigma)i};p\bigr)
\prod_{1\leq i<j\leq k}
\theta\bigl(q^{j-i},q^{i+j};p\bigr),
\end{equation}
where $\sigma\in\{0,1\}$.
\end{Theorem}

The labelling of the theorem by the affine root system
\smash{$\mathrm{C}_k^{(1)}$} reflects the fact that $k+1$ is the dual Coxeter
number of \smash{$\mathrm{C}_k^{(1)}$} and the product on the right
of \eqref{Eq_qt-one} is a specialisation of the~Weyl--Kac
denominator for \smash{$\mathrm{C}_k^{(1)}$}.
For $t=0$, the theorem simplifies to
\begin{align}
\sum_{\substack{\la\text{ even} \\[1pt] \la_1\leq 2k}}
q^{(\sigma+1)\abs{\la}/2} s_{\la}\bigl(1,q,q^2,\dots\bigr)
&=\prod_{i\geq 1} \frac{1-q^{(\sigma+1)(k+i)}}{1-q^{(\sigma+1)i}}
\prod_{1\leq i<j} \frac{1-q^{2k+i+j}}{1-q^{i+j}} \notag \\
&=\lim_{n\to\infty} \phi_{\sigma}
\bigl(\eup^{-k\omega_n}\ch L(k\omega_n)\bigr),
\label{Eq_C}
\end{align}
where $L(k\omega_n)$ is the irreducible
$\mathrm{Sp}(2n,\mathbb{C})$-module of highest weight $k\omega_n$
(with $\omega_1,\dots,\omega_n$ the fundamental weights of
$\mathrm{Sp}(2n,\mathbb{C})$), and $\phi_{\sigma}$
is the specialisation
\[
\phi_{\sigma}\colon\ \mathbb{Z}[\eup^{-\alpha_1},\dots,\eup^{-\alpha_n}]
\to\mathbb{Z}[q],\qquad
\eup^{-\alpha_i}\mapsto \begin{cases}
q & \text{for $1\leq i\leq n-1$}, \\
q^{\sigma+1} & \text{for $i=n$}.\end{cases}
\]

For the affine root system \smash{$\mathrm{A}_{2k}^{(2)}$}, with (dual)
Coxeter number $2k+1$, we have two identities.
Given a partition $\la$, let $\la^{\textup{o}}$ be the partition consisting
of the odd parts of $\la$, so that $l(\la^{\textup{o}})$ is the number of
parts of $\la$ that are odd.

\begin{Theorem}[$q,t$-Rogers--Ramanujan-type identities for
\smash{$\mathrm{A}_{2k}^{(2)}$}]\label{Thm_qtRR2}
For $k$ a positive integer, let $p=tq^{2k+1}$.
Then
\begin{subequations}
\begin{align}
&\sum_{\substack{\la \\[1pt] \la_1\leq 2k}}
q^{(\sigma+1)\abs{\la}/2} P_{\la}\bigl(1,q,q^2,\dots;t\bigr) \label{Eq_A2k2a}\\
&\qquad=\frac{(p;p)_{\infty}^k}
{(\sigma+1)(q;q)_{\infty}^k}
\prod_{i=1}^k \theta\bigl({-}q^{i-(\sigma+1)/2};p\bigr)
\theta\bigl(pq^{2i-\sigma-1};p^2\bigr)
\prod_{1\leq i<j\leq k}\theta\bigl(q^{j-i},q^{i+j-\sigma-1};p\bigr)
\notag
\end{align}
and
\begin{align}
&\sum_{\substack{\la \\[1pt] \la_1\leq 2k}}
q^{(\sigma+1)\abs{\la}/2}t^{l(\la^{\textup{o}})/2}
P_{\la}\bigl(1,q,q^2,\dots;t\bigr) \label{Eq_A2k2b}\\
&\qquad=\frac{(p;p)_{\infty}^k}{(q;q)_{\infty}^k}
\prod_{i=1}^k \theta\bigl({-}p^{1/2}q^{i-\sigma/2};p\bigr)
\theta\bigl(q^{2i-\sigma};p^2\bigr)
\prod_{1\leq i<j\leq k}\theta\bigl(q^{j-i},q^{i+j-\sigma};p\bigr),
\notag
\end{align}
\end{subequations}
where $\sigma\in\{0,1\}$.
\end{Theorem}

We remark that by
\[
\prod_{i=1}^k\theta\bigl(-q^{i-1};p\bigr)\theta\bigl(pq^{2i-2};p^2\bigr)
\prod_{1\leq i<j\leq k}\theta\bigl(q^{i+j-2};p\bigr)=
2\prod_{1\leq i<j\leq k}\theta\bigl(q^{i+j-1};p\bigr)
\]
the identity \eqref{Eq_A2k2a} for $\sigma=1$ can be written as
\[
\sum_{\substack{\la \\[1pt] \la_1\leq 2k}}
q^{\abs{\la}} P_{\la}\bigl(1,q,q^2,\dots;t\bigr)
=\frac{(p;p)_{\infty}^k}{(q;q)_{\infty}^k}
\prod_{1\leq i<j\leq k}\theta\bigl(q^{j-i},q^{i+j-1};p\bigr).
\]
For $t=0$, \eqref{Eq_A2k2a} yields
\begin{align}
\sum_{\substack{\la \\[1pt] \la_1\leq 2k}}
q^{(\sigma+1)\abs{\la}/2} s_{\la}\bigl(1,q,q^2,\dots\bigr)
&=\prod_{i\geq 1} \frac{1-q^{(\sigma+1)(k+i-1/2)}}
{1-q^{(\sigma+1)(i-1/2)}}
\prod_{1\leq i<j} \frac{1-q^{2k+i+j}}{1-q^{i+j}} \notag \\
&=\lim_{n\to\infty} \hat{\phi}_{\sigma}
\bigl(\eup^{-k\omega_n}\ch L(k\omega_n)\bigr),
\label{Eq_B}
\end{align}
where $L(k\omega_n)$ is the irreducible
$\mathrm{SO}(2n+1,\mathbb{C})$-module of highest weight $k\omega_n$
and $\hat{\phi}_{\sigma}$ is the specialisation
\[
\hat{\phi}_{\sigma}\colon\ \mathbb{Z}\bigl[\eup^{-\alpha_1},\dots,\eup^{-\alpha_n}\bigr]
\to\mathbb{Z}\bigl[q^{(\sigma+1)/2}\bigr],\qquad
\eup^{-\alpha_i}\mapsto \begin{cases}
q & \text{for $1\leq i\leq n-1$}, \\
q^{(\sigma+1)/2} & \text{for $i=n$}.\end{cases}
\]
The $t=0$ case of \eqref{Eq_A2k2b} once again yields \eqref{Eq_C}.

We conclude with two theorems for \smash{$\mathrm{D}_{k+1}^{(2)}$}
and \smash{$\mathrm{A}_{2k-1}^{(2)}$}, respectively.
For both these root systems, the dual Coxeter number is $2k$.
For $\la$ a partition, we denoted by $m_i(\la)$ the multiplicity
of parts of size $i$.

\begin{Theorem}[$q,t$-Rogers--Ramanujan identity for
\smash{$\mathrm{D}_{k+1}^{(2)}$}]\label{Thm_qtRR3}
For $k$ a positive integer, let $p=tq^{2k}$. Then
\begin{align*}
&\sum_{\substack{\la \\[1pt] \la_1\leq 2k}} q^{(\sigma+1)\abs{\la}/2}
\Biggl( \prod_{i=1}^{2k-1}\bigl({-}t^{1/2};t^{1/2}\bigr)_{m_i(\la)}\Biggr)
P_{\la}\bigl(1,q,q^2,\dots;t\bigr) \\
&\qquad =\frac{\bigl(p^{1/2};p^{1/2}\bigr)_{\infty}
(p;p)_{\infty}^{k-1}}{(\sigma+1)(q;q)_{\infty}^k}
\prod_{i=1}^k \theta\bigl({-}q^{i-(\sigma+1)/2};p^{1/2}\bigr)
\prod_{1\leq i<j\leq k}\theta\bigl(q^{j-i},q^{i+j-\sigma-1};p\bigr),
\end{align*}
where $\sigma\in\{0,1\}$.
\end{Theorem}

\begin{Theorem}[$q,t$-Rogers--Ramanujan identity for
\smash{$\mathrm{A}_{2k-1}^{(2)}$}]\label{Thm_qtRR4}
For $k$ a positive integer, let $p=tq^{2k}$.
Then
\begin{align*}
&\sideset{}{'}\sum_{\substack{\la \\[1pt] \la_1\leq 2k}}
q^{(\sigma+1)\abs{\la}/2}t^{l(\la^{\textup{o}})/2}
\Biggl( \prod_{i=1}^{2k-1}\bigl(t;t^2\bigr)_{\ceil{m_i(\la)/2}}\Biggr)
P_{\la}\bigl(1,q,q^2,\dots;t\bigr) \\
&\qquad =\frac{(p;p)_{\infty}^2(p;p)_{\infty}^{k-1}}
{(q;q)_{\infty}^k}
\prod_{i=1}^k \theta\bigl(q^{2i-\sigma};p^2\bigr)
\prod_{1\leq i<j\leq k}\theta\bigl(q^{j-i},q^{i+j-\sigma};p\bigr),
\end{align*}
where $\sigma\in\{0,1\}$ and the prime in the sum over $\la$ denotes
the restriction that the odd parts of~$\la$ have even multiplicity.
\end{Theorem}

In the $t=0$ limit, these theorems simplify to \eqref{Eq_B} and \eqref{Eq_C},
respectively.

It will be shown in Section~\ref{Sec_qtRR} how
Conjectures~\ref{Con_JTC} and \ref{Con_JTBC} imply the above results.
Of course, since the latter are conjectures, they do not imply
proofs of Theorems~\ref{Thm_qtRR}--\ref{Thm_qtRR4}, but once discovered
it is not hard to provide a proof that does not rely on the conjectures,
see Section~\ref{Sec_Ismail}.
In Section~\ref{Sec_qtRR}, it will also be shown how the Andrews--Gordon
and GOW identities follow from~\eqref{Eq_qt-one} by a novel manifestation
of level-rank duality.
For now, we remark that by
\[
P_{(2^r)}\bigl(1,q,q^2,\dots;q\bigr)=\frac{q^{r^2-r}}{(q;q)_r}
\]
the identity \eqref{Eq_qt-one} for $t=q$ and $k=1$ \big(and thus $p=q^5$\big)
simplifies to the Rogers--Ramanujan identities
\cite{Rogers94,Rogers17,RR19}
\[
\sum_{r=0}^{\infty}\frac{q^{r^2}}{(q;q)_r}
=\frac{\bigl(q^2,q^3,q^5;q^5\bigr)_{\infty}}{(q;q)_{\infty}}
\quad\text{and}\quad
\sum_{r=0}^{\infty}\frac{q^{r^2+r}}{(q;q)_r}
=\frac{\bigl(q,q^4,q^5;q^5\bigr)_{\infty}}{(q;q)_{\infty}},
\]
where $(a;q)_r=\prod_{i=0}^{r-1}\bigl(1-aq^i\bigr)$ and
$(a_1,\dots,a_k;q)_{\infty}=\prod_{i=1}^k (a_i;q)_{\infty}$.


\subsection{Outline}

The remainder of this paper is organised as follows.
In Section~\ref{Sec_SF}, we prepare the necessary symmetric function
material needed for the rest of the paper.
Most of what is covered in this section is well known, and apart from some
easy to prove, marginally new results for Hall--Littlewood polynomials,
the only new results of some depth are the two determinant
identities for elementary symmetric functions stated in
Lemma~\ref{Lem_determinant}.
These identities will be important in our applications of
Conjectures~\ref{Con_JTB}--\ref{Con_JTBC}.
In Section~\ref{Sec_GLn}, we prove Theorem~\ref{Thm_GLn} using results
from crystal base theory due to Schilling and Shimozono.
We also discuss a number of known special cases of the theorem
as well as higher-level generalisations of Theorem~\ref{Thm_GLn}
related to level-restricted Kostka polynomials.
In Section~\ref{Sec_JTBC}, we use bounded Littlewood identities
and Bailey pairs to prove the Conjectures~\ref{Con_JTB}--\ref{Con_JTBC}
for $k=1$.
We further apply known results for cylindric Schur functions to prove
two of the six identities from the conjectures for $t=1$.
In Section~\ref{Sec_qtRR}, we show that the
$q,t$-Rogers--Ramanujan identities stated in
Theorems~\ref{Thm_qtRR}--\ref{Thm_qtRR4} arise by specialising
Conjectures~\ref{Con_JTB}--\ref{Con_JTBC}.
It is then shown that by specialising $t$ to arbitrary
positive powers of $q$ the $q,t$-Rogers--Ramanujan
identities lead to many new and old Rogers--Ramanujan identities
for specialised characters of standard modules of affine Lie algebras.
We use these Rogers--Ramanujan identities for standard modules together with
Ismail's analytic argument to provide a proof of the
$q,t$-Rogers--Ramanujan identities that is not reliant on the
validity of Conjectures~\ref{Con_JTB}--\ref{Con_JTBC}.
In Section~\ref{Sec_open}, we conclude the main part of the paper
with a list of open problems related to our work.
Finally, in the appendix a proof is given to a new family of
Rogers--Ramanujan-type identities related to one of our applications
of Theorem~\ref{Thm_qtRR4} discussed in Section~\ref{Sec_standard}.
This family of identities was discovered by Matthew Russell after reading
an earlier version of this paper.

\section{Symmetric functions}\label{Sec_SF}


\subsection{(Cylindric) Schur functions}

Let $\la=(\la_1,\la_2,\dots)$ be a partition, i.e., a weakly decreasing
sequence of nonnegative integers such that $\abs{\la}:=\la_1+\la_2+\cdots$
is finite.
If $\abs{\la}=m$, we say that $\la$ is a partition of $m$, denoted as~${\la\vdash m}$.
The parts of $\la$ are the positive $\la_i$ in the sequence,
and the number of parts of $\la$ is its length, denoted by $l(\la)$.
For all $i\geq 1$, $m_i=m_i(\la)$ is the multiplicity of parts of size
$i$ in $\la$.
We alternatively write $\la$ in multiplicity notation as
$\la=(r^{m_r},\dots,1^{m_1})$, where $r$ is the largest part of~$\la$.
Here we typically omit $i^{m_i}$ if $m_i=0$.
For example, the partition $(4,3,3,3,1,1)$ is also written as
$\bigl(4^1,3^3,1^2\bigr)$.
A partition of the form $(k^r)$ is referred to as a rectangular
partition or a~partition of rectangular shape.
Given partitions $\la=(\la_1,\la_2,\dots)$ and $\mu=(\mu_1,\mu_2,\dots)$,
we write~${\mu\subseteq\la}$ if $\la_i-\mu_i\geq 0$ for all $i\geq 1$,
and say that $\mu$ is contained in $\la$.
Given $\mu\subseteq\la$, we write~${\mu\prec\la}$ if the interlacing conditions
$\la_1\geq\mu_1\geq\la_2\geq\mu_2\geq\cdots$ hold.
The set of partitions of length at most $n$ will be denoted by
$\Par_n$ and the set of partitions contained in the rectangle~$(k^n)$
by~$\Par_{n,k}$.
Given a partition $\la$, its Young diagram corresponds to the diagram
obtained by drawing $l(\la)$ rows of left-aligned boxes or squares, such that
the $i$-th row contains~$\la_i$ squares.
The Young diagram of the partition $\nu=(4,3,3,3,1,1)$ is given by
the left-most of the following two diagrams:
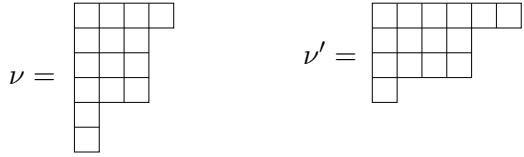
\begin{center}
\begin{tikzpicture}[scale=0.33,line width=0.3pt]
\draw (-1.7,3) node {$\nu=$};
\draw (0,0)--(1,0)--(1,2)--(3,2)--(3,5)--(4,5)--(4,6)--(0,6)--cycle;
\draw (0,1)--(1,1);\draw (0,2)--(1,2);\draw (0,3)--(3,3);
\draw (0,4)--(3,4);\draw (0,5)--(3,5);
\draw (1,2)--(1,6);\draw (2,2)--(2,6);\draw (3,5)--(3,6);
\draw (10.3,4) node {$\nu'=$};
\draw (12,2)--(13,2)--(13,3)--(16,3)--(16,5)--(18,5)--(18,6)--(12,6)--cycle;
\draw (12,3)--(13,3);\draw (12,4)--(16,4);\draw(12,5)--(16,5);
\draw (13,3)--(13,6);\draw (14,3)--(14,6);\draw (15,3)--(15,6);
\draw (16,5)--(16,6);\draw (17,5)--(17,6);
\end{tikzpicture}
\end{center}
We typically do not distinguish between a partition and its Young diagram.
The conjugate of the partition $\la$, denoted $\la'$, is obtained by reflecting
the diagram of $\la$ in the main diagonal, so that the conjugate of the
partition $\nu$ in our example corresponds to the above Young diagram on
the right.
In our subsequent discussion of cylindric tableaux, it will be convenient to
assume that Young diagrams are made up of unit squares, so that the length
of the $i$-th row of $\la$ is $\la_i$ and the length of the $j$-th column is
$\la'_j$.
If $\mu\subseteq\la$, we denote by $\la/\mu$ the skew diagram obtained
by removing those boxes of $\la$ that are also contained in $\mu$.
If $\mu\prec\la$, then the skew diagram $\la/\mu$ has at most one box in each
column and is known as a horizontal strip.

For partitions $\mu\subseteq\la$ and (weak) composition
$\alpha=(\alpha_1,\dots,\alpha_n)$ such that $\abs{\alpha}=\abs{\la/\mu}$,
let $\SSYT_n(\la/\mu,\alpha)$ denote the set of semistandard Young
tableaux of shape $\la/\mu$ and filling $\alpha$.
That is, a tableaux $T$ in $\SSYT_n(\la/\mu,\alpha)$ correspond to a
filling of the Young diagram of $\la/\mu$ such that $\alpha_i$ squares
are filled with the number $i$ and such that rows are weakly increasing
from left to right and columns are strictly increasing from top to bottom.
Also set $\SSYT_n(\la/\mu)=\bigcup_{\alpha} \SSYT_n(\la/\mu,\alpha)$.
An example of tableau in $\SSYT_7((6,6,6,4,1)/(3,2),(1,2,3,3,4,3,2))$
is given by
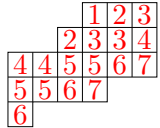
\begin{center}
\begin{tikzpicture}[scale=0.33,line width=0.3pt]
\draw (0,0)--(1,0)--(1,1)--(4,1)--(4,2)--(6,2)--(6,5)--(3,5)--
(3,4)--(2,4)--(2,3)--(0,3)--cycle;
\draw (0,1)--(1,1); \draw (0,2)--(4,2);
\draw (2,3)--(6,3); \draw (3,4)--(6,4);
\draw (1,1)--(1,3); \draw (2,1)--(2,3);
\draw (3,1)--(3,4); \draw (4,2)--(4,5);
\draw (5,2)--(5,5);
\begin{scope}[red]
\draw (3.5,4.5) node {$1$};
\draw (4.5,4.5) node {$2$};
\draw (5.5,4.5) node {$3$};
\draw (2.5,3.5) node {$2$};
\draw (3.5,3.5) node {$3$};
\draw (4.5,3.5) node {$3$};
\draw (5.5,3.5) node {$4$};
\draw (0.5,2.5) node {$4$};
\draw (1.5,2.5) node {$4$};
\draw (2.5,2.5) node {$5$};
\draw (3.5,2.5) node {$5$};
\draw (4.5,2.5) node {$6$};
\draw (5.5,2.5) node {$7$};
\draw (0.5,1.5) node {$5$};
\draw (1.5,1.5) node {$5$};
\draw (2.5,1.5) node {$6$};
\draw (3.5,1.5) node {$7$};
\draw (0.5,0.5) node {$6$};
\end{scope}
\end{tikzpicture}
\end{center}
Recall from \eqref{Eq_sT} that, for $\la\in\Par_n$, the Schur
function $s_{\la}(x_1,\dots,x_n)$ may be expressed as a~sum over
tableaux in $\SSYT_n(\la)$. We now wish to extend this to the case of
cylindric Schur functions~\cite{HKKO25,Lee19,McNamara06,Postnikov05}.
For positive integers $k$, $n$ and nonnegative integer $\ell$, let
\smash{$\Par_{n,k}^{\ell}$} be the set of partitions $\la$ contained in $(k^n)$
such that $\la'_1-\la'_k\leq\ell$.
For \smash{$\la\in\Par_{n,k}^{\ell}$} and $T\in\SSYT_n(\la,\alpha)$,
let $\overline{T}$ be the copy of $T$ obtained by translating $T$ by
$k$ units to the right and $\ell$ units up.
Since~${\la'_1-\la'_k\leq\ell}$, the union of $T$ and $\overline{T}$
is a column strict skew tableaux of shape~${\nu/\bigl(k^{\ell}\bigr)}$, where
\[
\nu_i=\begin{cases}
\la_i+k & \text{for $1\leq i\leq\ell$}, \\
\la_i+\la_{i-\ell} & \text{for $\ell<i\leq n+\ell$}
\end{cases}
\]
and $\la_i=0$ for $i>n$.
If, beyond column-strictness, $T\cup\overline{T}$ is semistandard,
$T$ is said to be cylindric.
The set of all cylindric semistandard Young tableaux of shape $\la$
and filling $\alpha$ will be denoted by \smash{$\CSSYT_{n;k,\ell}(\la,\alpha)$},
and once again we set
\smash{$\CSSYT_{n;k,\ell}(\la)=\bigcup_{\alpha}\CSSYT_{n;k,\ell}(\la,\alpha)$}.
Alternatively, \smash{$T\in\CSSYT_{n;k,\ell}(\la,\alpha)$} may be viewed as a
tableau of shape \smash{$\la\in\Par_{n,k}^{\ell}$} wrapped around a cylinder
of width $k$ with vertical offset $\ell$ as in the following diagram for
$\la=(6,6,6,4,1)$, $k=6$, $\ell=3$, $n=7$ and $\alpha=(4,4,3,3,4,3,2)$:
\usetikzlibrary{math}
\tikzmath{\a=6;}
\tikzmath{\b=3;}
\begin{center}
\begin{tikzpicture}[scale=0.33,line width=0.3pt]
\draw[thick] (0,0)--(1,0)--(1,1)--(4,1)--(4,2)--(6,2)--(6,5)--(0,5)--cycle;
\draw (0,1)--(1,1); \draw (0,2)--(4,2);
\draw (0,3)--(6,3); \draw (0,4)--(6,4);
\draw (1,1)--(1,5); \draw (2,1)--(2,5);
\draw (3,1)--(3,5); \draw (4,2)--(4,5);
\draw (5,2)--(5,5);
\begin{scope}[color=red]
\draw (0.5,4.5) node {$1$};
\draw (1.5,4.5) node {$1$};
\draw (2.5,4.5) node {$1$};
\draw (3.5,4.5) node {$1$};
\draw (4.5,4.5) node {$2$};
\draw (5.5,4.5) node {$3$};
\draw (0.5,3.5) node {$2$};
\draw (1.5,3.5) node {$2$};
\draw (2.5,3.5) node {$2$};
\draw (3.5,3.5) node {$3$};
\draw (4.5,3.5) node {$3$};
\draw (5.5,3.5) node {$4$};
\draw (0.5,2.5) node {$4$};
\draw (1.5,2.5) node {$4$};
\draw (2.5,2.5) node {$5$};
\draw (3.5,2.5) node {$5$};
\draw (4.5,2.5) node {$6$};
\draw (5.5,2.5) node {$7$};
\draw (0.5,1.5) node {$5$};
\draw (1.5,1.5) node {$5$};
\draw (2.5,1.5) node {$6$};
\draw (3.5,1.5) node {$7$};
\draw (0.5,0.5) node {$6$};
\end{scope}
\draw[thick,->] (-0.5,4)--(-0.5,5);
\draw[thick,->] (-0.5,3)--(-0.5,2);
\draw (-0.5,3.5) node {$\ell$};
\draw[thick,->] (3.5,5.5)--(6,5.5);
\draw[thick,->] (2.5,5.5)--(0,5.5);
\draw (3,5.5) node {$k$};
\draw[thin] (6,3)--(7,3)--(7,4)--(10,4)--(10,5)--(12,5)--(12,8)--(6,8)--cycle;
\draw (6,4)--(7,4); \draw (6,5)--(10,5);
\draw (6,6)--(12,6); \draw (6,7)--(12,7);
\draw (7,4)--(7,8); \draw (8,4)--(8,8);
\draw (9,4)--(9,8); \draw (10,5)--(10,8);
\draw (11,5)--(11,8);
\begin{scope}[color=gray]
\draw (\a+0.5,\b+4.5) node {$1$};
\draw (\a+1.5,\b+4.5) node {$1$};
\draw (\a+2.5,\b+4.5) node {$1$};
\draw (\a+3.5,\b+4.5) node {$1$};
\draw (\a+4.5,\b+4.5) node {$2$};
\draw (\a+5.5,\b+4.5) node {$3$};
\draw (\a+0.5,\b+3.5) node {$2$};
\draw (\a+1.5,\b+3.5) node {$2$};
\draw (\a+2.5,\b+3.5) node {$2$};
\draw (\a+3.5,\b+3.5) node {$3$};
\draw (\a+4.5,\b+3.5) node {$3$};
\draw (\a+5.5,\b+3.5) node {$4$};
\draw (\a+0.5,\b+2.5) node {$4$};
\draw (\a+1.5,\b+2.5) node {$4$};
\draw (\a+2.5,\b+2.5) node {$5$};
\draw (\a+3.5,\b+2.5) node {$5$};
\draw (\a+4.5,\b+2.5) node {$6$};
\draw (\a+5.5,\b+2.5) node {$7$};
\draw (\a+0.5,\b+1.5) node {$5$};
\draw (\a+1.5,\b+1.5) node {$5$};
\draw (\a+2.5,\b+1.5) node {$6$};
\draw (\a+3.5,\b+1.5) node {$7$};
\draw (\a+0.5,\b+0.5) node {$6$};
\end{scope}
\draw[thin] (-6,-3)--(-5,-3)--(-5,-2)--(-2,-2)--(-2,-1)--(0,-1)--(0,2)--(-6,2)--cycle;
\draw (-6,-2)--(-5,-2); \draw (-6,-1)--(-2,-1);
\draw (-6,0)--(0,0); \draw (-6,1)--(0,1);
\draw (-5,-3)--(-5,2); \draw (-4,-2)--(-4,2);
\draw (-3,-2)--(-3,2); \draw (-2,-1)--(-2,2);
\draw (-1,-1)--(-1,2);
\begin{scope}[color=gray]
\draw (-\a+0.5,-\b+4.5) node {$1$};
\draw (-\a+1.5,-\b+4.5) node {$1$};
\draw (-\a+2.5,-\b+4.5) node {$1$};
\draw (-\a+3.5,-\b+4.5) node {$1$};
\draw (-\a+4.5,-\b+4.5) node {$2$};
\draw (-\a+5.5,-\b+4.5) node {$3$};
\draw (-\a+0.5,-\b+3.5) node {$2$};
\draw (-\a+1.5,-\b+3.5) node {$2$};
\draw (-\a+2.5,-\b+3.5) node {$2$};
\draw (-\a+3.5,-\b+3.5) node {$3$};
\draw (-\a+4.5,-\b+3.5) node {$3$};
\draw (-\a+5.5,-\b+3.5) node {$4$};
\draw (-\a+0.5,-\b+2.5) node {$4$};
\draw (-\a+1.5,-\b+2.5) node {$4$};
\draw (-\a+2.5,-\b+2.5) node {$5$};
\draw (-\a+3.5,-\b+2.5) node {$5$};
\draw (-\a+4.5,-\b+2.5) node {$6$};
\draw (-\a+5.5,-\b+2.5) node {$7$};
\draw (-\a+0.5,-\b+1.5) node {$5$};
\draw (-\a+1.5,-\b+1.5) node {$5$};
\draw (-\a+2.5,-\b+1.5) node {$6$};
\draw (-\a+3.5,-\b+1.5) node {$7$};
\draw (-\a+0.5,-\b+0.5) node {$6$};
\draw[thin] (12,6)--(14,6);
\draw[thin] (12,8)--(12,9);
\draw[thin,dotted] (13,7.5) node {$\iddots$};
\draw[thin] (-8,-1)--(-6,-1);
\draw[thin] (-6,-3)--(-6,-4);
\draw[thin,dotted] (-7,-2) node {$\iddots$};
\end{scope}
\end{tikzpicture}
\end{center}
Combining \eqref{Eq_sT} and \eqref{Eq_JTe}, we have
\begin{equation}\label{Eq_xTe}
\sum_{T\in\SSYT_n(\la)} x^T
=\det_{1\leq i,j\leq k} \bigl(e_{\la'_i-i+j}(x)\bigr)
\end{equation}
for $\la\in\Par_{n,k}$ and $x=(x_1,\dots,x_n)$.
Huh et al.~\cite[Proposition 2.6]{HKKO25} generalised this as follows.

\begin{Proposition}\label{Prop_cSchur}
Let $\la\in\Par_{n,k}^{\ell}$ and $x=(x_1,\dots,x_n)$.
Then
\begin{equation}\label{Eq_cSchur}
\sum_{T\in\CSSYT_{n;k,\ell}(\la)} x^T =
\sum_{\substack{y\in\mathbb{Z}^k \\[1pt] \abs{y}=0}}
\det_{1\leq i,j\leq k} \bigl(e_{\la'_i-i+j-(k+\ell)y_i}(x)\bigr).
\end{equation}
\end{Proposition}

The symmetric function in this proposition is known as a cylindric
Schur function and is denoted as \smash{$s_{\mu[k,\ell]'}(x)$} for
$\mu=\la'$ in \cite{HKKO25}.
For $\ell\geq l(\la)$, $\CSSYT_{n;k,\ell}(\la)=\SSYT_n(\la)$
and \smash{$\det\bigl(e_{\la'_i-i+j-(k+\ell)y_i}(x)\bigr)=0$} unless $y_1=\dots=y_n=0$.
Hence \eqref{Eq_cSchur} includes \eqref{Eq_xTe} as a special case.

For later reference, we consider \eqref{Eq_cSchur} for $\ell=0$ and
$\ell=1$.
Since \smash{$\Par_{n,k}^0=\{(k^r) \mid 0\leq r\leq n\}$} and since the cyclic
condition for $\ell=0$ implies that all boxes in row $j$ of
$T\in\CSSYT_{n;k,0}((k^r))$ have the same filling, say $i_j$,
column-strictness implies
\begin{equation}\label{Eq_nul}
\sum_{T\in\CSSYT_{n;k,0}((k^r))}x^T=\sum_{1\leq i_1<i_2<\cdots<i_r\leq n}
x_{i_1}^k \cdots x_{i_r}^k = e_r\bigl(x_1^k,\dots,x_n^k\bigr)=e_r\bigl(x^k\bigr).
\end{equation}
For $\ell=1$, the set of admissible partitions is the set of
near-rectangles
\[
\Par_{n,k}^1=\{(k^r,a) \mid r\geq 0,\, 0\leq a<k\}.
\]
Since $\ell=1$, the set $\CSSYT_{n;k,1}((k^r,a),\alpha)$ for
$\abs{\alpha}=kr+a$ contains at most one tableau.
Column-strictness imposes the obvious condition $\alpha_i\leq k$
for all $1\leq i\leq n$.
Moreover, if a given row of $T\in\SSYT_n((k^r,a),\alpha)$ ends with
a box with filling $i$ and the next row begins with a box
with filling $j<i$, then $T\not\in\CSSYT_{n;k,1}((k^r,a),\alpha)$.
Hence, assuming $\abs{\alpha}=kr+a$ and $\alpha_i\leq k$,
the only $T\in\SSYT_n((k^r,a),\alpha)$ contained in
$\CSSYT_{n;k,1}((k^r,a),\alpha)$ is given by the tableau
whose word obtained by reading rows from left
to right starting with the top row
and ending with the bottom row is given by
$1^{\alpha_1},2^{\alpha_2},\dots,n^{\alpha_n}$ as in
\begin{center}
\begin{tikzpicture}[scale=0.33,line width=0.3pt]
\draw (0,0)--(3,0)--(3,1)--(6,1)--(6,5)--(0,5)--cycle;
\foreach \y in {1,2,3,4} \draw (0,\y)--(6,\y);
\foreach \x in {1,2} \draw (\x,0)--(\x,5);
\foreach \x in {3,4,5} \draw (\x,1)--(\x,5);
\begin{scope}[color=red]
\draw (0.5,4.5) node {$1$};
\draw (1.5,4.5) node {$1$};
\draw (2.5,4.5) node {$1$};
\draw (3.5,4.5) node {$1$};
\draw (4.5,4.5) node {$2$};
\draw (5.5,4.5) node {$2$};
\draw (0.5,3.5) node {$2$};
\draw (1.5,3.5) node {$2$};
\draw (2.5,3.5) node {$3$};
\draw (3.5,3.5) node {$3$};
\draw (4.5,3.5) node {$3$};
\draw (5.5,3.5) node {$4$};
\draw (0.5,2.5) node {$4$};
\draw (1.5,2.5) node {$4$};
\draw (2.5,2.5) node {$5$};
\draw (3.5,2.5) node {$5$};
\draw (4.5,2.5) node {$5$};
\draw (5.5,2.5) node {$5$};
\draw (0.5,1.5) node {$6$};
\draw (1.5,1.5) node {$6$};
\draw (2.5,1.5) node {$6$};
\draw (3.5,1.5) node {$6$};
\draw (4.5,1.5) node {$6$};
\draw (5.5,1.5) node {$7$};
\draw (0.5,0.5) node {$7$};
\draw (1.5,0.5) node {$7$};
\draw (2.5,0.5) node {$8$};
\end{scope}
\draw[thick,->] (6.5,3.5)--(6.5,5);
\draw[thick,->] (6.5,2.5)--(6.5,1);
\draw (6.5,3) node {$r$};
\draw[thick,->] (3.5,5.5)--(6,5.5);
\draw[thick,->] (2.5,5.5)--(0,5.5);
\draw (3,5.5) node {$k$};
\draw[thick,->] (2,-0.5)--(3,-0.5);
\draw[thick,->] (1,-0.5)--(0,-0.5);
\draw (1.5,-0.5) node {$a$};
\end{tikzpicture}
\end{center}
Hence for $\la\in\Par_{n,k}^1$,
\begin{equation}
\sum_{T\in\CSSYT_{n;k,1}(\la)} x^T =
\sum_{\substack{\mu\in\Par_{n,k}\\[1pt]
\abs{\mu}=\abs{\la}}} \sum_{\alpha\in S_n\cdot\mu} x^{\alpha}
=\sum_{\substack{\mu\in\Par_{n,k}\\[1pt]\abs{\mu}=\abs{\la}}}
m_{\mu}(x), \label{Eq_een}
\end{equation}
where $S_n\cdot\mu$ denotes the $S_n$-orbit of $\mu=(\mu_1,\dots,\mu_n)$
and $m_{\mu}(x)$ is the monomial symmetric function indexed by $\mu$.


\subsection[Hall--Littlewood polynomials for GL(n,C)]{Hall--Littlewood polynomials for $\boldsymbol{\mathrm{GL}(n,\mathbb{C})}$}
\label{Sec_HL-Gln}

Let $x=(x_1,\dots,x_n)$ and $S_n$ the symmetric group of degree $n$
with natural action on polynomials in $x$.
Then the ordinary or $\mathrm{GL}(n,\mathbb{C})$ Hall--Littlewood
polynomial $P_{\la}(x;t)$ is defined as \cite{Macdonald95}
\begin{equation*}
P_{\la}(x;t)=\frac{1}{v_{\la}(t)}
\sum_{w\in S_n} w\biggl(x_1^{\la_1}\cdots x_n^{\la_n}
\prod_{1\leq i<j\leq n}\frac{x_i-tx_j}{x_i-x_j}\biggr),
\end{equation*}
where
\[
v_{\la}(t)=
\frac{(t;t)_{n-l(\la)}}{(1-t)^{n-l(\la)}}\prod_{i\geq 1}
\frac{(t;t)_{m_i(\la)}}{(1-t)^{m_i(\la)}}.
\]
The Hall--Littlewood polynomials \smash{$\{P_{\la}(x;t)\}_{l(\la)\leq n}$}
form a basis of the ring of symmetric functions with coefficients
in $\mathbb{Z}[t]$.
In analogy with the Schur functions $s_{\la}(x)=P_{\la}(x;0)$, we set
$P_{\la}(x;t)=0$ if $\la$ is a partition such that $l(\la)>n$.
By the stability property,
\[
P_{\la}(x_1,\dots,x_n,0;t)=\begin{cases}
P_{\la}(x_1,\dots,x_n;t) & \text{if $l(\la)\leq n$}, \\
0 & \text{otherwise},
\end{cases}
\]
the Hall--Littlewood polynomials may be extended to symmetric functions
in countably many variables, $P_{\la}(x_1,x_2,\dots;t)$.
Then the number of variables in a result such as Theorem~\ref{Thm_GLn}
becomes irrelevant, and alternatively it may be stated as a result in
the algebra of symmetric functions:
\begin{equation}\label{Eq_Thm_GLn2}
P_{(k^r)}(t)=
\sum_{\substack{y\in\mathbb{Z}^k\\[1pt] \abs{y}=0}}
\det_{1\leq i,j\leq k}
\bigl(t^{k\binom{y_i}{2}+iy_i} e_{r-i+j-ky_i} \bigr),
\end{equation}
with no reference to a particular choice of alphabet.

Special cases of $P_{\la}(x;t)$ other than the Schur function are
\begin{equation}\label{Eq_Pm}
P_{\la}(x;1)=m_{\la}(x)
\end{equation}
and $P_{(1^r)}(x)=e_r(x)$.

For $\alpha$ a composition, let
\smash{$e_{\alpha}(x)=\prod_{i\geq 1} e_{\alpha_i}(x)$}.
Then the transition coefficients $\mathcal{R}_{\la,\alpha}(t)$
between elementary symmetric functions and Hall--Littlewood polynomials
are defined by
\begin{equation}\label{Eq_transition}
e_{\alpha}(x)=\sum_{\la}\mathcal{R}_{\la,\alpha}(t) P_{\la}(x;t).
\end{equation}

For integers $k$, $n$, let
\[
\qbin{n}{k}_q=\begin{cases}
\displaystyle \frac{(q;q)_n}{(q;q)_k(q;q)_{n-k}} & \text{if $0\leq k\leq n$}, \\[3mm]
0 & \text{otherwise},
\end{cases}
\]
be a Gaussian polynomial or $q$-binomial coefficient.

\begin{Theorem}[{Kirillov~\cite[Theorem 3.4]{Kirillov00}}]
For $\alpha=(\alpha_1,\dots,\alpha_k)$,
\begin{equation}\label{Eq_Kir}
\mathcal{R}_{\la,\alpha}(t)=
\sum_{\substack{0=\nu^{(0)}\subseteq\nu^{(1)}
\subseteq\cdots\subseteq \nu^{(k)}=\la' \\[1pt]
\abs{\nu^{(i)}/\nu^{(i-1)}}=\alpha_i}}
\prod_{i=2}^k \prod_{j=1}^{i-1}
\qBin{\nu^{(i)}_j-\nu^{(i)}_{j+1}}{\nu^{(i)}_j-\nu^{(i-1)}_j}_t.
\end{equation}
\end{Theorem}
The summand on the right vanishes unless the interlacing condition
$\nu^{(a-1)}\prec\nu^{(a)}$ holds for all $1\leq a\leq k$, so that
one may also think of the sum over the $\nu^{(a)}$ as a sum over
semistandard Young tableaux of shape $\la'$ and filling $\alpha$.
This implies that $\mathcal{R}_{\la,\alpha}(t)=0$ if $\la_1>k$,
where $\alpha=(\alpha_1,\dots,\alpha_k)$, \label{page_R} but also
that if $\la=(k^r,\mu)$ where $\mu$ is a partition such that
$\mu_1<k$, then
\begin{equation}\label{Eq_reduce}
\mathcal{R}_{(k^r,\mu),(\alpha_1,\dots,\alpha_k)}(t)=
\mathcal{R}_{\mu,(\alpha_1-r,\dots,\alpha_k-r)}(t).
\end{equation}
The expression \eqref{Eq_Kir} for $k=2$ simplifies to
\begin{equation}\label{Eq_R}
\mathcal{R}_{(2^r,1^s),(\alpha_1,\alpha_2)}(t)=
\delta_{2r+s,\alpha_1+\alpha_2}\qbin{s}{\alpha_1-r}_t,
\end{equation}
where $\delta_{i,j}$ is a Kronecker delta.
This yields the following simple result, which will be used in the
proof of Conjectures~\ref{Con_JTB}--\ref{Con_JTBC} for $k=1$.
Recall that $\dot{e}_r(x)=e_r\bigl(x^{\pm}\bigr)$.

\begin{Lemma}\label{Lem_eP}
Let $x=(x_1,\dots,x_n)$ and $k$ an integer such that $-n\leq k\leq n$.
Then
\begin{equation*}
(x_1\cdots x_n) \dot{e}_{n-k}(x)=
\sum_{\substack{r,s\geq 0\\[1pt] s-k \text{ even}}}
\qbin{s}{\frac{1}{2}(s-k)}_t P_{(2^r,1^s)}(x;t).
\end{equation*}
\end{Lemma}

Since the summand on the right vanishes unless $\abs{k}\leq s\leq n-r$,
only finitely many terms contribute to the sum, as it should.

\begin{proof}
By
\begin{equation}\label{Eq_branching}
e_r(x_1,\dots,x_n,y_1,\dots,y_m)=
\sum_{i=0}^r e_i(x_1,\dots,x_n) e_{r-i}(y_1,\dots,y_m)
\end{equation}
and
\[
e_r\bigl(x^{-1}\bigr)=(x_1\cdots x_n)^{-1}\, e_{n-r}(x),
\]
we have
\[
(x_1\cdots x_n) \dot{e}_{n-k}(x)=\sum_{i\geq 0} e_i(x) e_{i+k}(x),
\]
where the summand is nonzero for $\max\{0,-k\}\leq i\leq \min\{n,n-k\}$
only.
Hence
\begin{align*}
(x_1\cdots x_n) \dot{e}_{n-k}(x)
&=\sum_{i\geq 0} e_{(i,i+k)}(x) \\
&=\sum_{i,r,s\geq 0}
\delta_{2r+s,k+2i}\qbin{s}{i-r}_t P_{(2^r,1^s)}(x;t) \\
&=\sum_{\substack{r,s\geq 0\\[1pt] s-k \text{ even}}}
\qbin{s}{\frac{1}{2}(s-k)}_t P_{(2^r,1^s)}(x;t),
\end{align*}
where the second equality follows from \eqref{Eq_R}.
\end{proof}

For $m$ a nonnegative integer, let $H_m(x;t)$ be the
Rogers--Szeg\H{o} polynomial \cite[p.~49]{Andrews76}
\begin{equation}\label{Eq_RS}
H_m(x;t)=\sum_{i=0}^m x^i \qbin{m}{i}_t
\end{equation}
and, for $k$ a positive integer, let \smash{$h_{\la}^{(k)}(a,b;t)$} be the
Rogers--Szeg\H{o} polynomial indexed by the partition $\la$
\cite{BW15,RW21}:
\[
h_{\la}^{(k)}(a,b;t)=
\prod_{\substack{i=1 \\[1pt] i~\text{odd}}}^{k-1}
(-a)^{m_i(\la)}H_{m_i(\la)}(b/a;t)
\prod_{\substack{i=1 \\[1pt] i~\text{even}}}^{k-1}
H_{m_i(\la)}(ab;t).
\]
This polynomial is symmetric in $a$ and $b$, and
\[
h_{(1^r)}^{(k)}(a,b;t)=(-a)^r H_r(b/a;t)=
(-1)^r \sum_{i=0}^r a^{r-i}b^i \qbin{r}{i}_t
\]
provided that $k\geq 2$.
For notational convenience, we also define
\[
\tilde{h}_{\la}^{(k)}(a,b;t):=
a^{\la'_{2k+1}} h_{\la}^{(2k+1)}(-a,-b;t)
= a^{l(\la^{\odd})}
\prod_{i=1}^{2k} H_{m_{2i-1}(\la)}(b/a;t) H_{m_{2i}(\la)}(ab;t),
\]
which is \emph{not} symmetric in $a$ and $b$.

\begin{Proposition}\label{Prop_33bounded}
For $x=(x_1,\dots,x_n)$ and $k$ a nonnegative integer,
\begin{equation}\label{Eq_33bounded}
\Biggl(\prod_{i=1}^n (1+ax_i)\Biggr)
\sum_{\substack{\mu \\[1pt] \mu_1\leq 2k}} b^{l(\mu^{\odd})}
P_{\mu}(x;t)
=\sum_{\substack{\la \\[1pt] \la_1\leq 2k+1}}
\tilde{h}_{\la}^{(k)}(a,b;t) P_{\la}(x;t).
\end{equation}
\end{Proposition}

\begin{proof}
The proof is essentially a repeat of the proof of
\cite[equation~(3.3)]{W06a}, which is an identity
equivalent to the $k\to\infty$ limit of \eqref{Eq_33bounded}.

We begin by recalling the $e$-Pieri rule for Hall--Littlewood polynomials
\cite[p.~215]{Macdonald95}
\[
P_{\mu}(x;t) e_r(x) = \sum_{\la\vdash\abs{\mu}+r} P_{\la}(x;t)
\prod_{i\geq 1} \qbin{\la'_i-\la'_{i+1}}{\la'_i-\mu'_i}_t.
\]
The summand on the right vanishes unless $\mu'\prec\la'$.
Multiplying both sides by $a^r$ and then summing over $r$, we obtain
\[
\Biggl(\prod_{i=1}^n (1+ax_i)\Biggr) P_{\mu}(x;t) =
\sum_{\la} a^{\abs{\la/\mu}} P_{\la}(x;t)
\prod_{i\geq 1} \qbin{\la'_i-\la'_{i+1}}{\la'_i-\mu'_i}_t.
\]
Further multiplying this by \smash{$b^{l(\mu^{\odd})}$} and summing
over $\mu$ such that $\mu_1\leq 2k$, we thus find
\[
\Biggl(\prod_{i=1}^n (1+ax_i)\Biggr)
\sum_{\substack{\mu \\[1pt] \mu_1\leq 2k}} b^{l(\mu^{\odd})}
P_{\mu}(x;t)
=\sum_{\substack{\la \\[1pt] \la_1\leq 2k+1}} P_{\la}(x;t)
\sum_{\substack{\mu \\[1pt] \mu_1\leq 2k}}
a^{\abs{\la/\mu}} b^{l(\mu^{\odd})}
\prod_{i=1}^{2k} \qbin{\la'_i-\la'_{i+1}}{\la'_i-\mu'_i}_t.
\]
Here we have used the fact that the summand vanishes unless
$\mu'\prec\la'$, which, given that $\mu_1\leq 2k$, implies
that $\la_1\leq 2k+1$.
In the sum over $\mu$ we now make the substitutions
$\mu'_{2i-1}\mapsto j_i+\la'_{2i}$ and
$\mu'_{2i}\mapsto \la'_{2i}-l_i$ for $1\leq i\leq k$.
Then \smash{$\abs{\la/\mu}\mapsto l(\la^{\odd})+\sum_{i=1}^k (l_i-j_i)$} and
\smash{$l(\mu^{\odd})\mapsto \sum_{i=1}^k (j_i+l_i)$}, so that
\begin{align*}
&\Biggl(\prod_{i=1}^n (1+ax_i)\Biggr)
\sum_{\substack{\mu \\[1pt] \mu_1\leq 2k}} b^{l(\mu^{\odd})}
P_{\mu}(x;t) \\
&\qquad=\sum_{\substack{\la \\[1pt] \la_1\leq 2k+1}}
a^{l(\la^{\odd})} P_{\la}(x;t)
\sum_{\substack{j_1,\dots,j_k\geq 0\\[1pt] l_1,\dots,l_k\geq 0}}\,
\prod_{i=1}^k
(b/a)^{j_i} \qbin{\la'_{2i-1}-\la'_{2i}}{j_i}_t (ab)^{l_i}
\qbin{\la'_{2i}-\la'_{2i+1}}{l_i}_t \\
&\qquad=\sum_{\substack{\la \\[1pt] \la_1\leq 2k+1}}
a^{l(\la^{\odd})} P_{\la}(x;t)
\prod_{i=1}^k H_{m_{2i-1}(\la)}(b/a;t)H_{m_{2i}(\la)}(ab;t) \\
&\qquad=\sum_{\substack{\la \\[1pt] \la_1\leq 2k+1}}
\tilde{h}_{\la}^{(k)}(a,b;t) P_{\la}(x;t),
\end{align*}
as claimed.
\end{proof}

According to \cite[equation~(2.3.7)]{RW21},
\[
\tilde{h}_{\la}^{(k)}(a,0;t)=a^{l(\la^{\odd})}
\qquad\text{and}\qquad
\tilde{h}_{\la}^{(k)}\bigl(1,t^{1/2};t\bigr)
=\prod_{i=1}^{2k} \bigl(-t^{1/2};t^{1/2}\bigr)_{m_i(\la)}.
\]
This implies the following corollary of Proposition~\ref{Prop_33bounded}.

\begin{Corollary}
For $x=(x_1,\dots,x_n)$ and $k$ a nonnegative integer,
\begin{subequations}
\begin{align}\label{Eq_oddk1}
&\Biggl(\prod_{i=1}^n (1+ax_i)\Biggr)
\sum_{\substack{\mu \text{ even} \\[1pt] \mu_1\leq 2k}}
P_{\mu}(x;t)=\sum_{\substack{\la \\[1pt] \la_1\leq 2k+1}}
a^{l(\la^{\odd})} P_{\la}(x;t), \\
&\Biggl(\prod_{i=1}^n (1+x_i)\Biggr)
\sum_{\substack{\mu \\[1pt] \mu_1\leq 2k}} t^{l(\mu^{\odd})/2}P_{\mu}(x;t)
=\sum_{\substack{\la \\[1pt] \la_1\leq 2k+1}}
\Biggl(\prod_{i=1}^{2k}\bigl(-t^{1/2};t^{1/2}\bigr)_{m_i(\la)}\Biggr)
P_{\la}(x;t).
\label{Eq_oddk2}
\end{align}
\end{subequations}
\end{Corollary}

To conclude our discussion of the Hall--Littlewood polynomials
for $\mathrm{GL}(n,\mathbb{C})$, we define its modified analogue.
Consider the expansion of the Hall--Littlewood symmetric functions
in terms of the Newton power sums
\[
P_{\la}(x_1,x_2,\dots;t)=\sum_{\mu} c_{\la\mu}(t) p_{\mu}(x_1,x_2,\dots),
\]
where $p_{\mu}=\prod_{i\geq 1} p_{\mu_i}$ and
$p_0=1$, $p_r(x_1,x_2,\dots)=x_1^r+x_2^r+\cdots$ for $r\geq 1$.
Then the modified Hall--Littlewood polynomial
$P'_{\la}(x_1,x_2,\dots,x_n;t)$ is defined as
\[
P'_{\la}(x_1,\dots,x_n;t)=\sum_{\mu} c_{\la\mu}(t)
p_{\mu} (X_1(t),\dots,X_n(t) ),
\]
where $X_i(t)=\bigl(x_i,x_it,x_it^2,\dots\bigr)$.
Alternatively, using plethystic notation,
$P'_{\la}(x;t)=P'_{\la}(x/(1-t);t)$.
From the definition, it follows immediately that
\begin{equation}\label{Eq_PPp-spec}
P'_{\la}\bigl(1,q,\dots,q^{n-1};q^n\bigr)=
P_{\la}\bigl(1,q,q^2,\dots;q^n\bigr).
\end{equation}

\subsection[The B\_n, C\_n and BC\_n]{The $\boldsymbol{\mathrm{B}_n}$, $\boldsymbol{\mathrm{C}_n}$ and $\boldsymbol{\mathrm{BC}_n}$
Hall--Littlewood polynomials}\label{Sec_HL-BCn}

Let $W=S_n\ltimes (\mathbb{Z}/2\mathbb{Z})^n$ be the hyperoctahedral group
(or group of signed permutations) with standard action on Laurent
polynomials in $x=(x_1,\dots,x_n)$.
Then, for $\la$ a partition of length at most $n$, the
$\mathrm{BC}_n$ Hall--Littlewood polynomial
\smash{$P^{\mathrm{BC}_n}_{\la}(x;t,s_1,s_2)$} is defined as
\cite{Macdonald00,Venkateswaran15}
\begin{align}
&P^{(\mathrm{BC}_n)}_{\la}(x;t,s_1,s_2)
=\frac{1}{(s_1s_2;t)_{n-l(\la)}v_{\la}(t)} \notag \\
& \qquad{}\times
\sum_{w\in W} w\Biggl( \prod_{i=1}^n x_i^{-\la_i}
\frac{(1-s_1x_i)(1-s_2x_i)}{1-x_i^2} \prod_{1\leq i<j\leq n}
\frac{(tx_i-x_j)(1-tx_ix_j)}{(x_i-x_j)(1-x_ix_j)}\Biggr).
\label{Eq_PBCn-def}
\end{align}
The $\mathrm{BC}_n$ Hall--Littlewood polynomials
are polynomials in
\smash{$\mathbb{Z}[t,s_1,s_2]\bigl[x_1^{\pm},\dots,x_n^{\pm}\bigr]^W$},
normalised such that
\smash{$P^{\mathrm{BC}_n}_0(x;t,s_1,s_2)=1$}.
The top-degree homogeneous component of the polynomials
\smash{$P^{(\mathrm{BC}_n)}_{\la}(x;t,s_1,s_2)$}
is given by the ordinary Hall--Littlewood polynomial $P_{\la}(x;t)$.

The $\mathrm{B}_n$ and $\mathrm{C}_n$ Hall--Littlewood polynomials
arise as special cases of the $\mathrm{BC}_n$ Hall--Littlewood
polynomials:
\begin{subequations}
\begin{align}
\label{Eq_BCn_to_Bn}
&P^{(\mathrm{B}_n)}_{\la}(x;t,s)=
P^{(\mathrm{BC}_n)}_{\la}(x;t,s,-1) ,\\
&P^{(\mathrm{C}_n)}_{\la}(x;t,s)=
P^{(\mathrm{BC}_n)}_{\la}\bigl(x;t,s^{1/2},-s^{1/2}\bigr).
\label{Eq_BCn_to_Cn}
\end{align}
\end{subequations}
(For $\mathrm{B}_n$, there is once again an analogue for half-partitions $\la$
which is not needed in this paper.)
The odd-orthogonal, even-orthogonal and symplectic Schur functions
correspond to the special cases
\begin{gather*}
\so_{2n+1,\la}(x)=P^{(\mathrm{B}_n)}_{\la}(x;0,0), \qquad
\oeven_{2n,\la}(x)=P^{(\mathrm{B}_n)}_{\la}(x;0,1),
\qquad
\symp_{2n,\la}(x)=P^{(\mathrm{C}_n)}_{\la}(x;0,0).
\end{gather*}
For $\la=(k^n)$, the summand in \eqref{Eq_PBCn-def}
has $S_n$ symmetry so that
the sum over $W$ simplifies to a~sum over $(\mathbb{Z}/2\mathbb{Z})^n$,
leading to \cite[Lemma 2.5]{RW21}
\[
P^{(\mathrm{BC}_n)}_{(k^n)}(x;t,s_1,s_2)
=\sum_{\varepsilon_1,\dots,\varepsilon_n\in \{\pm 1\}}
\Phi\bigl(x_1^{\varepsilon_1},\dots,x_n^{\varepsilon_n};t,s_1,s_2\bigr)
\prod_{i=1}^n x_i^{-\varepsilon_i k},
\]
where
\[
\Phi(x_1,\dots,x_n;t,s_1,s_2):=
\prod_{i=1}^n \frac{(1-s_1x_i)(1-s_2x_i)}{1-x_i^2}
\prod_{1\leq i<j\leq n}\frac{1-tx_ix_j}{1-x_ix_j}.
\]
In particular,
\begin{equation}\label{Eq_PBCnt1}
P^{(\mathrm{BC}_n)}_{(k^n)}(x;1,s_1,s_2)
=\prod_{i=1}^n
\frac{(1-s_1x_i)(1-s_2x_i)x_i^{-k}
-(s_1-x_i)(s_2-x_i)x_i^k}{1-x_i^2}.
\end{equation}

In the proofs of special cases of our conjectures as well as in the
derivation of $q,t$-Ro\-gers--Ramanujan identities, we require a number
of bounded Littlewood identities, expressing the~$\mathrm{B}_n$,~$\mathrm{C}_n$
and~$\mathrm{BC}_n$ Hall--Littlewood polynomials indexed
by rectangular partitions of length $n$ in terms of ordinary
Hall--Littlewood polynomials.

\begin{Theorem}[bounded Littlewood identities for $\mathrm{B}_n$]
For $k$ a nonnegative integer,
\begin{subequations}
\begin{gather}\label{Eq_Macdonald}
\sum_{\substack{\la \\[1pt] \la_1\leq 2k}}
P_{\la}(x;t)=(x_1\cdots x_n)^k\,P_{(k^n)}^{\mathrm{B}_n}(x;t,0), \\
\label{Eq_RW1}
\sum_{\substack{\la \\[1pt] \la_1\leq 2k}}
\Biggl(\prod_{i=1}^{2k-1}\bigl({-}t^{1/2};t^{1/2}\bigr)_{m_i(\la)}\Biggr)
P_{\la}(x;t)=(x_1\cdots x_n)^k\,
P_{(k^n)}^{\mathrm{B}_n}\bigl(x;t,-t^{1/2}\bigr), \\
\label{Eq_RW2}
\sideset{}{''}\sum_{\substack{\la \\[1pt] \la_1\leq 2k}}
\Biggl(\prod_{i=1}^{2k-1}\bigl(t;t^2\bigr)_{m_i(\la)/2}\Biggr)
P_{\la}(x;t)=(x_1\cdots x_n)^k\,P_{(k^n)}^{\mathrm{B}_n}(x;t,1),
\end{gather}
\end{subequations}
where the double prime denotes the restriction that all parts of $\la$ that
are strictly less than $2k$ have even multiplicity.
\end{Theorem}

These three results (which also hold for \smash{$k+\frac{1}{2}$} a nonnegative
integer) are the special cases $t_2=0$, $t_2=-t^{1/2}$ and $t_2=1$
of the integral-$k$ case of \cite[Theorem~4.8]{RW21}.
The identity~\eqref{Eq_Macdonald} is equivalent to a result of
Macdonald \cite[pp.~232--233]{Macdonald95}, although his right-hand
side is not identified as a $\mathrm{B}_n$ Hall--Littlewood polynomial,
a fact that is crucial in our use of the result.
Since $P_{\la}^{\mathrm{B}_n}(x;t,1)=P_{\la}^{\mathrm{C}_n}(x;t,1)$,
the identity \eqref{Eq_RW2} may also be regarded as a result for
$\mathrm{C}_n$.
This, however, does not apply to its extension to half-integer
values of $k$.

\begin{Theorem}[bounded Littlewood identities for $\mathrm{C}_n$]
For $k$ a nonnegative integer,
\begin{subequations}
\begin{gather}\label{Eq_Stembridge}
\sum_{\substack{\la\textup{ even} \\[1pt] \la_1\leq 2k}}
P_{\la}(x;t)=(x_1\cdots x_n)^k \, P_{(k^n)}^{\mathrm{C}_n}(x;t,0), \\
\label{Eq_Stembridge2}
\sideset{}{'}\sum_{\substack{\la \\[1pt] \la_1\leq 2k}}
t^{l(\la^{\textup{o}})/2}
\Biggl(\prod_{i=1}^{2k-1}\bigl(t;t^2\bigr)_{\ceil{m_i(\la)/2}}\Biggr)
P_{\la}(x;t)=(x_1\cdots x_n)^k \, P_{(k^n)}^{\mathrm{C}_n}(x;t,t),
\end{gather}
\end{subequations}
where the prime denotes the restriction that the odd parts of
$\la$ have even multiplicity.
\end{Theorem}

In the above form, the identity \eqref{Eq_Stembridge} was first stated in
\cite[equation~(4.1.7)]{RW21} and is the special ${t_2=t_3=0}$ instance
of \cite[Theorem~4.7]{RW21}.
An equivalent identity, which makes no reference to the symplectic
Hall--Littlewood polynomials, was previously found by Stembridge
\cite[Theorem~1.2]{Stembridge90}.

Finally, we need one $\mathrm{BC}_n$ bounded Littlewood identity,
given by the $(t_2,t_3)=\bigl(-t^{1/2},0\bigr)$ special case of
\cite[equation~(4.1.15)]{RW21}.

\begin{Theorem}
For $k$ a nonnegative integer,
\begin{equation}\label{Eq_RW3}
\sum_{\substack{\la \\[1pt] \la_1\leq 2k}}
t^{l(\la^{\textup{o}})/2}
P_{\la}(x;t)=(x_1\cdots x_n)^k\,
P_{(k^n)}^{\mathrm{BC}_n}\bigl(x;t,-t^{1/2},0\bigr).
\end{equation}
\end{Theorem}


\subsection{Determinant identities for elementary symmetric functions}

The aim of this section is to show that, with the exception of
\eqref{Eq_PBn3}, each determinant in
Conjectures~\ref{Con_JTB}--\ref{Con_JTBC} admits an alternative
expression in terms of the elementary symmetric functions on the
alphabet \smash{$\bigl(x^{\pm},1\bigr)=\bigl(x_1,x_1^{-1},\dots,x_n,x_n^{-1},1\bigr)$}.

Let \smash{$\ddot{e}_r(x):=e_r\bigl(x^{\pm},1\bigr)$}.
By \eqref{Eq_branching},
\begin{equation}\label{Eq_ee}
\ddot{e}_r(x)=\dot{e}_r(x)+\dot{e}_{r-1}(x).
\end{equation}

\begin{Lemma}\label{Lem_determinant}
For $K$ an arbitrary integer,
\begin{subequations}
\begin{align}
\begin{aligned}[b]
&\sum_{y\in\mathbb{Z}^k}
\det_{1\leq i,j\leq k}\bigl(u^{y_i}
t^{\frac{1}{2}Ky_i^2-(j-\frac{1}{2})y_i}
 (\dot{e}_{n-i+j-Ky_i}(x)+\dot{e}_{n+i+j-Ky_i-1}(x) )\bigr)
  \\
&\qquad{}=\frac{1}{2}\sum_{y\in\mathbb{Z}^k}
\det_{1\leq i,j\leq k}\bigl(u^{y_i}
t^{\frac{1}{2}Ky_i^2-(j-\frac{1}{2})y_i}
 (\ddot{e}_{n-i+j-Ky_i+1}(x)+\ddot{e}_{n+i+j-Ky_i-1}(x) )\bigr)
 \end{aligned}
\label{Eq_doteddote}
\end{align}
and
\begin{align}
&\sum_{y\in\mathbb{Z}^k}
\det_{1\leq i,j\leq k}\bigl(u^{y_i}
t^{\frac{1}{2}Ky_i^2-j y_i} (
\dot{e}_{n-i+j-Ky_i}(x)-\dot{e}_{n+i+j-Ky_i}(x) )\bigr) \notag \\
&\qquad{}=
\sum_{y\in\mathbb{Z}^k}
\det_{1\leq i,j\leq k}\bigl(u^{y_i}
t^{\frac{1}{2}Ky_i^2-j y_i} (
\ddot{e}_{n-i+j-Ky_i+1}(x)-\ddot{e}_{n+i+j-Ky_i}(x) )\bigr).
\label{Eq_doteddote2}
\end{align}
\end{subequations}
\end{Lemma}

\begin{proof}
First we consider \eqref{Eq_doteddote}.
Fix $y\in\mathbb{Z}^k$ such that $y_1\geq y_2\geq\cdots\geq y_k$.
Then \eqref{Eq_doteddote} is a~consequence~of
\begin{align}
&\sum_{S_k\cdot y} \det_{1\leq i,j\leq k}\bigl( t^{-jy_i}
 (\dot{e}_{n-i+j-Ky_i}(x)+\dot{e}_{n+i+j-Ky_i-1}(x) )\bigr) \notag \\
&\qquad=\frac{1}{2}\sum_{S_k\cdot y} \det_{1\leq i,j\leq k}\bigl(
t^{-jy_i}
 (\ddot{e}_{n-i+j-Ky_i+1}(x)+\ddot{e}_{n+i+j-Ky_i-1}(x) )\bigr),
\label{Eq_stronger}
\end{align}
where $S_k\cdot y$ denotes the $S_k$ orbit of $y$.
Summing over all $k!$ permutations of $y$ instead of $S_k\cdot y$
amounts to multiplying \eqref{Eq_stronger} by the size of
the stabilizer of $y$, and in the following it will
be more convenient to consider \eqref{Eq_stronger} with $y$ replaced by
$w(y)=(y_{w_1},\dots,y_{w_k})$ and the sum over $S_k\cdot y$ replaced
by a sum over $w\in S_k$.
Then, after applying \eqref{Eq_ee} to the right-hand side, the identity
to be proved becomes the special case --
replace $(x_i,z_i)\mapsto (t^{-y_i},Ky_i)$ followed
by~${E_r\mapsto \dot{e}_{n-r}(x)}$~-- of the formal identity
\begin{align*}
&\sum_{w\in S_k} \det_{1\leq i,j\leq k}\bigl(x_{w_i}^j
\bigl(E_{z_{w_i}+i-j}+E_{z_{w_i}-i-j+1}\bigr)\bigr) \\
&\qquad{}=\frac{1}{2}\sum_{w\in S_k}
\det_{1\leq i,j\leq k}\bigl(
x_{w_i}^j \bigl(E_{z_{w_i}+i-j-1}+E_{z_{w_i}+i-j}+
E_{z_{w_i}-i-j+1}+E_{z_{w_i}-i-j+2}\bigr)\bigr).
\end{align*}
Dispensing with the determinants, this is the same as
\[
\sum_{\sigma,w\in S_k}\sgn(\sigma)\prod_{i=1}^k
x_{w_i}^{\sigma_i} F_{i,i}(\sigma,w)
=\frac{1}{2}\sum_{\sigma,w\in S_k}\sgn(\sigma)\prod_{i=1}^k
x_{w_i}^{\sigma_i} (F_{i,i-1}(\sigma,w)+F_{i,i}(\sigma,w) ),
\]
where
\[
F_{i,j}(\sigma,w):=E_{z_{w_i}+j-\sigma_i}+E_{z_{w_i}-j-\sigma_i+1}.
\]
Since $F_{i,-j}(\sigma,w)=F_{i,j+1}(\sigma,w)$ and thus
$F_{1,0}(\sigma,w)=F_{1,1}(\sigma,w)$, this may be simplified to
\begin{equation}\label{Eq_setsum}
\sum_{\varnothing\subset I\subseteq\{2,\dots,k\}}
\sum_{\sigma,w\in S_k}\sgn(\sigma)
\prod_{i=1}^k x_{w_i}^{\sigma_i}
\prod_{i\in I} F_{i,i-1}(\sigma,w)
\prod_{i\in\{1,\dots,k\}\setminus I}
\prod_{i\in I} F_{i,i}(\sigma,w)=0.
\end{equation}
Since \smash{$\prod_{i=1}^k x_{w_i}^{\sigma_i}=\prod_{i=1}^k x_i^{(\sigma w^{-1})_i}$},
we replace $w\mapsto w\sigma$ to obtain
\[
\sum_{\varnothing\subset I\subseteq\{2,\dots,k\}}
\sum_{\sigma,w\in S_k}\sgn(\sigma)
\prod_{i=1}^k x_i^{(w^{-1})_i}
\prod_{i\in I} F_{i,i-1}(\sigma,w\sigma)
\prod_{i\in\{1,\dots,k\}\setminus I}
\prod_{i\in I} F_{i,i}(\sigma,w\sigma)=0.
\]
In the following, we will show the stronger vanishing result
\begin{equation}\label{Eq_zero}
f_{I;w}:=\sum_{\sigma\in S_k}\sgn(\sigma)
\prod_{i\in I} F_{i,i-1}(\sigma,w\sigma)
\prod_{i\in\{1,\dots,k\}\setminus I}
\prod_{i\in I} F_{i,i}(\sigma,w\sigma)=0
\end{equation}
for fixed $\varnothing\subset I\subseteq\{2,\dots,k\}$ and $w\in S_k$.
For $j\in\{1,\dots,k-1\}$, let $s_j\in S_k$ denote the $j$-th adjacent
transposition, i.e., the 2-cycle $(j,j+1)$.
Then, for any $\tau\in S_k$,
\[
(\tau s_j)_i=\begin{cases}
\tau_{j+1} & \text{if $i=j$}, \\
\tau_j & \text{if $i=j+1$}, \\
\tau_i & \text{otherwise}.
\end{cases}
\]
Hence
\begin{align*}
F_{i,i-1}(\sigma s_j,w\sigma s_j)&=\begin{cases}
F_{i-1,i-1}(\sigma,w\sigma) & \text{if $i=j+1$}, \\
F_{i,i-1}(\sigma,w\sigma) & \text{if $i\neq j,j+1$}
\end{cases}
\end{align*}
and
\begin{align*}
F_{i,i}(\sigma s_j,w\sigma s_j)&=\begin{cases}
F_{i+1,i}(\sigma,w\sigma) & \text{if $i=j$}, \\
F_{i,i}(\sigma,w\sigma) & \text{if $i\neq j,j+1$}.
\end{cases}
\end{align*}
We now fix the index $j$ of $s_j$ as $j=\min\{I\}-1$.
Then $j+1\in I$ and $j\notin I$, and thus by replacing~$\sigma$
by~$\sigma s_j$ in \eqref{Eq_zero} and using that
$\sgn(\sigma s_j)=-\sgn(\sigma)$,
\begin{align*}
f_{I;w}&=-
\sum_{\sigma\in S_k}\biggl(\sgn(\sigma)
F_{j+1,j}(\sigma s_j,w\sigma s_j) F_{j,j}(\sigma s_j,w\sigma s_j) \\
& \qquad\qquad\quad\times
\prod_{\substack{i\in I\\[1pt] i\neq j+1}} F_{i,i-1}(\sigma s_j,w\sigma s_j)
\prod_{\substack{i\in\{1,\dots,k\}\setminus I \\[1pt] i\neq j}}
\prod_{i\in I} F_{i,i}(\sigma s_j,w\sigma s_j)\biggr) \\
&=-\sum_{\sigma\in S_k}\biggl(\sgn(\sigma)
F_{j,j}(\sigma,w\sigma) F_{j+1,j}(\sigma,w\sigma)
\prod_{\substack{i\in I\\[1pt] i\neq j+1}}\! F_{i,i-1}(\sigma,w\sigma)
\!\prod_{\substack{i\in\{1,\dots,k\}\setminus I \\[1pt] i\neq j}}
\prod_{i\in I} F_{i,i}(\sigma,w\sigma)\biggr) \\
&=-\sum_{\sigma\in S_k}\biggl(\sgn(\sigma)
\prod_{i\in I} F_{i,i-1}(\sigma,w\sigma)
\prod_{i\in\{1,\dots,k\}\setminus I}
\prod_{i\in I} F_{i,i}(\sigma,w\sigma)\biggr) \\
&=-f_{I;w}.
\end{align*}
Hence $f_{I;w}=0$.

Next we consider \eqref{Eq_doteddote2}.
Proceeding as in the first proof, this time it suffices to
prove the formal identity
\begin{align*}
&\sum_{w\in S_k} \det_{1\leq i,j\leq k}\bigl(x_{w_i}^j
\bigl(E_{z_{w_i}+i-j}-E_{z_{w_i}-i-j}\bigr)\bigr) \\
&\qquad{}=\sum_{w\in S_k}\det_{1\leq i,j\leq k}\bigl(
x_{w_i}^j \bigl(E_{z_{w_i}+i-j-1}+E_{z_{w_i}+i-j}-
E_{z_{w_i}-i-j}-E_{z_{w_i}-i-j+1}\bigr)\bigr).
\end{align*}
This is the same as
\begin{equation}\label{Eq_Ftilde}
\sum_{w,\sigma\in S_k}\sgn(\sigma)
\prod_{i=1}^k x_{w_i}^{\sigma_i} \tilde{F}_{i,i}(\sigma,w)
=\sum_{w,\sigma\in S_k}\sgn(\sigma) \prod_{i=1}^k
x_{w_i}^{\sigma_i} \bigl( \tilde{F}_{i,i-1}(\sigma,w)+
\tilde{F}_{i,i}(\sigma,w)\bigr),
\end{equation}
where
\[
\tilde{F}_{i,j}(\sigma,w):=E_{z_{w_i}+j-\sigma_i}-E_{z_{w_i}-j-\sigma_i}.
\]
Since $\tilde{F}_{i,0}=0$ and thus $F_{1,0}=0$, \eqref{Eq_Ftilde}
can be rewritten exactly as \eqref{Eq_setsum} but with $F$ replaced
by $\tilde{F}$.
The remainder of the proof if identical to the first proof.
\end{proof}

\section{Proof, special cases and generalisations of Theorem~\ref{Thm_GLn}}
\label{Sec_GLn}

Before discussing a number of special cases and generalisations of
Theorem~\ref{Thm_GLn}, we provide a~short proof of the theorem based on
a result due to Schilling and Shimozono~\cite{SS00}.
For the sake of brevity, throughout this section we use
$Q:=\bigl\{y\in\mathbb{Z}^k\mid \abs{y}=0\bigr\}$ for the $\mathrm{A}_{k-1}$ root
lattice.
We also mostly work in the algebra of symmetric functions writing
$e_r$, $P_{\la}(t)$ and $s_{\la}$ instead of~$e_r(x)$,~$P_{\la}(x;t)$ and $s_{\la}(x)$.

\begin{proof}[Proof of Theorem~\ref{Thm_GLn}]
We will prove the result in the form \eqref{Eq_Thm_GLn2}, or, equivalently,
in the form
\[
P_{(k^r)}(t)=\sum_{y\in Q} \sum_{\sigma\in S_k} \sgn(\sigma)
t^{\frac{1}{2}\sum_{i=1}^k (ky_i+2i)y_i} e_{(r^k)+\sigma-\rho-ky},
\]
where $\rho:=(1,2,\dots,k)$ and where the elementary symmetric
function in the summand vanishes unless
$(r^k)+\sigma-\rho-ky$ is a (weak) composition.
By \eqref{Eq_transition},
\[
P_{(k^r)}(t)=\sum_{\la}\sum_{y\in Q}\sum_{\sigma\in S_k}
\sgn(\sigma) t^{\frac{1}{2}\sum_{i=1}^k (ky_i+2i)y_i}
\mathcal{R}_{\la,(r^k)+\sigma-\rho-ky}(t) P_{\la}(t),
\]
where $\mathcal{R}_{\la,\alpha}(t):=0$ if $\min\{\alpha\}<0$.
Equating coefficients of $P_{\la}(t)$ on both sides, we are left
to show that
\begin{equation}\label{Eq_delta}
\sum_{y\in Q}\sum_{\sigma\in S_k}
\sgn(\sigma) t^{\frac{1}{2}\sum_{i=1}^k (ky_i+2i)y_i}
\mathcal{R}_{\la,(r^k)+\sigma-\rho-ky}(t)=\delta_{\la,(k^r)}
\end{equation}
for all partitions $\la\vdash rk$ such that $\la_1\leq k$.
(As noted on p.~\pageref{page_R}, if $\la_1>k$ or $\la\neq\abs{kr}$,
then~${\mathcal{R}_{\la,(r^k)+\sigma-\rho-ky}(t)=0}$, in which case
there is nothing to prove.)
For $s$ a nonnegative integer such that $s\leq r$, let $r-s$ be the
multiplicity of parts of size $k$ in $\la$, and by mild abuse of
notation, write $\la=(k^{r-s},\mu)$ where $\mu\vdash ks$ such that
$\mu_1<k$.
Then, by \eqref{Eq_reduce}, the identity \eqref{Eq_delta} may be
simplified to
\begin{equation}\label{Eq_delta2}
\sum_{y\in Q}\sum_{\sigma\in S_k}
\sgn(\sigma) t^{\frac{1}{2}\sum_{i=1}^k (ky_i+2i)y_i}
\mathcal{R}_{\mu,(s^k)+\sigma-\rho-ky}(t)=\delta_{s,0}.
\end{equation}
Recalling \eqref{Eq_R}, for $k=2$ and $\mu=\bigl(1^{2s}\bigr)$ this is the
well-known (see, e.g., \cite[equation~(2.3)]{W99})
\begin{equation}\label{Eq_unit}
\sum_{y\in\mathbb{Z}}(-1)^y t^{\binom{y}{2}} \qbin{2s}{s-y}_t=\delta_{s,0}.
\end{equation}
For $k=3$, it is the $\ell=0$ instance of \cite[Proposition 5.1]{ASW99},
where $\mu$ and $s$ are parametrised as~$\mu=\bigl(2^{2L_2-L_1},1^{2L_1-L_2}\bigr)$
for $4L_2\geq 2L_1\geq L_2$ and $s=L_2$.
The identity \eqref{Eq_delta2} for general $k$ was first proposed in
\cite[equation~(6.5)]{W99} and subsequently proved by Schilling and
Shimozono \cite[equation~(6.6)]{SS00} as an identity for
\smash{$U_q\bigl(\mathrm{A}_{k-1}^{(1)}\bigr)$} level-restricted path at trivial level $0$.
\end{proof}

In the remainder of this section, we discuss a number of special cases
and generalisations of Theorem~\ref{Thm_GLn}.

On~\cite[p.~214]{Macdonald95}, Macdonald gives the following
Schur expansion for the Hall--Littlewood polynomial indexed by
a partition of length at most one.

\begin{Lemma}\label{Lem_Macdonald}
For $k$ a nonnegative integer,
\[
P_{(k)}(t)=\sum_{a=0}^{k-1} (-t)^a s_{(k-a,1^a)},
\]
where $s_{(k-0,1^0)}:=s_k$ and $s_{(1,1^{k-1})}:=s_{(1^k)}$.
\end{Lemma}

It is not hard to show that this follows from Theorem~\ref{Thm_GLn}.

\begin{proof}[Proof of Lemma~\ref{Lem_Macdonald}]
Equation~\eqref{Eq_Pe} for $r=1$ simplifies to
\[
P_{(k)}(t)=\sum_{y\in Q} \det_{1\leq i,j\leq k}
\bigl(t^{k\binom{y_i}{2}+iy_i} e_{1-i+j-ky_i} \bigr).
\]
Since $e_r=0$ for $r<0$, all entries in the first row of the
determinant are zero unless $y_1\leq 1$.
Moreover, for $2\leq i\leq k$ all entries in row $i$ are zero
unless $y_i\leq 0$.
Since $y\in Q$, this implies that there are exactly $k$ choices for
$y$ such that the determinant is non-vanishing:
\smash{$y=\bigl(0^k\bigr)$} and \smash{$y=\bigl(1,0^{k-a-1},-1,0^{a-1}\bigr)$} for $1\leq a\leq k-1$.
For \smash{$y=\bigl(0^k\bigr)$}, this yields the Schur function $s_{(k)}$
by~\eqref{Eq_JTe}.
The contribution to the sum from \smash{$y=\bigl(1,0^{k-a-1},-1,0^{a-1}\bigr)$} is
\[
D_a(t):=t^a \det_{1\leq i,j\leq k}
\bigl(e_{j-i+1+k\delta_{i,k-a+1}-k\delta_{i,1}} \bigr).
\]
All entries in the first row of the determinant are zero except
for the last, which is $1$.
By a~Laplace expansion along the first row,
\[
D_a(t)=t^a (-1)^{k-1} \det_{1\leq i,j\leq k-1}
\bigl(e_{j-i+k\delta_{i,k-a}}\bigr).
\]
In the determinant, we cycle rows
$1,2,\dots,k-a\mapsto k-a,1,2,\dots,k-a-1$.
In other words, row $k-a$ is put as the first row,
pushing down rows $1,\dots,k-a-1$, leading to
\[
D_a(t)=(-t)^a \det_{1\leq i,j\leq k-1}
\bigl(e_{j-i+1+a\delta_{i,1}-\chi(i>k-a)}\bigr),
\]
where $\chi(\text{true})=1$ and $\chi(\text{false})=0$.
By \eqref{Eq_JTe} with $k\mapsto k-1$, this yields
$D_a(t)=(-t)^a s_{\la}$, where
\[
\la'=\bigl(a+1,1^{k-a-1},0^{a-1}\bigr)=\bigl(a+1,1^{k-a-1}\bigr),
\]
and thus $\la=(k-a,1^a)$.
Hence
\[
P_{(k)}(t)=s_{(k)}+\sum_{a=1}^{k-1} (-t)^a s_{(k-a,1^a)},
\]
as required.
\end{proof}

In \cite[Theorem 7.2]{LS06}, Lassalle and Schlosser derived an expression
for $P_{\la}(t)$ for arbitrary $\la$ in terms of elementary symmetric
functions.
Their formula does not have a determinantal form, and it seems difficult
to deduce Theorem~\ref{Thm_GLn} from the $\la=(k^r)$ case of the
Lassalle--Schlosser theorem. For $k=2$, however, the two theorems
are trivially the same.

\begin{Lemma}[Lassalle--Schlosser]
For $r$ a nonnegative integer,
\[
P_{(2^r)}(t)
=\sum_{i=-r}^r (-1)^i t^{\binom{i}{2}} e_{r+i} e_{r-i}.
\]
\end{Lemma}

\begin{proof}
According to \eqref{Eq_Pe} for $k=2$,
\begin{align*}
P_{(2^r)}(t)
&=\sum_{\substack{y_1,y_2\in\mathbb{Z} \\[1pt] y_1+y_2=0}}
t^{2\binom{y_1}{2}+2\binom{y_2}{2}+y_1+2y_2}
 ( e_{r-2y_1} e_{r-2y_2}-e_{r-2y_1+1} e_{r-2y_2-1} ) \\
&=\sum_{i\in\mathbb{Z}}
t^{2i^2-i} ( e_{r-2i} e_{r+2i}-e_{r-2i+1} e_{r+2i-1} ) \\
&=\sum_{i=-r}^r(-1)^i t^{\binom{i}{2}} e_{r+i} e_{r-i},
\end{align*}
which is the Lassalle--Schlosser expression.
\end{proof}

By \eqref{Eq_Pm}, the $t=1$ specialisation of Theorem~\ref{Thm_GLn} is
\begin{equation}\label{Eq_tiseen}
m_{(k^r)}(x;t)=
\sum_{y\in Q} \det_{1\leq i,j\leq k} (e_{r-i+j-ky_i}(x) ).
\end{equation}
This identity is an immediate consequence of Proposition~\ref{Prop_cSchur}.

\begin{proof}[Proof of \eqref{Eq_tiseen}]
For $x=(x_1,\dots,x_n)$, let $f\bigl(x^k\bigr)=f\bigl(x_1^k,x_2^k,\dots,x_n^k\bigr)$.
Then
\[
m_{(k^r)}(x)=e_r\bigl(x^k\bigr).
\]
We thus need to show that
\[
\sum_{y\in Q} \det_{1\leq i,j\leq k} (e_{r-i+j-ky_i}(x) )
=e_r\bigl(x^k\bigr).
\]
This follows from the $\ell=0$ cases of Proposition~\ref{Prop_cSchur}
and \eqref{Eq_nul}.
\end{proof}

Let \smash{$\la\in\Par_{n,k}^{\ell}$}.
Then the above connection with cylindric Schur functions suggest that
the symmetric function
\begin{equation*}
S_{\la}^{k,\ell}(t)=\sum_{y\in Q} \det_{1\leq i,j\leq k}
\bigl(t^{(k+\ell)\binom{y_i}{2}+iy_i} e_{\la_i'-i+j-(k+\ell)y_i} \bigr)
\end{equation*}
-- which is a $t$-deformation of the right-hand side of \eqref{Eq_cSchur}
-- warrants further study.
As in the Schur case, for $\ell\geq l(\la)$ the summand on the right
vanishes unless $y=\bigl(0^k\bigr)$, and for such $\ell$,
\smash{$S_{\la}^{k,\ell}(t)=s_{\la}$}.
Also, from \eqref{Eq_Pe}, \smash{$S_{(k^r)}^{k,0}(t)=P_{(k^r)}(t)$}.
Using further results by Schilling and Shimozono, the Hall--Littlewood
expansion of \smash{$S_{\la}^{k,\ell}(t)$} for all admissible partitions $\la$
may be given.
Below we will describe the simplest case $\la=(k^r)$ and $\ell\geq 1$.
By abuse of notation, we denote the type $\mathrm{A}$ Cartan matrix
and its inverse by $C$ and $C^{-1}$ respectively, without reference to
the rank of the underlying root system which will always be assumed to
be clear from the context in which these matrices are used.
In the case of $\mathrm{A}_{k-1}$, $C$ and $C^{-1}$ are
$(k-1)\times (k-1)$ matrices with entries
\smash{$C_{a,b}=2\delta_{a,b}-\delta_{a,b-1}-\delta_{a,b+1}$}
and \smash{$C^{-1}_{a,b}=\min\{a,b\}-ab/k$} for $1\leq a,b\leq k-1$.
For $\mu=\bigl(k^{m_k(\mu)},\dots,1^{m_1(\mu)}\bigr)\vdash kr$, let
\[
\varphi_{\ell}(\mu)=\frac{1}{2\ell}
\sum_{a,b=1}^{k-1} C^{-1}_{a,b} m_a(\mu) m_b(\mu)
\]
and define \smash{$K_{(r^k),\mu}^{\ell}(t)$} as the coefficients in the
expansion of \smash{$S_{(k^r)}^{\ell}(t)$} in terms of Hall--Littlewood polynomials:
\[
S_{(k^r)}^{k,\ell}(t)=t^{-k\binom{r}{2}}
\sum_{\substack{\mu\vdash kr\\[1pt] \mu_1\leq k}} t^{\varphi_{\ell}(\mu)}
K_{(r^k),\mu}^{\ell}(t) P_{\mu}(t).
\]

\begin{Proposition}\label{Prop_Kostka}
Let $\mu\vdash kr$ such that $\mu_1\leq k$.
Then the polynomial
$K_{(r^k),\mu}^{\ell}(t)$
is the level-$\ell$ restricted generalised Kostka polynomial
$K_{\la R}^{\ell}(t)$
\textup{(}see, e.g., \textup{\cite{SS00,SS01}}\textup{)}
for $\la=\bigl(r^k\bigr)$ and $R$ the sequence of single-column partitions
\[
R=\bigl(
\underbrace{(1),\dots,(1)}_{m_1(\mu) \text{ times}},
\underbrace{\bigl(1^2\bigr),\dots,\bigl(1^2\bigr)}_{m_2(\mu) \text{ times}},\dots,
\underbrace{\bigl(1^k\bigr),\dots,\bigl(1^k\bigr)}_{m_k(\mu) \text{ times}}
\bigr).
\]
\end{Proposition}

By \cite[equation~(6.7)]{SS01},
\[
K_{(k^r),\mu}^{\ell}(t)=t^{k\binom{r}{2}}
\sum t^{\frac{1}{2}\sum_{a,b=1}^{k-1}\sum_{i,j=1}^{\ell-1}
C_{a,b} C^{-1}_{i,j} \tau_i^{(a)}\tau_j^{(b)}}
\prod_{a=1}^{k-1}\prod_{i=1}^{\ell-1}
\qBin{\nu_i^{(a)}+\tau_i^{(a)}}{\tau_i^{(a)}}_t,
\]
where the sum is over \smash{$\tau_i^{(a)}\in\mathbb{N}_0$} for
$1\leq a\leq k-1$ and $1\leq i\leq \ell-1$ such that
\[
\nu_i^{(a)}:=m_a(\mu) C^{-1}_{i,1}-
\sum_{b=1}^{k-1} \sum_{j=1}^{\ell-1}
C_{a,b} C^{-1}_{i,j} \tau_j^{(b)}\in\mathbb{N}_0.
\]
In particular,
\begin{equation*}
S_{(k^r)}^{k,1}(t)=
\sum_{y\in Q} \det_{1\leq i,j\leq k}
\bigl(t^{(k+1)\binom{y_i}{2}+iy_i} e_{r-i+j-(k+1)y_i} \bigr)=
\sum_{\substack{\mu\vdash kr\\[1pt] \mu_1\leq k}} t^{\varphi_{\ell}(\mu)}
P_{\mu}(t).
\end{equation*}

\begin{proof}[Proof of Proposition~\ref{Prop_Kostka}]
The proof of the positive-level case is the exact same as the proof for the
level-$0$ case, but now uses
\[
\sum_{y\in Q}\sum_{\sigma\in S_k}
\sgn(\sigma) t^{\frac{1}{2}\sum_{i=1}^k ((k+1)y_i+2i)y_i}
\mathcal{R}_{\mu,(r^k)+\sigma-\rho-(k+1)y}(t)=
t^{-k\binom{r}{2}+\varphi_{\ell}(\mu)}
K_{(r^k),\mu}^{\ell}(t)
\]
for $\mu\vdash kr$ such that $\mu_1\leq k$.
This identity was first conjectured in \cite[equation~(9.2)]{SW99}
and proved by Schilling and Shimozono for $\ell=1$ in \cite{SS00}
and for arbitrary positive level in \cite[Section 6.2]{SS01}.
\end{proof}

We remark that the $t=\ell=1$ case of Proposition~\ref{Prop_Kostka} is
\[
\sum_{y\in Q} \det_{1\leq i,j\leq k}
\bigl(e_{r-i+j-(k+1)y_i} \bigr)=
\sum_{\substack{\mu\vdash kr\\[1pt] \mu_1\leq k}} m_{\mu}(t),
\]
which follows from the $\ell=1$ case of
Proposition~\ref{Prop_cSchur} and \eqref{Eq_een}.

We conclude this section with a straightforward corollary of
Theorem~\ref{Thm_GLn} in the form of an~\smash{$\mathrm{A}_{k-1}^{(1)}$}
basic hypergeometric summation.

\begin{Corollary}
For $k\geq 1$ and $0\leq r\leq n$ integers,
\begin{gather*}
\sum_{y\in Q} \prod_{1\leq i<j\leq k}\! \bigl(1-t^{k(y_i-y_j)-i+j}\bigr)
\!\prod_{i=1}^k t^{k(k+1)\binom{y_i}{2}-iy_i}
\qbin{n+k-1}{n-r-ky_i+i-1}_t\!
 =\qbin{n}{r}_t \prod_{i=1}^k (t^{n+i};t)_{k-i}.
\end{gather*}
\end{Corollary}

For $k=3$, this result, which plays a key role in the $\mathrm{A}_2$
Bailey lemma \cite{ASW99,W25}, was first proved in~\cite[p.~692]{ASW99}.
For general $k$, different proofs may be found in \cite{Krattenthaler01, W99,W25}, see equations~(3),~(6.6) and~(A.3) in these papers,
respectively.

\begin{proof}
By \cite[p.~27]{Macdonald95}
\[
e_r\bigl(1,t,\dots,t^{n-1}\bigr)=t^{\binom{r}{2}}\qbin{n}{r}_t
\]
and \cite[p.~213]{Macdonald95}
\begin{equation}\label{Eq_213}
P_{(k^r)}\bigl(1,t\dots,t^{n-1};t\bigr)=t^{k\binom{r}{2}} \qbin{n}{r}_t,
\end{equation}
the principal specialisation $x_i=t^{i-1}$ of \eqref{Eq_Pe} is given by
\[
\sum_{y\in Q} \det_{1\leq i,j\leq k}
\biggl(t^{k(k+1)\binom{y_i}{2}+iy_i+(j-i)(r-i-ky_i)}
\qbin{n}{r-i+j-ky_i}_t \biggr)=\qbin{n}{r}_t.
\]
The determinant on the left evaluates by \cite[p.~169]{Krattenthaler90}
\[
\det_{1\leq i<j\leq k} \biggl(t^{(j-i)a_i}\qbin{n}{a_i+j}_t\biggr)
=\prod_{1\leq i<j\leq k} (1-t^{a_i-a_j})
\prod_{i=1}^k \frac{1}{(t^{n+i};t)_{k-i}}
\qbin{n+k-1}{n-a_i-1}_t
\]
for $a_i=r-i-ky_i$.
Finally, replacing $y_i$ by $-y_i$ for all $1\leq i\leq k$, we are done.
\end{proof}

\section{Proofs of special cases of Conjectures~\ref{Con_JTB}--\ref{Con_JTBC}}
\label{Sec_JTBC}

In this section, we prove some special cases of the $\mathrm{B}_n$,
$\mathrm{C}_n$ and $\mathrm{BC}_n$ Jacobi--Trudi identities.

\begin{Lemma}\label{Lem_kis1}
Conjecture~$\ref{Con_JTB}$ holds for $k=1$.
\end{Lemma}

Before proving the lemma, we recall that a Bailey pair relative to
$q^{\ell}$ (for $\ell$ a nonnegative integer) is a pair of sequences
$(\alpha_n)_{n\geq 0}$, $(\beta_n)_{n\geq 0}$ such that
\begin{equation}\label{Eq_BP}
\beta_n=\frac{1}{(q^{\ell+1};q)_{2n}}\sum_{k=0}^n
\alpha_k \qbin{2n+\ell}{n-k}_q,
\end{equation}
see, e.g.,~\cite{Andrews84,Bailey48,McLaughlin18,W01a}.
Slater \cite{Slater51} compiled an extensive list of Bailey pairs,
which when inserted into \eqref{Eq_BP} result in identities
for sums of $q$-binomial coefficients.
The proof of Lemma~\ref{Lem_kis1} and its $\mathrm{C}_n$
and $\mathrm{BC}_n$ analogues given in Lemmas~\ref{Lem_kis1b} and
\ref{Lem_kis1c} below is based on the identities arising from the
Bailey pairs A(1)--A(4), E(1), E(2), F(1),~F(2).

\begin{proof}
For $k=1$, the conjecture simplifies to the three identities
\begin{subequations}
\begin{align}
\label{Eq_A12}
&P_{(1^n)}^{\mathrm{B}_n}(x;t,0)
=\sum_{y\in\mathbb{Z}} (-1)^y t^{\frac{1}{2}(3y-1)y}
(\dot{e}_{n-3y}(x)+\dot{e}_{n-3y+1}(x)), \\
\label{Eq_F12}
&P_{(1^n)}^{\mathrm{B}_n}\bigl(x;t,-t^{1/2}\bigr)
=\sum_{y\in\mathbb{Z}} t^{\frac{1}{2}(2y-1)y}
(\dot{e}_{n-2y}(x)+\dot{e}_{n-2y+1}(x)), \\
\label{Eq_E1}
&P_{(1^n)}^{\mathrm{B}_n}(x;t,1)
=\sum_{y\in\mathbb{Z}} (-1)^y t^{y^2} \dot{e}_{n-2y}(x).
\end{align}
\end{subequations}
First we consider \eqref{Eq_A12}.
By the $k=1$ case of the bounded Littlewood identity
\eqref{Eq_Macdonald} as well as Lemma~\ref{Lem_eP},
identity \eqref{Eq_A12} is equivalent to
\begin{align*}
\sum_{s,r\geq 0} P_{(2^r,1^s)}(x;t)
={}&\sum_{y\in\mathbb{Z}} (-1)^y t^{\frac{1}{2}(3y-1)y}
\sum_{\substack{r,s\geq 0\\[1pt] s+y \text{ even}}}
\qbin{s}{\frac{1}{2}(s-3y)}_t P_{(2^r,1^s)}(x;t) \\
&{+}\,\sum_{y\in\mathbb{Z}} (-1)^y t^{\frac{1}{2}(3y-1)y}
\sum_{\substack{r,s\geq 0\\[1pt] s+y \text{ odd}}}
\qbin{s}{\frac{1}{2}(s-3y+1)}_t P_{(2^r,1^s)}(x;t).
\end{align*}
We next equate coefficients of $P_{(2^r,1^s)}(x;t)$ and make the
substitution $t\mapsto q$.
This leads to the $q$-binomial identity
\[
\sum_{y\in\mathbb{Z}}
(-1)^y q^{\frac{1}{2}(3y-1)y} \qbin{s}{\ceil{\frac{1}{2}(s-3y)}}_q=1.
\]
For $s=2n$, this follows from the Bailey pair A(1) in Slater's list
and for $s=2n+1$ it follows from the pair A(2).

We apply analogous steps to \eqref{Eq_F12} and \eqref{Eq_E1},
this time using the bounded Littlewood identities
\eqref{Eq_RW1} and \eqref{Eq_RW2} for $k=1$.
This yields
\[
\sum_{y\in\mathbb{Z}} q^{\frac{1}{2}(2y-1)y}
\qbin{s}{\ceil{\frac{1}{2}(s-2y)}}_q
=\bigl({-}q^{1/2};q^{1/2}\bigr)_{\!s}
\qquad\text{and}\qquad
\sum_{y\in\mathbb{Z}} (-1)^y q^{y^2} \qbin{2n}{n-y}=\bigl(q;q^2\bigr)_{\!n}.
\]
(In the last identity, $s$ has been replaced by $2n$; the coefficient
of $P_{(2^r,1^s)}(x;t)$ for odd $s$ is zero on both sides of~\eqref{Eq_E1}.)
For $s=2n$, the first of these identities follows from the Bailey pair~F(1)
while for $s=2n+$ it follows from F(2).
The second identity follows from the Bailey pair~E(1).
\end{proof}

\begin{Lemma}\label{Lem_kis1b}
Conjecture~$\ref{Con_JTC}$ holds for $k=1$.
\end{Lemma}

\begin{proof}
The identities to be proved are
\[
P_{(1^n)}^{\mathrm{C}_n}(x;t,0)
=\sum_{y\in\mathbb{Z}}(-1)^y t^{\binom{y}{2}}
\dot{e}_{n-2y}(x)
\qquad\text{and}\qquad
P_{(1^n)}^{\mathrm{C}_n}(x;t,t)
=\sum_{y\in\mathbb{Z}}
(-1)^y t^{2\binom{y}{2}} \dot{e}_{n-2y}(x).
\]
Following the approach of the previous proof, using Lemma~\ref{Lem_eP}
and the $k=1$ instances of~\eqref{Eq_Stembridge} and~\eqref{Eq_Stembridge2},
the underlying polynomial identities turn out
to be \eqref{Eq_unit} and
\[
\sum_{y\in\mathbb{Z}}(-1)^y q^{2\binom{y}{2}} \qbin{2n}{n-y}_q
=q^n \bigl(q;q^2\bigr)_n,
\]
where we have replaced $t$ and $s$ by $q$ and $n$, respectively.
It follows from the Bailey pair~E(4).
\end{proof}

\begin{Lemma}\label{Lem_kis1c}
Conjecture~$\ref{Con_JTBC}$ holds for $k=1$.
\end{Lemma}

\begin{proof}
The claim is that
\[
P_{(1^n)}^{\mathrm{BC}_n}\bigl(x;t,-t^{1/2},0\bigr)
=\sum_{y\in\mathbb{Z}}
(-1)^y t^{\frac{1}{2}(3y-2)y}
 (\dot{e}_{n-3y}(x)-\dot{e}_{n-3y+2}(x) ).
\]
By Lemma~\ref{Lem_eP} and \eqref{Eq_RW3},
this leads to
\[
\sum_{\substack{y\in\mathbb{Z}\\[1pt] y-s \text{ even}}}
q^{\frac{1}{2}(3y-1)y} \biggl( \qbin{s}{\frac{1}{2}(s-3y)}_q
-\qbin{s}{\frac{1}{2}(s-3y+2)}_q \biggr)=q^{s/2},
\]
where once again we have replaced $t$ by $q$.
For $s=2n$, this follows from the Bailey pair A(3) and for $s=2n+1$
from the pair~A(4).
\end{proof}

\begin{Lemma}
The $\mathrm{B}_n$ and $\mathrm{BC}_n$ identities \eqref{Eq_PBn}
and \eqref{Eq_PBCn} hold for $t=1$.
\end{Lemma}

\begin{proof}
Throughout the proof, $K=2k+1$.

By \eqref{Eq_BCn_to_Bn} and \eqref{Eq_PBCnt1}, the $t=1$ specialisation
of \eqref{Eq_PBn} is
\[
\sum_{y\in\mathbb{Z}^k}\det_{1\leq i,j\leq k}((-1)^{y_i}
(\dot{e}_{n-i+j-Ky_i}(x)+\dot{e}_{n+i+j-Ky_i-1}(x)))
=\prod_{i=1}^n \frac{x_i^{-k}-x_i^{k+1}}{1-x_i},
\]
Similarly, by \eqref{Eq_PBCnt1}, the $t=1$ case of \eqref{Eq_PBCn} is
\[
\sum_{y\in\mathbb{Z}^k} \det_{1\leq i,j\leq k}((-1)^{y_i}
(\dot{e}_{n-i+j-Ky_i}(x)-\dot{e}_{n+i+j-Ky_i}(x)))
=\prod_{i=1}^n \frac{x_i^{-k}-x_i^{k+1}}{1-x_i}.
\]
By multilinearity and
\begin{equation*}
\dot{e}_{n-r}(x)=(x_1\cdots x_n)^{-1} \sum_{m=0}^n e_m(x) e_{m+r}(x),
\end{equation*}
we must thus show that
\[
\det_{1\leq i,j\leq k}
\bigl(\overline{F}_{i-j,K}(x)+\overline{F}_{1-i-j,K}(x)\bigr)=
\det_{1\leq i,j\leq k}
\bigl(\overline{F}_{i-j,K}(x)-\overline{F}_{-i-j,K}(x)\bigr)=
\prod_{i=1}^n \frac{1-x_i^K}{1-x_i},
\]
where (see \cite[Theorem 3.1]{HKKO25})
\[
\overline{F}_{i,N}(x):=\sum_{y,m\in\mathbb{Z}} (-1)^y e_m(x)e_{m+Ny+i}(x).
\]
Easily established relations for this function are
\[
\overline{F}_{i,N}(x)=\overline{F}_{-i,N}(x)=-\overline{F}_{i-N,N}(x).
\]
The substitution $(i,j)\mapsto (k+1-i,k+1-j)$ in either one of
the determinants thus leads to the other, so that it suffices
to prove the equality of the first determinant and the product
expression on the right.
By the first of the above relations for $\overline{F}_{i,N}(x)$, it
follows that the first determinant is exactly the right-hand side
of \cite[equation~(3.2)]{HKKO25} for $w=1$ and $h=k$.
Hence we may replace the determinant by the $w=1$ and $h=k$ instance
of the left-hand side of \cite[equation~(3.2)]{HKKO25}, which in our
notation is
\[
\sum_{\la\in\Par_{n,2k}^1} \sum_{T\in\CSSYT_{n;2k,1}(\la)} x^T.
\]
By \eqref{Eq_een}, this is equal to
\[
\sum_{\mu\in\Par_{n,2k}} m_{\mu}(x)=
\prod_{i=1}^n \frac{1-x_i^{2k+1}}{1-x_i}
=\prod_{i=1}^n \frac{1-x_i^K}{1-x_i},
\]
as required.
\end{proof}

Out of the remaining identities \eqref{Eq_PBn2}, \eqref{Eq_PBn3},
\eqref{Eq_PCn} and \eqref{Eq_PCn2}, we only know how to recast the
$t=1$ case of \eqref{Eq_PCn} in terms of cylindric tableaux.
As a first step, by \eqref{Eq_BCn_to_Cn} and \eqref{Eq_PBCnt1}
the~${t=1}$ case of \eqref{Eq_PCn} can be rewritten as
\[
\sum_{y\in\mathbb{Z}^k}
\det_{1\leq i,j\leq k} (
\dot{e}_{n-Ky_i-i+j}(x)
-\dot{e}_{n-Ky_i-i-j}(x) )
=\prod_{i=1}^n \frac{x_i^{k+1}-x_i^{-k-1}}{x_i-x_i^{-1}},
\]
where $K:=2k+2$.
Following \cite{HKKO25} and defining
\begin{equation}\label{Eq_FiN}
F_{i,N}(x):=\sum_{y,m\in\mathbb{Z}} e_m(x)e_{m+Ny+i}(x),
\end{equation}
this can also be stated as
\begin{equation}\label{Eq_HKKO36}
\det_{1\leq i,j\leq k} ( F_{i-j,K}(x)-F_{i+j,K}(x) )
=\prod_{i=1}^n \frac{1-x_i^{2k+2}}{1-x_i^2}.
\end{equation}
Since $F_{i,N}(x)=F_{-i,N}(x)$, the left-hand side of
\eqref{Eq_HKKO36} is exactly the right-hand side of
\cite[equation~(3.6)]{HKKO25} for $w=2$ and $h=k$.
Furthermore,
\[
\prod_{i=1}^n \frac{1-x_i^{2k+2}}{1-x_i^2}=
\sum_{\substack{\mu \text{ even}\\[1pt] \mu_i\leq 2k}}
m_{\mu}(x_1,\dots,x_n).
\]
Equating this with the left-hand side of
\cite[equation~(3.6)]{HKKO25} yields
\[
\biggl(\:\sideset{}{'}\sum_{\la\in\Par_{n,2k}^2}
-\sideset{}{''}\sum_{\la\in\Par_{n,2k}^2}\biggr)
\sum_{T\in\CSSYT_{n;2k,2}(\la)} x^T
=\sum_{\substack{\mu \text{ even}\\[1pt] \mu_i\leq 2k}}
m_{\mu}(x_1,\dots,x_n).
\]
Here the prime (resp.\ double prime) denotes the restriction that
\smash{$\la\in\Par_{n,2k}^2$} must be of the form $((2k)^r,2a,2b)$
(resp.\ $((2k)^r,2a+1,2b+1)$) for $r\geq 0$ and
$0\leq b\leq a<k$.
A proof of this identity would imply a proof of \eqref{Eq_PCn} for $t=1$,
but so far we only managed to find a proof when $k=1$.

\section[q,t-Rogers--Ramanujan identities]{$\boldsymbol{q,t}$-Rogers--Ramanujan identities}\label{Sec_qtRR}

In this section, we will first show how
Conjectures~\ref{Con_JTB}--\ref{Con_JTBC} imply the
$q,t$-Rogers--Ramanujan identities stated in
Theorems~\ref{Thm_qtRR}--\ref{Thm_qtRR4}.
We will then show that after specialising the Hall--Littlewood parameter
$t$, many well-known Rogers--Ramanujan identities follow by a new
manifestation of level-rank duality.
Finally, we will provide proofs of
Theorems~\ref{Thm_qtRR}--\ref{Thm_qtRR4} based on Ismail's argument.


\subsection[From Jacobi--Trudi to q,t-Rogers--Ramanujan identities]{From Jacobi--Trudi to $\boldsymbol{q,t}$-Rogers--Ramanujan identities}

Because it is the most important example and includes the classical
Rogers--Ramanujan identities as a special case, we will first show that
\eqref{Eq_PCn} implies \eqref{Eq_qt-one}.
We refer to this as a `conditional proof' since it assumes the validity
of the conjectural~\eqref{Eq_PCn}.

\begin{proof}[{Conditional proof of Theorem~\ref{Thm_qtRR}}]
By the bounded Littlewood identity \eqref{Eq_Stembridge} and the
determinant identity \eqref{Eq_doteddote2} for $u=1$, \eqref{Eq_PCn}
can be rewritten as
\begin{align*}
(x_1\cdots x_n)^{-k}
\sum_{\substack{\la\text{ even} \\[1pt] \la_1\leq 2k}}
P_{\la}(x;t)&=\sum_{y\in\mathbb{Z}^k}
\det_{1\leq i,j\leq k}\bigl(t^{\frac{1}{2}Ky_i^2-j y_i}
 (\dot{e}_{n-i+j-Ky_i}-\dot{e}_{n+i+j-Ky_i} )\bigr) \\
&=\sum_{y\in\mathbb{Z}^k}
\det_{1\leq i,j\leq k}\bigl(t^{\frac{1}{2}Ky_i^2-j y_i}
 (\ddot{e}_{n-i+j-Ky_i+1}-\ddot{e}_{n+i+j-Ky_i} )\bigr),
\end{align*}
where $K=2k+2$.
We view this as two identities, one between the first two
expressions and one between the first and third expression.
In the first identity (second) identity we specialise
$x_i=q^{i-1/2}$, $x_i=q^i$, for $1\leq i\leq n$.
By \cite[p.~27]{Macdonald95}
\[
e_{n-r}\bigl(q^{n-1/2},q^{n-3/2},\dots,q^{1/2-n}\bigr)
=q^{\frac{1}{2}r^2-\frac{1}{2}n^2}\qbin{2n}{n-r}_q
\]
and
\[
e_{n-r}\bigl(q^n,q^{n-1},\dots,q^{-n}\bigr)
=q^{\binom{r+1}{2}-\binom{n+1}{2}}\qbin{2n+1}{n-r}_q,
\]
this yields
\begin{align}
&\sum_{\substack{\la\text{ even} \\[1pt] \la_1\leq 2k}}
q^{\frac{1}{2}(\sigma+1)\abs{\la}}
P_{\la}\bigl(1,q,\dots,q^{n-1};t\bigr) \notag \\
&\qquad{}=\sum_{y\in\mathbb{Z}^k}
\det_{1\leq i,j\leq k}\biggl(
t^{\frac{1}{2}Ky_i^2-j y_i}q^{\frac{1}{2}(Ky_i+i-j)(Ky_i+i-j-\sigma)}
\qbin{2n+\sigma}{n+Ky_i+i-j}_q \notag \\
&\hphantom{\qquad{}=\sum_{y\in\mathbb{Z}^k}
\det_{1\leq i,j\leq k}\biggl(}{}\, -
t^{\frac{1}{2}Ky_i^2+j y_i}q^{\frac{1}{2}(Ky_i+i+j)(Ky_i+i+j-\sigma)}
\qbin{2n+\sigma}{n+Ky_i+i+j}_q \biggr) \label{Eq_finite}
\end{align}
for $\sigma\in\{0,1\}$.
In the above we have also used multilinearity to replace $y_i$ by
$-y_i$ for all $1\leq i\leq n$ in the second term in the determinant.
The next step is to let $n$ tend to infinity using
\[
\lim_{n\to\infty} \qbin{2n+a}{n+b}_q=\frac{1}{(q;q)_{\infty}}.
\]
Thus
\begin{align*}
&\sum_{\substack{\la\text{ even} \\[1pt] \la_1\leq 2k}}
q^{\frac{1}{2}(\sigma+1)\abs{\la}} P_{\la}\bigl(1,q,q^2,\dots;t\bigr) \\
&\qquad{}=\frac{1}{(q;q)_{\infty}^k}
\sum_{y\in\mathbb{Z}^k} \det_{1\leq i,j\leq k}
\bigl( x_i^{Ky_i+i-j} p^{\frac{1}{2}Ky_i^2-j y_i}-
x_i^{Ky_i+i+j} p^{\frac{1}{2}Ky_i^2+j y_i} \bigr),
\end{align*}
where \smash{$p=tq^K$} and \smash{$x_i=q^{i-\frac{1}{2}\sigma}$}.
The right-hand side can be written in product form
by the \smash{$\mathrm{C}_k^{(1)}$} Macdonald identity \cite{Macdonald72}
\begin{align*}
&\sum_{y\in\mathbb{Z}^k}
\det_{1\leq i,j\leq k}\bigl(
x_i^{2(k+1)y_i+i-j} p^{(k+1)y_i^2-jy_i}
-x_i^{2(k+1)y_i+i+j} p^{(k+1)y_i^2+jy_i}\bigr) \\
&\qquad{}=(p;p)_{\infty}^k\prod_{i=1}^k \theta\bigl(x_i^2;p\bigr)
\prod_{1\leq i<j\leq k} \theta(x_j/x_i,x_ix_j;p),
\end{align*}
leading to
\[
\sum_{\substack{\la\text{ even} \\[1pt] \la_1\leq 2k}}
q^{\frac{1}{2}(\sigma+1)\abs{\la}} P_{\la}\bigl(1,q,q^2,\dots;t\bigr)
=\frac{(p;p)_{\infty}^k}{(q;q)_{\infty}^k}
\prod_{i=1}^k \theta\bigl(q^{2i-\sigma};p\bigr)
\prod_{1\leq i<j\leq k} \theta\bigl(q^{j-i},q^{i+j-\sigma};p\bigr).
\]
Since
\[
\prod_{i=1}^k \theta\bigl(q^{2i-1};p\bigr)
\prod_{1\leq i<j\leq k} \theta\bigl(q^{i+j-1};p\bigr)
=\prod_{i=1}^k \theta\bigl(q^i;p\bigr)
\prod_{1\leq i<j\leq k} \theta\bigl(q^{i+j};p\bigr),
\]
this may also be stated as in Theorem~\ref{Thm_qtRR}.
\end{proof}

We remark that \eqref{Eq_finite} for $k=1$ can be written in the
compact form
\begin{equation}\label{Eq_poly1}
\sum_{r=0}^n q^{(\sigma+1)r} P_{(2^r)}\bigl(1,q,\dots,q^{n-1};t\bigr)
=\sum_{i=-\infty}^{\infty}
(-1)^i t^{\binom{i}{2}} q^{i(2i-\sigma)} \qbin{2n+\sigma}{n+2i}_q.
\end{equation}
Since
\[
P_{(2^r)}\bigl(1,q,\dots,q^{n-1};1\bigr)=\qbin{n}{r}_{q^2},
\]
this is a $t$-deformation of
\begin{equation}\label{Eq_FQ}
\sum_{r=0}^N q^{r(r+\sigma)} \qbin{n}{r}_{q^2}
=\sum_{i=-\infty}^{\infty}(-1)^i q^{i(2i-\sigma)}\qbin{2n+u}{n+2i}_q
\end{equation}
due to Foda and Quano \cite[Theorem 1.2, $k=2$]{FQ95}, see also
\cite[equations~(3.2-b) and (3.3-b)]{Sills03} by Sills.\footnote{By the
$q$-binomial theorem, the left-hand side of \eqref{Eq_FQ}
can be simplified to $\bigl(-q^{\sigma+1};q^2\bigr)_n$.}
Similarly, by~\eqref{Eq_213}, with $k=2$ and $t\mapsto q$,
\eqref{Eq_poly1} is also a $t$-deformation of
\[
\sum_{r=0}^n q^{r(r+\sigma)} \qbin{n}{r}_q
=\sum_{i=-\infty}^{\infty}
(-1)^i q^{i(5i-1)/2-i\sigma} \qbin{2n+\sigma}{n+2i}_q.
\]
For $\sigma=0$, this is Bressoud's finite analogue of the first
Rogers--Ramanujan identity \cite[equation~(1.1)]{Bressoud81}
and for $\sigma=1$ it is a companion of Bressoud's identity
for the second Rogers--Ramanujan identity, given in~\cite[p.~249]{W01}.

The conditional proofs of the remaining $q,t$-Rogers--Ramanujan
identities all proceed in the exact same manner as the proof of
Theorem~\ref{Thm_qtRR}, and below we only summarise the key
steps in each proof.

\begin{proof}[Conditional proofs of the remaining $\boldsymbol{q,t}$-Rogers--Ramanujan
identities]
{\it From identi- \linebreak ties~\eqref{Eq_PBn}  and~\eqref{Eq_PBCn} to Theorem~$\ref{Thm_qtRR2}$.}
We rewrite the left-hand side of \eqref{Eq_PBn} (\eqref{Eq_PBCn})
using the bounded Littlewood identity \eqref{Eq_Macdonald}
(\eqref{Eq_RW3}), while for the right-hand side we obtain a second
determinantal expression using \eqref{Eq_doteddote} (\eqref{Eq_doteddote2})
for $u=-1$.
Next we specialise $x_i=q^{i-1/2}$ using the first determinantal
expression on the right or $x_i=q^i$ using the second determinantal
expression, and take the $n\to\infty$ limit.
This leads to
\begin{align*}
\begin{split}
&\sum_{\substack{\la \\[1pt] \la_1\leq 2k}}
q^{\frac{1}{2}(\sigma+1)\abs{\la}}P_{\la}\bigl(1,q,q^2,\dots;t\bigr) \\
&\qquad = \frac{1}{(\sigma+1)(q;q)_{\infty}^k} \sum_{y\in\mathbb{Z}^k}
\det_{1\leq i,j\leq k}\bigl(
x_i^{Ky_i+i-j} p^{\frac{1}{2}Ky_i^2-(j-\frac{1}{2})y_i}-
x_i^{Ky_i+i+j-1} p^{\frac{1}{2}Ky_i^2+(j-\frac{1}{2})y_i}\bigr)
\end{split}
\end{align*}
and
\begin{align*}
&\sum_{\substack{\la \\[1pt] \la_1\leq 2k}}
q^{\frac{1}{2}(\sigma+1)\abs{\la}}t^{\frac{1}{2}l(\la^{\textup{o}})}
P_{\la}\bigl(1,q,q^2,\dots;t\bigr) \\
&\qquad = \frac{1}{(q;q)_{\infty}^k} \sum_{y\in\mathbb{Z}^k}
\det_{1\leq i,j\leq k}\bigl(
x_i^{Ky_i+i-j} p^{\frac{1}{2}Ky_i^2-jy_i}-
x_i^{Ky_i+i+j} p^{\frac{1}{2}Ky_i^2+jy_i}\bigr),
\end{align*}
where $\sigma\in\{0,1\}$, $p=tq^K$, $K=2k+1$ and
$x_i=-q^{i-(\sigma+1)/2}$, $x_i=-q^{i-\sigma/2}$, in the first
(second) identity.
By the \smash{$\mathrm{A}_{2k}^{(2)}$} Macdonald identity \cite{Macdonald72}
in the `$\mathrm{B}_k$ form'
\begin{align}
&\sum_{y\in\mathbb{Z}^k}
\det_{1\leq i,j\leq k}\bigl(
x_i^{(2k+1)y_i+i-j} p^{\frac{1}{2}(2k+1)y_i^2-(j-\frac{1}{2})y_i}
-x_i^{(2k+1)y_i+i+j-1} p^{\frac{1}{2}(2k+1)y_i^2+(j-\frac{1}{2})y_i}\bigr)
\notag \\
&\qquad=(p;p)_{\infty}^k\prod_{i=1}^k \theta(x_i;p)\theta\bigl(px_i^2;p^2\bigr)
\prod_{1\leq i<j\leq k} \theta(x_j/x_i,x_ix_j;p) \label{Eq_MacA2k2}
\end{align}
or in the `$\mathrm{C}_k$ form'
\begin{align*}
\begin{split}
&\sum_{y\in\mathbb{Z}^k}
\det_{1\leq i,j\leq k}\bigl(
x_i^{(2k+1)y_i+i-j} p^{\frac{1}{2}(2k+1)y_i^2-jy_i}
-x_i^{(2k+1)y_i+i+j} p^{\frac{1}{2}(2k+1)y_i^2+jy_i}\bigr) \\
&\qquad=(p;p)_{\infty}^k\prod_{i=1}^k
\theta\bigl(p^{1/2}x_i;p\bigr)\theta\bigl(x_i^2;p^2\bigr)
\prod_{1\leq i<j\leq k} \theta(x_j/x_i,x_ix_j;p),
\end{split}
\end{align*}
Theorem~\ref{Thm_qtRR2} results.

\textit{From \eqref{Eq_PBn2} to Theorem~$\ref{Thm_qtRR3}$.}
We rewrite the left-hand side of \eqref{Eq_PBn2} using the bounded
Littlewood identity \eqref{Eq_RW1} and obtain a
second determinantal expression on the right using~\eqref{Eq_doteddote} for $u=1$.
Then specialising as in the previous two cases and taking the
$n\to\infty$ limit yields
\begin{align*}
&\sum_{\substack{\la \\[1pt] \la_1\leq 2k}}
q^{\frac{1}{2}(\sigma+1)\abs{\la}}
\Biggl(\prod_{i=1}^{2k-1}\bigl({-}t^{1/2};t^{1/2}\bigr)_{m_i(\la)}\Biggr)
P_{\la}\bigl(1,q,q^2,\dots;t\bigr) \\
&\qquad = \frac{1}{(\sigma+1)(q;q)_{\infty}^k} \sum_{y\in\mathbb{Z}^k}
\det_{1\leq i,j\leq k}\bigl(
x_i^{Ky_i+i-j} p^{\frac{1}{2}Ky_i^2-(j-\frac{1}{2})y_i}-
x_i^{Ky_i+i+j-1} p^{\frac{1}{2}Ky_i^2+(j-\frac{1}{2})y_i}\bigr),
\end{align*}
where $\sigma\in\{0,1\}$, $x_i=q^{i-(\sigma+1)/2}$, $p=tq^K$ and
$K=2k$.
By the \smash{$\mathrm{D}_{k+1}^{(2)}$} Macdonald identity \cite{Macdonald72}
\begin{align*}
&\sum_{y\in\mathbb{Z}^k}
\det_{1\leq i,j\leq k}\bigl(
x_i^{2ky_i+i-j} p^{ky_i^2-(j-\frac{1}{2})y_i}
-x_i^{2ky_i+i+j-1} p^{ky_i^2+(j-\frac{1}{2})y_i}\bigr) \\
&\qquad=\bigl(p^{1/2};p^{1/2}\bigr)_{\infty}(p;p)_{\infty}^{k-1}
\prod_{i=1}^k \theta\bigl(x_i;p^{1/2}\bigr)
\prod_{1\leq i<j\leq k} \theta(x_j/x_i,x_ix_j;p),
\end{align*}
this gives Theorem~\ref{Thm_qtRR3}.

\textit{From \eqref{Eq_PCn2} to Theorem~$\ref{Thm_qtRR4}$.}
We rewrite the left-hand side of \eqref{Eq_PCn2} using the bounded
Littlewood identity \eqref{Eq_Stembridge2} and obtain a
second determinantal form on the right using
\eqref{Eq_doteddote2} for $u=1$.
Then specialising the $x_i$ as before and taking the $n\to\infty$ limit
results in
\begin{align*}
&\sideset{}{'}\sum_{\substack{\la \\[1pt] \la_1\leq 2k}}
q^{\frac{1}{2}(\sigma+1)}t^{\frac{1}{2}l(\la^{\textup{o}})}
\Biggl(\prod_{i=1}^{2k-1}\bigl(t;t^2\bigr)_{\ceil{m_i(\la)/2}}\Biggr)
P_{\la}\bigl(1,q,q^2,\dots;t\bigr) \\
&\qquad = \frac{1}{(q;q)_{\infty}^k}
\sum_{\substack{y\in\mathbb{Z}^k\\[1pt] \abs{y} \text{ even}}}
\det_{1\leq i,j\leq k}\bigl(
x_i^{Ky_i+i-j} p^{\frac{1}{2}Ky_i^2-jy_i}-
x_i^{Ky_i+i+j} p^{\frac{1}{2}Ky_i^2+jy_i}\bigr),
\end{align*}
where $\sigma\in\{0,1\}$, $x_i=q^{i-\sigma/2}$, $p=tq^K$ and
$K=2k$.
By the \smash{$\mathrm{A}_{2k-1}^{(2)}$} Macdonald identity \cite{Macdonald72}
\begin{align*}
&\sum_{\substack{y\in\mathbb{Z}^k\\[1pt] \abs{y} \text{ even}}}
\det_{1\leq i,j\leq k}\bigl(
x_i^{2ky_i+i-j} p^{ky_i^2-jy_i}-
x_i^{2ky_i+i+j} p^{ky_i^2+jy_i}\bigr) \\
&\qquad=\bigl(p^2;p^2\bigr)_{\infty}(p;p)_{\infty}^{k-1}
\prod_{i=1}^k \theta\bigl(x_i^2;p^2\bigr)
\prod_{1\leq i<j\leq k} \theta(x_j/x_i,x_ix_j;p),
\end{align*}
we obtain Theorem~\ref{Thm_qtRR4}.
\end{proof}

\subsection[From q,t-Rogers--Ramanujan to Rogers--Ramanujan identities for standard modules]{From $\boldsymbol{q,t}$-Rogers--Ramanujan to Rogers--Ramanujan identities\\
for standard modules}\label{Sec_standard}

Thanks to the groundbreaking work of Lepowsky together with his
students and collaborators
\cite{BM94,Capparelli93,Capparelli96,CLM06,CMPP22,Kanade18,KR15,Lepowsky82,
LM78b,LM78a,LP85,LW81b,LW81a,LW82,LW84,MP87,MP99,Misra84,Nandi14},
it is well known that many Rogers--Ramanujan identities admit an
interpretation as identities for principal characters of standard modules
of affine Lie algebras, see also \cite[Section~5]{Sills18}.
In a beautiful manifestation of level-rank duality \cite{Frenkel82},
the $q,t$-Rogers--Ramanujan identities of
Theorems~\ref{Thm_qtRR}--\ref{Thm_qtRR4} imply many
Rogers--Ramanujan identities for the characters of
standard modules of affine Lie algebras, albeit typically not principal
ones.
Most of these were previously found, such as the GOW identities
for~\smash{$\mathrm{A}_{2n}^{(2)}$},~\smash{$\mathrm{C}_n^{(1)}$} and
\smash{$\mathrm{D}_{n+1}^{(2)}$} \cite{GOW16} or the identities from \cite{RW21},
with a total of eight new identities.
Below~we will show how all these Rogers--Ramanujan arise via
transformations for theta functions of level-rank duality type,
such as
\begin{equation}\label{Eq_theta-duality}
\frac{(p;p)_{\infty}^k}{(q;q)_{\infty}^k}
\prod_{1\leq i<j\leq k} \theta\bigl(q^{j-i},q^{i+j-1};p\bigr)
=\frac{(p;p)_{\infty}^n}{(q;q)_{\infty}^n}
\prod_{1\leq i<j\leq n} \theta\bigl(q^{j-i},q^{i+j-1};p\bigr),
\end{equation}
for $p=q^{2k+2n-1}$.

Let $\mathfrak{g}$ be one of the affine Lie algebras shown in
Figure~\ref{Fig_1}, and let $L(\La)$ denote the integrable highest
weight or standard module of $\mathfrak{g}$ of highest weight $\La$.
The character of $L(\La)$ is defined as
\[
\ch L(\La)=\sum_{\mu\in\mathfrak{h}^{\ast}} \dim(V_{\mu}) \eup^{\mu},
\]
where $V_{\mu}$ is the weight space of weight $\mu$ in the weight-space
decomposition of $L(\La)$ and $\exp(\cdot)$ is a formal exponential.
Each Rogers--Ramanujan identity below is a $q$-series identity for
\begin{equation*}
\chi_{\La}(a,b;q):=\phi_{a,b;q}\bigl(\eup^{-\La} \ch L(\La)\bigr),
\end{equation*}
where $\phi_{a,b;q}$ is the specialisation
\begin{align*}
\phi_{a,b;q}\colon\ \mathbb{Z}\bigl[\hspace{-0.9pt}\bigl[\eup^{-\alpha_0},\dots,\eup^{-\alpha_n}\bigr]\hspace{-0.9pt}\bigr]
 \to\mathbb{Z}[[a,b,q]], \qquad
\eup^{-\alpha_i}\mapsto \begin{cases}
a & \text{for $i=0$}, \\
q & \text{for $1\leq i\leq n-1$}, \\
b & \text{for $i=n$}. \end{cases}
\end{align*}
When $a=b=q$, this is the well-known principal specialisation
\cite{Lepowsky79} for which Lepowsky's numerator formula
\cite{Lepowsky82} holds
\begin{equation}\label{Eq_Lepowsky}
\phi_{q,q;q}\bigl(\eup^{-\La} \ch L(\La)\bigr)=
\prod_{\alpha\in R^{+}}
\biggl(\frac{1-q^{\ip{\La+\rho}{\alpha^{\vee}}}}
{1-q^{\ip{\rho}{\alpha^{\vee}}}}\biggr)^{\mult(\alpha)},
\end{equation}
where $\rho$ is the Weyl vector of $\mathfrak{g}$ and, for $\alpha$
a root, $\alpha^{\vee}$ is the corresponding coroot.

\begin{figure}[t]\label{Fig_1}
\tikzmath{\x=2.2;}
\tikzmath{\y=0;}
\tikzmath{\z=-2.8;}
\tikzmath{\w=5.2;}
\tikzmath{\d=10;}
\tikzmath{\e=15;}

\begin{center}
\begin{tikzpicture}[scale=0.6]
\draw (-1,\x) node {$\mathrm{B}_n^{(1)}$};
\draw (1,\x)--(1,3.2);
\draw (0,\x)--(5,2.2);
\draw (5,\x+0.07)--(6,\x+0.07);
\draw (5,\x-0.07)--(6,\x-0.07);
\draw (5.4,\x+0.2)--(5.6,\x)--(5.4,\x-0.2);
\foreach \c in {0,...,6} \draw[fill=blue] (\c,\x) circle (0.08cm);
\draw[fill=blue] (1,\x+1) circle (0.08cm);
\draw (0,\x-0.45) node {$\sc 1$};
\draw (1,\x+1.45) node {$\sc 0$};
\draw (1,\x-0.45) node {$\sc 2$};
\draw (2,\x-0.45) node {$\sc 3$};
\draw (5,\x-0.45) node {$\sc n-1$};
\draw (6,\x-0.45) node {$\sc n$};
\draw (\d-1.2,\x) node {$\mathrm{A}_{2n-1}^{(2)}$};
\draw (1+\d,\x)--(1+\d,\x+1);
\draw (0+\d,\x)--(\d+5,\x);
\draw (\d+5.6,\x+0.2)--(\d+5.4,\x)--(\d+5.6,\x-0.2);
\draw (\d+5,\x+0.07)--(\d+6,\x+0.07);
\draw (\d+5,\x-0.07)--(\d+6,\x-0.07);
\draw[fill=blue] (1+\d,\x+1) circle (0.08cm);
\foreach \c in {0,...,6} \draw[fill=blue] (\c+\d,\x) circle (0.08cm);
\draw (\d,\x-0.45) node {$\sc 1$};
\draw (\d+1,\x+1.45) node {$\sc 0$};
\draw (\d+1,\x-0.45) node {$\sc 2$};
\draw (\d+2,\x-0.45) node {$\sc 3$};
\draw (\d+5,\x-0.45) node {$\sc n-1$};
\draw (\d+6,\x-0.45) node {$\sc n$};
\draw (-1,\y) node {$\mathrm{C}_n^{(1)}$};
\draw (0,\y+0.07)--(1,\y+0.07);
\draw (0,\y-0.07)--(1,\y-0.07);
\draw (0.4,\y+0.2)--(0.6,\y)--(0.4,\y-0.2);
\draw (1,\y)--(5,\y);
\draw (5.6,\y+0.2)--(5.4,\y)--(5.6,\y-0.2);
\draw (5,\y+0.07)--(6,\y+0.07);
\draw (5,\y-0.07)--(6,\y-0.07);
\foreach \c in {0,...,6} \draw[fill=blue] (\c,\y) circle (0.08cm);
\draw (0,\y-0.45) node {$\sc 0$};
\draw (1,\y-0.45) node {$\sc 1$};
\draw (2,\y-0.45) node {$\sc 2$};
\draw (5,\y-0.45) node {$\sc n-1$};
\draw (6,\y-0.45) node {$\sc n$};
\draw (\d-1.2,\y) node {$\mathrm{D}_{n+1}^{(2)}$};
\draw (\d,\y+0.07)--(\d+1,\y+0.07);
\draw (\d,\y-0.07)--(\d+1,\y-0.07);
\draw (\d+0.6,\y+0.2)--(\d+0.4,\y)--(\d+0.6,\y-0.2);
\draw (\d+1,\y)--(\d+5,\y);
\draw (\d+5.4,\y+0.2)--(\d+5.6,\y)--(\d+5.4,\y-0.2);
\draw (\d+5,\y+0.07)--(\d+6,\y+0.07);
\draw (\d+5,\y-0.07)--(\d+6,\y-0.07);
\foreach \c in {0,...,6} \draw[fill=blue] (\c+\d,\y) circle (0.08cm);
\draw (\d+0,\y-0.45) node {$\sc 0$};
\draw (\d+1,\y-0.45) node {$\sc 1$};
\draw (\d+2,\y-0.45) node {$\sc 2$};
\draw (\d+5,\y-0.45) node {$\sc n-1$};
\draw (\d+6,\y-0.45) node {$\sc n$};
\draw (-1,\z) node {$\mathrm{D}_n^{(1)}$};
\draw (1,\z)--(1,\z+1);
\draw (0,\z)--(6,\z);
\draw (5,\z)--(5,\z+1);
\foreach \c in {0,...,6} \draw[fill=blue] (\c,\z) circle (0.08cm);
\draw[fill=blue] (1,\z+1) circle (0.08cm);
\draw[fill=blue] (5,\z+1) circle (0.08cm);
\draw (0,\z-0.45) node {$\sc 1$};
\draw (1,\z+1.45) node {$\sc 0$};
\draw (1,\z-0.45) node {$\sc 2$};
\draw (2,\z-0.45) node {$\sc 3$};
\draw (5,\z+1.45) node {$\sc n-1$};
\draw (6,\z-0.45) node {$\sc n$};
\draw (-1,\w) node {$\mathrm{A}_2^{(2)}$};
\draw (0,\w+0.1)--(1,\w+0.1);
\draw (0,\w+0.033)--(1,\w+0.033);
\draw (0,\w-0.033)--(1,\w-0.033);
\draw (0,\w-0.1)--(1,\w-0.1);
\draw (0.6,\w+0.2)--(0.4,\w)--(0.6,\w-0.2);
\foreach \c in {0,...,1} \draw[fill=blue] (\c,\w) circle (0.08cm);
\draw (0,\w-0.45) node {$\sc 0$};
\draw (1,\w-0.45) node {$\sc 1$};
\draw (\d-1,\w) node {$\mathrm{A}_{2n}^{(2)}$};
\draw (\d,\w+0.07)--(\d+1,\w+0.07);
\draw (\d,\w-0.07)--(\d+1,\w-0.07);
\draw (\d+0.6,\w+0.2)--(\d+0.4,\w)--(\d+0.6,\w-0.2);
\draw (\d+1,\w)--(\d+5,\w);
\draw (\d+5.6,\w+0.2)--(\d+5.4,\w)--(\d+5.6,\w-0.2);
\draw (\d+5,\w+0.07)--(\d+6,\w+0.07);
\draw (\d+5,\w-0.07)--(\d+6,\w-0.07);
\foreach \c in {0,...,6} \draw[fill=blue] (\c+\d,\w) circle (0.08cm);
\draw (\d,\w-0.45) node {$\sc 0$};
\draw (\d+1,\w-0.45) node {$\sc 1$};
\draw (\d+2,\w-0.45) node {$\sc 2$};
\draw (\d+5,\w-0.45) node {$\sc n-1$};
\draw (\d+6,\w-0.45) node {$\sc n$};
\end{tikzpicture}
\end{center}
\caption{The Dynkin diagrams of the affine Lie algebras
\smash{$\mathrm{A}_{2n}^{(2)}$} ($n\geq 1$),
\smash{$\mathrm{B}_n^{(1)}$} ($n\geq 3$), \smash{$\mathrm{A}_{2n-1}^{(2)}$} ($n\geq 3$),
\smash{$\mathrm{C}_n^{(1)}$} ($n\geq 2$), \smash{$\mathrm{D}_{n+1}^{(2)}$} ($n\geq 2$)
and \smash{$\mathrm{D}_n^{(1)}$} ($n\geq 4$).}
\end{figure}
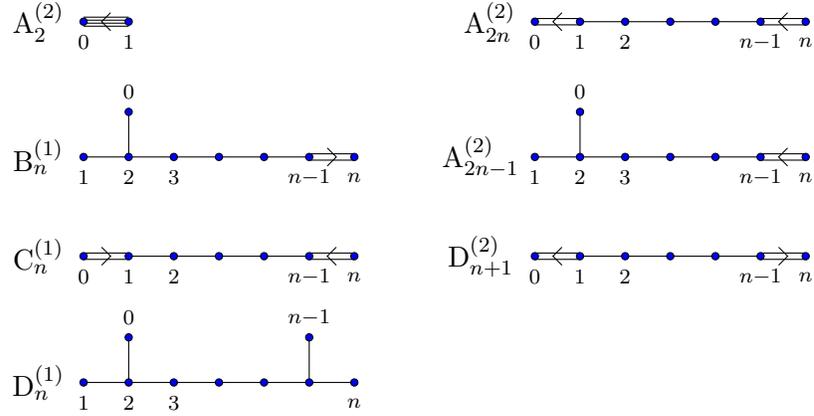

After these Lie algebraic preliminaries, we are ready to give our list of
Rogers--Ramanujan identities, together with an interpretation in terms of
specialised characters of standard modules.
In many cases, this identification needs a minor modification when
$n$ is small.
For example, \smash{$\mathrm{C}_1^{(1)}\cong \mathrm{A}_1^{(1)}$},
\smash{$\mathrm{A}_1^{(2)}\cong \mathrm{A}_1^{(1)}$},
\smash{$\mathrm{A}_3^{(2)}\cong \mathrm{D}_3^{(2)}$}, and so on.
In some but not all cases, this also means that for small $n$ the
highest weight indexing the character and/or the specialisation
$\phi_{a,b;q}$ are not correct as stated.
For example, in Corollary~\ref{Cor_17} for $n=1$ the correct
affine Lie algebra, standard module and specialisation are
\smash{$\mathrm{A}_1^{(1)}$}, \smash{$L(k\La_0+k\La_1)$} and \smash{$\phi_{q^3,q}$}, instead of
the stated~\smash{$\mathrm{A}_1^{(2)}$},~{$L(k\La_1)$} and $\phi_{q^2,q}$.
Similar such modifications apply to other corollaries.

We start with three applications of Theorem~\ref{Thm_qtRR}.

\begin{Corollary}[\smash{$\mathrm{A}_{2n}^{(2)}$} Rogers--Ramanujan identities,
{\cite[Theorem 1.1]{GOW16}}]\label{Cor_1}
For $k$, $n$ positive integers, let $p=q^{2k+2n+1}$.
Then
\begin{align}
\begin{aligned}[b]
\chi_{k\La_n}(-1,q;q)&=
\sum_{\substack{\la \textup{ even} \\[1pt] \la_1\leq 2k}}
q^{\abs{\la}/2} P_{\la}\bigl(1,q,q^2,\dots;q^{2n-1}\bigr)  \\
&=\frac{(p;p)_{\infty}^n}{(q;q)_{\infty}^n}
\prod_{i=1}^n \theta\bigl(q^{i+k};p\bigr)
\prod_{1\leq i<j\leq n} \theta\bigl(q^{j-i},q^{i+j-1};p\bigr)
\end{aligned}
\label{Eq_GOW1}
\end{align}
and
\begin{align}
\chi_{2k\La_0}(-1,q;q)&=
\sum_{\substack{\la \textup{ even} \\[1pt] \la_1\leq 2k}}
q^{\abs{\la}} P_{\la}\bigl(1,q,q^2,\dots;q^{2n-1}\bigr) \notag \\
&=\frac{(p;p)_{\infty}^n}{(q;q)_{\infty}^n}
\prod_{i=1}^n \theta\bigl(q^i;p\bigr)
\prod_{1\leq i<j\leq n} \theta\bigl(q^{j-i},q^{i+j};p\bigr),
\label{Eq_GOW2}
\end{align}
\end{Corollary}

By the principal specialisation formula \cite[p.~213]{Macdonald95}
\begin{equation}\label{Eq_PS}
P_{\la}\bigl(1,q,q^2,\dots;q\bigr)=
\prod_{i\geq 1} \frac{q^{\binom{\la'_i}{2}}}
{(q;q)_{\la'_i-\la'_{i+1}}},
\end{equation}
it follows that Corollary~\ref{Cor_1} for $n=1$ yields the
$i=k+1$ and $i=1$ instances of the modulus $2k+3$ Andrews--Gordon
identities \cite{Andrews74}
\[
\sum_{n_1\geq\cdots\geq n_k\geq 0}
\frac{q^{n_1^2+\dots+n_k^2+n_i+\cdots+n_k}}
{(q;q)_{n_1-n_2}\cdots(q;q)_{n_{k-1}-n_k}(q;q)_{n_k}} =
\frac{(q^i,q^{2k-i+3},q^{2k+3};q^{2k+3})_{\infty}}{(q;q)_{\infty}}
\]
for $1\leq i\leq k+1$.

\begin{proof}
Corollary~\ref{Cor_1} follows after specialising $t=q^{2n-1}$ in
Theorem~\ref{Thm_qtRR} and using the theta function identity
\begin{align*}
&\frac{(p;p)_{\infty}^k}{(q;q)_{\infty}^k}
\prod_{i=1}^k \theta\bigl(q^{(2-\sigma)i};p\bigr)
\prod_{1\leq i<j\leq k} \theta\bigl(q^{j-i},q^{i+j};p\bigr) \\
&\qquad=\frac{(p;p)_{\infty}^n}{(q;q)_{\infty}^n}
\prod_{i=1}^n \theta\bigl(q^{i+(1-\sigma)k};p\bigr)
\prod_{1\leq i<j\leq n} \theta\bigl(q^{j-i},q^{i+j+\sigma-1};p\bigr),
\end{align*}
for $p=q^{2k+2n+1}$ and $\sigma\in\{0,1\}$.
\end{proof}

\begin{Corollary}[$\mathrm{C}_n^{(1)}$ Rogers--Ramanujan identities,
{\cite[Theorem 1.2]{GOW16}}]
For $k$, $n$ positive integers, let $p=q^{2k+2n+2}$.
Then
\begin{align}
\chi_{k\La_0}(q,q;q)&=
\sum_{\substack{\la \textup{ even} \\[1pt] \la_1\leq 2k}}
q^{\abs{\la}/2} P_{\la}\bigl(1,q,q^2,\dots;q^{2n}\bigr) \notag \\
&=\frac{\bigl(p^{1/2};p^{1/2}\bigr)_{\infty}(p;p)_{\infty}^{n-1}}
{\bigl(q;q^2\bigr)_{\infty}(q;q)_{\infty}^n}
\prod_{i=1}^n \theta\bigl(q^i;p^{1/2}\bigr)
\prod_{1\leq i<j\leq n} \theta\bigl(q^{j-i},q^{i+j};p\bigr).
\label{Eq_GOW3}
\end{align}
\end{Corollary}

\begin{proof}
The result follows by specialising $t=q^{2n}$ in the $\sigma=0$
instance of \eqref{Eq_qt-one} and using
\begin{align*}
&\frac{(p;p)_{\infty}^k}{(q;q)_{\infty}^k}
\prod_{i=1}^k \theta\bigl(q^{2i};p\bigr)
\prod_{1\leq i<j\leq k} \theta\bigl(q^{j-i},q^{i+j};p\bigr) \notag \\
&\qquad=\frac{\bigl(p^{1/2};p\bigr)_{\infty}(p;p)_{\infty}^n}
{\bigl(q;q^2\bigr)_{\infty}(q;q)_{\infty}^n}
\prod_{i=1}^n \theta\bigl(q^i;p^{1/2}\bigr)
\prod_{1\leq i<j\leq n} \theta\bigl(q^{j-i},q^{i+j};p\bigr),
\end{align*}
where $p=q^{2k+2n+2}$.
\end{proof}

\begin{Corollary}[\smash{$\mathrm{D}_{n+1}^{(2)}$} Rogers--Ramanujan identities,
{\cite[Theorem 1.2]{GOW16}}]\label{Cor_3}
For $k$, $n$ positive integers, let $p=q^{2k+2n}$.
Then
\begin{align}
\chi_{2k\La_0}(-1,-1;q)&=
\sum_{\substack{\la \textup{ even} \\[1pt] \la_1\leq 2k}}
q^{\abs{\la}} P_{\la}\bigl(1,q,q^2,\dots;q^{2n-2}\bigr) \notag \\
&=\frac{(p;p)_{\infty}^n}{\bigl(q^2;q^2\bigr)_{\infty}(q;q)_{\infty}^{n-1}}
\prod_{1\leq i<j\leq n} \theta\bigl(q^{j-i},q^{i+j-1};p\bigr).
\label{Eq_GOW4}
\end{align}
\end{Corollary}

\begin{proof}
This follows by specialising $t=q^{2n}$ in \eqref{Eq_qt-one} for
$\sigma=1$ and using
\begin{gather}
\frac{(p;p)_{\infty}^k}{(q;q)_{\infty}^k}
\prod_{i=1}^k \theta\bigl(q^i;p\bigr)
\prod_{1\leq i<j\leq k} \theta\bigl(q^{j-i},q^{i+j};p\bigr)\nonumber\\
\qquad{}
=\frac{(p;p)_{\infty}^n}{\bigl(q^2;q^2\bigr)_{\infty}(q;q)_{\infty}^{n-1}}
\prod_{1\leq i<j\leq n} \theta\bigl(q^{j-i},q^{i+j-1};p\bigr),\label{Eq_thetaGOW}
\end{gather}
where $p=q^{2k+2n}$.
\end{proof}

The next four corollaries for even $k$ are all consequences of
\eqref{Eq_A2k2a}, whereas for odd $k$ they follow from
Corollaries~\ref{Cor_1}--\ref{Cor_3} combined with \eqref{Eq_oddk1}
for $a=1$ and $n\to\infty$.

\begin{Corollary}[\smash{$\mathrm{D}_{n+1}^{(2)}$} Rogers--Ramanujan identities,
{\cite[Theorem 5.14]{RW21}}]\label{Cor_4}
For $k$, $n$ positive integers, let $p=q^{k+2n}$.
Then
\begin{align*}
\chi_{k\La_0}\bigl(q^{1/2},-1;q\bigr)&=
\sum_{\substack{\la \\[1pt] \la_1\leq k}}
q^{\abs{\la}/2} P_{\la}\bigl(1,q,q^2,\dots;q^{2n-1}\bigr) \\
&=\frac{\bigl(-q^{1/2};q\bigr)_{\infty}(p;p)_{\infty}^n}{(q;q)_{\infty}^n}
\prod_{i=1}^n \theta\bigl(p^{1/2}q^{i-1/2};p\bigr)
\prod_{1\leq i<j\leq n} \theta\bigl(q^{j-i},q^{i+j-1};p\bigr).
\end{align*}
\end{Corollary}

By \eqref{Eq_PS}, for $n=1$ the corollary is the $i=k+1$ case of
\cite[Theorem 3]{Andrews10}
\begin{align}
&\sum_{n_1\geq\cdots\geq n_k\geq 0}
\frac{q^{\frac{1}{2}(n_1^2+\dots+n_k^2)+n_i+n_{i+2}+\dots+n_{k-1}}}
{(q;q)_{n_1-n_2}\cdots(q;q)_{n_{k-1}-n_k}(q;q)_{n_k}} \notag \\
&\qquad{}=\frac{\bigl(-q^{1/2};q\bigr)_{\infty}
\bigl(q^{i/2},q^{k-i/2+2},q^{k+2};q^{k+2}\bigr)_{\infty}}{(q;q)_{\infty}},
\label{Eq_Andrews10}
\end{align}
where $1\leq i\leq k+1$ such that $i+k$ is odd, see also
\cite[p.~449]{BIS00} and \cite[Theorem 4.4]{W03}.

\begin{proof}
For $k$ even, the claim follows by specialising $t=q^{2n-1}$ in
\eqref{Eq_A2k2a} for $\sigma=0$ and applying the $a=p^{1/2}$ case of
\begin{align}
&\frac{(p;p)_{\infty}^k}{(q;q)_{\infty}^k}\,
\prod_{i=1}^k \theta\bigl({-}p^{1/2}q^{i-1/2}/a;p\bigr)
\theta\bigl(a^2q^{2i-1};p^2\bigr)
\prod_{1\leq i<j\leq k} \theta\bigl(q^{j-i},q^{i+j-1};p\bigr) \notag \\
&\qquad=\frac{\bigl(-q^{1/2};q\bigr)_{\infty}(p;p)_{\infty}^n}{(q;q)_{\infty}^n}
\prod_{i=1}^n \theta\bigl(aq^{i-1/2};p\bigr)
\prod_{1\leq i<j\leq n} \theta\bigl(q^{j-i},q^{i+j-1};p\bigr),
\label{Eq_theta}
\end{align}
where $p=q^{2k+2n}$ and \smash{$a\in\bigl\{1,p^{1/2}\bigr\}$}.
For $k$ odd, we combine \eqref{Eq_GOW1} with the $a=1$, $n\to\infty$
and~${x_i=q^{i-1/2}}$ specialisation of \eqref{Eq_oddk1}, and use
\[
\prod_{i=1}^n \theta\bigl(q^{i+k};q^{2k+2n+1}\bigr)=
\prod_{i=1}^n \theta\bigl(q^{i+k+n};q^{2k+2n+1}\bigr)
\prod_{i=1}^n \theta\bigl(q^{(k+n+1)/2}\cdot q^{i-1/2};q^{2k+2n+1}\bigr)
\]
to write the product of theta function in the desired form.
\end{proof}

The next result is our first new Rogers--Ramanujan identity.

\begin{Corollary}[\smash{$\mathrm{A}_{2n-1}^{(2)}$} Rogers--Ramanujan identities]
\label{Cor_5}
For $k$, $n$ positive integers, let $p=q^{k+2n}$.
Then
\begin{align}
\chi_{k\La_0}(q,q;q)&=
\sum_{\substack{\la \\[1pt] \la_1\leq k}}
q^{\abs{\la}} P_{\la}\bigl(1,q,q^2,\dots;q^{2n-1}\bigr) \notag \\
&=\frac{(p;p)_{\infty}^n}{\bigl(q;q^2\bigr)_{\infty}(q;q)_{\infty}^n}
\prod_{i=1}^n \theta\bigl(q^i;p\bigr)
\prod_{1\leq i<j\leq n} \theta\bigl(q^{j-i},q^{i+j};p\bigr).
\label{Eq_GOW5}
\end{align}
\end{Corollary}

By \eqref{Eq_PS}, the $n=1$ case yields the $i=0$ instance of
the following generating function identity for certain Kleshchev
multipartitions \cite[Theorem 3]{Andrews10}, \cite[Theorem 3.5]{CSXY22},
\cite[Theorem 1.4]{KY13} or tight cylindric partitions of two rows
\cite[Theorem 13]{KR26}:
\begin{equation}\label{Eq_AKY}
\sum_{n_1\geq\cdots\geq n_k\geq 0}
\frac{q^{\binom{n_1+1}{2}+\dots+\binom{n_k+1}{2}
-n_2-n_4-\dots-n_{2i}}}
{(q;q)_{n_1-n_2}\cdots(q;q)_{n_{k-1}-n_k}(q;q)_{n_k}}
=\frac{\bigl(q^{i+1},q^{k-i+1},q^{k+2};q^{k+2}\bigr)_{\infty}}
{\bigl(q;q^2\bigr)_{\infty}(q;q)_{\infty}},
\end{equation}
where $0\leq i\leq \floor{k/2}$.

\begin{proof}
For even $k$, the result follows by choosing $t=q^{2n-1}$ in
\eqref{Eq_A2k2a} for $\sigma=1$ and applying~\eqref{Eq_thetaGOW} with
$k$ and $n$ interchanged, recalling that
\smash{$(q;q)_{\infty}/\bigl(q^2;q^2\bigr)_{\infty}=\bigl(q;q^2\bigr)_{\infty}$}.
For odd $k$, the result follows from \eqref{Eq_GOW2} and the $a=1$,
$n\to\infty$ and $x_i=q^i$ case of \eqref{Eq_oddk1}.
The identification of the result as $\chi_{k\La_0}(q,q;q)$ for
the affine Lie algebra \smash{$\mathrm{A}_{2n-1}^{(2)}$} is a direct
consequence of \eqref{Eq_Lepowsky}.
This formula implies that for the level-$\ell$ dominant integral
weights of \smash{$\mathrm{A}_{2n-1}^{(2)}$} parametrised as
\[
\La=(\ell-\la_1-\la_2)\La_0+(\la_1-\la_2)\La_1+\dots+
(\la_{n-1}-\la_n)\La_{n-1}+\la_n\La_n,
\]
where $(\la_1,\dots,\la_n)$ a partition such that $\la_1+\la_2\leq\ell$,
\[
\chi_{\La}(q,q;q)=\frac{(p;p)_{\infty}^n}{\bigl(q;q^2\bigr)_{\infty}(q;q)_{\infty}^n}
\prod_{i=1}^n \theta\bigl(q^{\la_i+n-i+1};p\bigr)
\prod_{1\leq i<j\leq n}
\theta\bigl(q^{\la_i-\la_j-i+j},q^{\la_i+\la_j+2n-i-j+2};p\bigr).
\]
Here $p=q^{\lev(\La)+h^{\vee}}=q^{\ell+2n}$.
For $\ell=k$ and $\la_1=\dots=\la_n=0$ this yields the right-hand side
of~\eqref{Eq_GOW5}.
\end{proof}

\begin{Corollary}[\smash{$\mathrm{A}_{2n}^{(2)}$} Rogers--Ramanujan identities,
{\cite[Theorem 5.13]{RW21}}]
For $k$, $n$ positive integers, let $p=q^{k+2n+1}$.
Then
\begin{align*}
\chi_{k\La_0}\bigl(q^{1/2},q;q\bigr)&=
\sum_{\substack{\la \\[1pt] \la_1\leq k}}
q^{\abs{\la}/2} P_{\la}\bigl(1,q,q^2,\dots;q^{2n}\bigr) \\
&=\frac{\bigl(p^{1/2};p^{1/2}\bigr)_{\infty}(p;p)_{\infty}^{n-1}}
{ \bigl(q^{1/2};q^{1/2}\bigr)_{\infty}(q;q)_{\infty}^{n-1}}
\prod_{i=1}^n \theta\bigl(q^i;p^{1/2}\bigr)
\prod_{1\leq i<j\leq n} \theta\bigl(q^{j-i},q^{i+j};p\bigr).
\end{align*}
\end{Corollary}

\begin{proof}
For even $k$, we take $t=q^{2n}$ in the $\sigma=0$ case of
\eqref{Eq_A2k2a} and apply
\begin{align*}
&\frac{(p;p)_{\infty}^k}{(q;q)_{\infty}^k}
\prod_{i=1}^k \theta\bigl({-}q^{i-1/2};p\bigr)
\theta\bigl(pq^{2i-1};p^2\bigr)
\prod_{1\leq i<j\leq k} \theta\bigl(q^{j-i},q^{i+j-1};p\bigr) \\
&\qquad=\frac{\bigl(p^{1/2};p^{1/2}\bigr)_{\infty}
(p;p)_{\infty}^{n-1}}
{\bigl(q^{1/2};q^{1/2}\bigr)_{\infty}(q;q)_{\infty}^{n-1}}
\prod_{i=1}^n \theta\bigl(q^i;p^{1/2}\bigr)
\prod_{1\leq i<j\leq n} \theta\bigl(q^{j-i},q^{i+j};p\bigr),
\end{align*}
where $p=q^{2k+2n+1}$.
The odd case arises from \eqref{Eq_GOW3} and \eqref{Eq_oddk1},
specialised in the same way as in the proof of Corollary~\ref{Cor_4}.
\end{proof}

A fourth and final application of \eqref{Eq_A2k2a} gives
our second new Rogers--Ramanujan identity.

\begin{Corollary}[$\mathrm{B}_n^{(1)}$ Rogers--Ramanujan identities]
\label{Cor_7}
For $k$, $n$ positive integers, let $p=q^{k+2n-1}$.
Then
\begin{gather}\label{Eq_Bn1}
\chi_{k\La_0}(q,-1;q)=
\sum_{\substack{\la \\[1pt] \la_1\leq k}}
q^{\abs{\la}} P_{\la}\bigl(1,q,q^2,\dots;q^{2n-2}\bigr)
=\frac{(p;p)_{\infty}^n}{(q;q)_{\infty}^n}
\prod_{1\leq i<j\leq n} \theta\bigl(q^{j-i},q^{i+j-1};p\bigr),
\end{gather}
\end{Corollary}

\begin{proof}
For the even-$k$ case, we specialise $t=q^{2n-2}$ in
\eqref{Eq_A2k2a} for $\sigma=1$ and apply \eqref{Eq_theta-duality}.
The odd-$k$ case follows from \eqref{Eq_GOW4} and \eqref{Eq_oddk1}
along the lines of the proof of Corollary~\ref{Cor_5}.
The identification of the identity in terms of \smash{$\mathrm{B}_n^{(1)}$}
arises as follows.
We parametrise the \mbox{level-$\ell$}~dominant integral weights of
\smash{$\mathrm{B}_n^{(1)}$} as
\[
\La=(\ell-\la_1-\la_2)\La_0+(\la_1-\la_2)\La_1+\dots+
(\la_{n-1}-\la_n)\La_{n-1}+2\la_n\La_n,
\]
where $(\la_1,\dots,\la_n)$ is a partition or half-partition
such that $\la_1+\la_2\leq\ell$.
Specialising the Weyl--Kac character formula
\cite[equation~(10.4.1)]{Kac90} for the \smash{$\mathrm{B}_n^{(1)}$}-module
$L(\La)$ according to $\phi_{q,-1;q}$, and using the \smash{$\mathrm{D}_n^{(1)}$}
Macdonald identity -- the fact that this is the right Macdonald identity
is reflected in the fact that $2n-2$ is the dual Coxeter number of
\smash{$\mathrm{D}_n^{(1)}$} -- yields
\begin{align}
&\chi_{\La}(q,-1;q) \notag \\
&\qquad = \begin{cases}\displaystyle
\frac{(p;p)_{\infty}^n}{(q;q)_{\infty}^n}
\prod_{1\leq i<j\leq n}
\theta\bigl(q^{\la_i-\la_j-i+j},q^{\la_i+\la_j+2n-i-j+1};p\bigr)
&\text{if $\la$ is a partition}, \\[3mm]
0 &\text{if $\la$ is a half-partition},
\end{cases}\!\!\! \label{Eq_Bn-specialisation}
\end{align}
where \smash{$p=q^{\ell+h^{\vee}}=q^{\ell+2n-1}$}.
For $\ell=k$ and $\la_1=\dots=\la_n=0$, this is the product on the right
of~\eqref{Eq_Bn1}.
\end{proof}

The next four corollaries are consequences of \eqref{Eq_A2k2b}.

\begin{Corollary}[\smash{$\mathrm{A}_{2n-1}^{(2)}$} Rogers--Ramanujan identities,
{\cite[Theorem 5.17]{RW21}}]\label{Cor_8}
For $k$, $n$ positive integers, let $p=q^{2k+2n}$.
Then
\begin{align*}
\chi_{k\La_n}(q,q;q)&=\sum_{\substack{\la \\[1pt] \la_1\leq 2k}}
q^{\abs{\la}/2+(n-1/2)l(\la^{\textup{o}})}
P_{\la}\bigl(1,q,q^2,\dots;q^{2n-1}\bigr) \\
&=\frac{(p;p)_{\infty}^n}{\bigl(q;q^2\bigr)_{\infty}(q;q)_{\infty}^n}
\prod_{i=1}^n \theta\bigl(p^{1/2}q^{i-1};p\bigr)
\prod_{1\leq i<j\leq n} \theta\bigl(q^{j-i},q^{i+j-2};p\bigr).
\end{align*}
\end{Corollary}

This result, which for $n=1$ yields \eqref{Eq_AKY} for
$(i,k)\mapsto(k,2k)$, is a companion of Corollary~\ref{Cor_5}.

\begin{proof}
The claim follows by taking $\sigma=0$ and $t=q^{2n-1}$ in
\eqref{Eq_A2k2b} and applying the theta function identity
\begin{align*}
&\frac{(p;p)_{\infty}^k}{(q;q)_{\infty}^k}
\prod_{i=1}^k \theta\bigl({-}p^{1/2}q^i;p\bigr)
\theta\bigl(q^{2i};p^2\bigr)
\prod_{1\leq i<j\leq k}\theta\bigl(q^{j-i},q^{i+j};p\bigr) \\
&\qquad= \frac{(p;p)_{\infty}^n}{\bigl(q;q^2\bigr)_{\infty}(q;q)_{\infty}^n}
\prod_{i=1}^n \theta\bigl(p^{1/2}q^{i-1};p\bigr)
\prod_{1\leq i<j\leq n} \theta\bigl(q^{j-i},q^{i+j-2};p\bigr)
\end{align*}
for $p=q^{2k+2n}$.
\end{proof}

The next corollary, which is a companion to Corollary~\ref{Cor_4} and
gives \eqref{Eq_Andrews10} for $(k,i)\mapsto (2k,1)$, is our third new
Rogers--Ramanujan identity.

\begin{Corollary}[\smash{$\mathrm{D}_{n+2}^{(2)}$} Rogers--Ramanujan identities]
\label{Cor_9}
For $k$, $n$ positive integers, let $p=q^{2k+2n}$.
Then
\begin{align*}
\chi_{k\La_n}\bigl(q^{1/2},-1;q\bigr)&=
\sum_{\substack{\la \\[1pt] \la_1\leq 2k}}
q^{\abs{\la}+(n-1/2)l(\la^{\textup{o}})}
P_{\la}\bigl(1,q,q^2,\dots;q^{2n-1}\bigr) \\
&=\frac{\bigl(-q^{1/2};q\bigr)_{\infty}(p;p)_{\infty}^n}{(q;q)_{\infty}^n}
\prod_{i=1}^n \theta\bigl(q^{i-1/2};p\bigr)
\prod_{1\leq i<j\leq n} \theta\bigl(q^{j-i},q^{i+j-1};p\bigr).
\end{align*}
\end{Corollary}

\begin{proof}
This follows by specialising $t=q^{2n-1}$ in \eqref{Eq_A2k2b} for
$\sigma=1$ and using \eqref{Eq_theta} for $a=1$.
The identification of the $q$-series as
\smash{$\chi_{k\La_n}\bigl(q^{1/2},-1;q\bigr)$} for the affine Lie algebra
$\mathrm{D}_{n+1}^{(2)}$ follows from the product formula for
\smash{$\chi_{\La}\bigl(q^{1/2},-1;q\bigr)$} given on \cite[p.~75]{RW21}.
\end{proof}

\begin{Corollary}[\smash{$\textrm{A}_{2n}^{(2)}$} Rogers--Ramanujan identities
{\cite[Theorem 5.13]{RW21}}]\label{Cor_10}
For $k$, $n$ positive integers, let $p=q^{2k+2n+1}$.
Then
\begin{align*}
\chi_{k\La_n}\bigl(q^{1/2},q;q\bigr)&=
\sum_{\substack{\la \\[1pt] \la_1\leq 2k}}
q^{\abs{\la}/2+n\, l(\la^{\textup{o}})}
P_{\la}\bigl(1,q,q^2,\dots;q^{2n}\bigr) \\
&=\frac{\bigl(p^{1/2};p^{1/2}\bigr)_{\infty}
(p;p)_{\infty}^{n-1}}{(q^{1/2};q^{1/2})_{\infty}(q;q)_{\infty}^{n-1}}
\prod_{i=1}^n \theta\bigl(q^{i-1/2};p^{1/2}\bigr)
\prod_{1\leq i<j\leq n} \theta\bigl(q^{j-i},q^{i+j-1};p\bigr).
\end{align*}
\end{Corollary}

\begin{proof}
The claim follows by specialising $\sigma=0$ and $t=q^{2n}$ in
\eqref{Eq_A2k2b} and using
\begin{align*}
\begin{split}
&\frac{(p;p)_{\infty}^k}{(q;q)_{\infty}^k}
\prod_{i=1}^k \theta\bigl({-}p^{1/2}q^i;p\bigr)
\theta\bigl(q^{2i};p^2\bigr)
\prod_{1\leq i<j\leq k}\theta\bigl(q^{j-i},q^{i+j};p\bigr) \\
&\qquad=\frac{\bigl(p^{1/2};p^{1/2}\bigr)_{\infty}
(p;p)_{\infty}^{n-1}}{(q^{1/2};q^{1/2})_{\infty}(q;q)_{\infty}^{n-1}}
\prod_{i=1}^n \theta\bigl(q^{i-1/2};p^{1/2}\bigr)
\prod_{1\leq i<j\leq n} \theta\bigl(q^{j-i},q^{i+j-1};p\bigr)
\end{split}
\end{align*}
for $p=q^{2k+2n+1}$.
\end{proof}

The next identity, which is a companion of \eqref{Eq_Bn1}, is our fourth
new Rogers--Ramanujan identity.

\begin{Corollary}[\smash{$\textrm{B}_n^{(1)}$} Rogers--Ramanujan identities]
\label{Cor_11}
For $k$, $n$ positive integers, let $p\!=\!q^{2k+2n-1}$.
Then
\begin{align*}
\begin{split}
\chi_{2k\La_n}\bigl(q,-1;q\bigr)&{}=
\sum_{\substack{\la \\[1pt] \la_1\leq 2k}}
q^{\abs{\la}+(n-1)\,l(\la^{\textup{o}})}
P_{\la}\bigl(1,q,q^2,\dots;q^{2n-2}\bigr)\\
&{}=\frac{(p;p)_{\infty}^n}{(q;q)_{\infty}^n}
\prod_{1\leq i<j\leq n} \theta\bigl(q^{j-i},q^{i+j-2};p\bigr).
\end{split}
\end{align*}
\end{Corollary}

\begin{proof}
We take $\sigma=1$, $t=q^{2n-2}$ in \eqref{Eq_A2k2b} and use
\begin{align*}
&\frac{(p;p)_{\infty}^k}{(q;q)_{\infty}^k}
\prod_{i=1}^k \theta\bigl({-}p^{1/2}q^{i-1/2};p\bigr)
\theta\bigl(q^{2i-1};p^2\bigr)
\prod_{1\leq i<j\leq k}\theta\bigl(q^{j-i},q^{i+j-1};p\bigr) \\
&\qquad=\frac{(p;p)_{\infty}^n}{(q;q)_{\infty}^n}
\prod_{1\leq i<j\leq n} \theta\bigl(q^{j-i},q^{i+j-2};p\bigr)
\end{align*}
for $p=q^{2k+2n-1}$.
The identification as $\chi_{2k\La_n}\bigl(q,-1;q\bigr)$
follows from \eqref{Eq_Bn-specialisation} with $k\mapsto 2k$
and $\la_i=k$ for all $1\leq i\leq k$.
\end{proof}

Next we consider applications of Theorem~\ref{Thm_qtRR3}.

\begin{Corollary}[\smash{$\textrm{B}_n^{(1)}$} Rogers--Ramanujan identities
{\cite[Theorem 5.15 and~Remark 5.16]{RW21}}]
For~$k$,~$n$ positive integers, let $p=q^{k+2n-1}$.
Then
\begin{align*}
&\chi_{k\La_{(1-\sigma)\La_0+\sigma\La_n}}\bigl(q^{1/2},q;q\bigr) \\
&\qquad=
\sum_{\substack{\la \\[1pt] \la_1\leq k}}
q^{(\sigma+1)\abs{\la}/2}
\Biggl(\prod_{i=0}^{k-1}\bigl({-}q^{n-1/2};q^{n-1/2}\bigr)_{m_i(\la)}\Biggr)
P_{\la}\bigl(1,q,q^2,\dots;q^{2n-1}\bigr) \\
&\qquad=\frac{\bigl(-q^{n-1/2};q^{n-1/2}\bigr)_{\infty}(p;p)_{\infty}^n}
{\bigl(q^{1/2};q^{1/2}\bigr)_{\infty}(q;q)_{\infty}^{n-1}}
\prod_{i=1}^n \theta\bigl(q^{i+(1-\sigma)k/2-1/2};p\bigr)
\prod_{1\leq i<j\leq n} \theta\bigl(q^{j-i},q^{i+j+\sigma-2};p\bigr),
\end{align*}
where $m_0(\la):=\infty$ and $\sigma\in\{0,1\}$.
\end{Corollary}

For $n=1$ and $q\mapsto q^2$, the above result implies the $i=k+1$
and $i=1$ cases of Bressoud's even modulus analogues
of the Andrew--Gordon identities \cite{Bressoud80,Bressoud80b}:
\[
\sum_{n_1\geq\cdots\geq n_k\geq 0}
\frac{q^{n_1^2+\dots+n_k^2+n_i+\dots+n_k}}
{(q;q)_{n_1-n_2}\cdots(q;q)_{n_{k-1}-n_k}\bigl(q^2;q^2\bigr)_{n_k}}
=\frac{\bigl(q^i,q^{2k-i+2},q^{2k+2};q^{2k+2}\bigr)_{\infty}}
{(q;q)_{\infty}},
\]
where $1\leq i\leq k+1$.

\begin{proof}
For even $k$, this follows by specialising $t=q^{2n-1}$ in
Theorem~\ref{Thm_qtRR3} and using
\begin{align*}
&\frac{\bigl(p^{1/2};p\bigr)_{\infty}
(p;p)_{\infty}^k}{(\sigma+1)(q;q)_{\infty}^k}\,
\prod_{i=1}^k \theta\bigl({-}q^{i-(\sigma+1)/2};p^{1/2}\bigr)
\prod_{1\leq i<j\leq k}\theta\bigl(q^{j-i},q^{i+j-\sigma-1};p\bigr) \\
&\qquad=\frac{(p;p)_{\infty}^n}{\bigl(q^{1/2};q\bigr)_{\infty}(q;q)_{\infty}^n}
\prod_{i=1}^n \theta\bigl(q^{i+(1-\sigma)k-1/2};p\bigr)
\prod_{1\leq i<j\leq n} \theta\bigl(q^{j-i},q^{i+j+\sigma-2};p\bigr)
\end{align*}
for $p=q^{2k+2n-1}$.
For odd $k$, the result follows from Corollaries~\ref{Cor_8} and~\ref{Cor_9} combined with~\eqref{Eq_oddk2} for $n\to\infty$ and $x_i=q^{i+(\sigma-1)/2}$.
\end{proof}

\begin{Corollary}[\smash{$\textrm{D}_{n+1}^{(2)}$} Rogers--Ramanujan identities
{\cite[Theorem 5.14]{RW21}}]
For $k$, $n$ positive integers, let $p=q^{k+2n}$.
Then
\begin{align*}
&\chi_{k\La_0}\bigl(q^{1/2},q^{1/2};q\bigr) \\
&\qquad=\sum_{\substack{\la \\[1pt] \la_1\leq k}}
q^{\abs{\la}/2}
\biggl(\prod_{i=0}^{k-1}\bigl({-}q^n;q^n\bigr)_{m_i(\la)}\biggr)
P_{\la}\bigl(1,q,q^2,\dots;q^{2n}\bigr) \\
&\qquad=\frac{\bigl(-q^{1/2};q\bigr)_{\infty}
\bigl(p^{1/2};p^{1/2}\bigr)_{\infty}(p;p)_{\infty}^{n-1}}
{\bigl(q^n;q^{2n}\bigr)_{\infty}\bigl(q^{1/2};q^{1/2}\bigr)_{\infty}(q;q)_{\infty}^{n-1}}
\prod_{i=1}^n \theta\bigl(q^{i-1/2};p^{1/2}\bigr)
\prod_{1\leq i<j\leq n} \theta\bigl(q^{j-i},q^{i+j-1};p\bigr),
\notag
\end{align*}
where $m_0(\la):=\infty$.
\end{Corollary}

By the \smash{$\textrm{D}_{n+1}^{(2)}$} diagram automorphism
$\alpha_i\mapsto\alpha_{n-i}$ for $0\leq i\leq n$, the above $q$-series
may also be identified as \smash{$\chi_{k\La_n}\bigl(q^{1/2},q^{1/2};q\bigr)$}.

\begin{proof}
For even $k$, this follows by specialising $\sigma=0$ and $t=q^{2n}$ in
Theorem~\ref{Thm_qtRR3} and the application of
\begin{align*}
&\frac{\bigl(p^{1/2};p\bigr)_{\infty}
(p;p)_{\infty}^k}{(q;q)_{\infty}^k}\,
\prod_{i=1}^k \theta\bigl({-}q^{i-1/2};p^{1/2}\bigr)
\prod_{1\leq i<j\leq k}\theta\bigl(q^{j-i},q^{i+j-1};p\bigr) \\
&\qquad=\frac{\bigl(-q^{1/2};q\bigr)_{\infty}\bigl(p^{1/2};p^{1/2}\bigr)_{\infty}
(p;p)_{\infty}^{n-1}}{\bigl(q^{1/2};q^{1/2}\bigr)_{\infty}(q;q)_{\infty}^{n-1}}
\prod_{i=1}^n \theta\bigl(q^{i-1/2};p^{1/2}\bigr)
\prod_{1\leq i<j\leq n} \theta\bigl(q^{j-i},q^{i+j-1};p\bigr)
\end{align*}
for $p=q^{2k+2n}$.
For odd $k$, the result follows from Corollary~\ref{Cor_10}
and \eqref{Eq_oddk2} for $n\to\infty$ and~${x_i=q^{i-1/2}}$.
\end{proof}

Our fifth new result is our first and only identity for \smash{$\mathrm{D}_n^{(1)}$}.

\begin{Corollary}[$\textrm{D}_n^{(1)}$ Rogers--Ramanujan identity]
\label{Cor_14}
For $k$, $n$ positive integers, let $p=q^{k+2n-2}$.
Then
\begin{align}
\chi_{k\La_0}(q,q;q)&=\sum_{\substack{\la \\[1pt] \la_1\leq k}}
q^{\abs{\la}}
\Biggl(\prod_{i=0}^{k-1}\bigl({-}q^{n-1};q^{n-1}\bigr)_{m_i(\la)}\Biggr)
P_{\la}\bigl(1,q,q^2,\dots;q^{2n-2}\bigr) \notag \\
&=\frac{(-q;q)_{\infty}\bigl(-q^{n-1};q^{n-1}\bigr)_{\infty}
(p;p)_{\infty}^n}{(q;q)_{\infty}^n}
\prod_{1\leq i<j\leq n} \theta\bigl(q^{j-i},q^{i+j-2};p\bigr),
\label{Eq_GOW6}
\end{align}
where $m_0(\la)=\infty$.
\end{Corollary}

Once again, by the automorphisms of the \smash{$\textrm{D}_n^{(1)}$} Dynkin diagram,
we may replace $\La_0$ in the above by $\La_1$, $\La_{n-1}$ or $\La_n$.

\begin{proof}
For even $k$, the second equality in \eqref{Eq_GOW6} follows by
taking $\sigma=1$ and $t=q^{2n-2}$ in Theorem~\ref{Thm_qtRR3}
and the use of
\begin{align*}
&\frac{\bigl(p^{1/2};p\bigr)_{\infty}
(p;p)_{\infty}^k}{2(q;q)_{\infty}^k}
\prod_{i=1}^k \theta\bigl({-}q^{i-1};p^{1/2}\bigr)
\prod_{1\leq i<j\leq k}\theta\bigl(q^{j-i},q^{i+j-2};p\bigr) \\
&\qquad=\frac{(-q;q)_{\infty}(p;p)_{\infty}^n}{(q;q)_{\infty}^n}
\prod_{1\leq i<j\leq n} \theta\bigl(q^{j-i},q^{i+j-2};p\bigr)
\end{align*}
for $p=q^{2k+2n-2}$.
For odd $k$, the result follows from Corollary~\ref{Cor_11} and
\eqref{Eq_oddk2} for $n\to\infty$ and~${x_i=q^i}$.
The identification of the $q$-series in terms of the \smash{$\mathrm{D}_n^{(1)}$}
standard module $L(k\La_0)$ follows from \eqref{Eq_Lepowsky}.
For \smash{$\mathrm{D}_n^{(1)}$} with parametrisation
\[
\La=(\ell-\la_1-\la_2)\La_0+(\la_1-\la_2)\La_1+\dots+
(\la_{n-1}-\la_n)\La_{n-1}+(\la_{n-1}+\la_n)\La_n,
\]
where $(\la_1,\dots,\la_n)$ is a generalised\footnote{Generalised in the
sense that $\la_n$ need not be nonnegative, as long as
$\abs{\la_n}\leq\la_{n-1}$.}
partition or half-partition such that $\la_1+\la_2\leq\ell$,
\eqref{Eq_Lepowsky} implies that
\[
\chi_{\La}(q,q;q)=
\frac{(-q;q)_{\infty}\bigl(-q^{n-1};q^{n-1}\bigr)_{\infty}(p,p)_{\infty}^n}
{(q;q)_{\infty}^n}
\prod_{1\leq i<j\leq n}
\theta\bigl(q^{\la_i-\la_j-i+j},q^{\la_i+\la_j+2n-i-j};p\bigr).
\]
Here \smash{$p=q^{\ell+h^{\vee}}=q^{\ell+2n-2}$}.
For $\ell=k$ and $\la_1=\dots=\la_n=0$, this gives yields the
product on the right of \eqref{Eq_GOW6}.
\end{proof}

Finally, we consider four applications of Theorem~\ref{Thm_qtRR4},
all but one of which are new.

\begin{Corollary}[\smash{$\textrm{A}_{2n}^{(2)}$} Rogers--Ramanujan identities]
\label{Cor_15}
For $k$ a positive integer and $n$ a nonnegative integer,
let $p=q^{2k+2n+1}$.
Then
\begin{align}
&\chi_{k\La_n}(q,q;q) \notag \\
&\qquad{}=\sideset{}{'}\sum_{\substack{\la \\[1pt] \la_1\leq 2k}}
q^{\abs{\la}/2+(n+1/2)l(\la^{\textup{o}})}
\Biggl(\prod_{i=0}^{2k-1}\bigl(q^{2n+1};q^{4n+2}\bigr)_{\ceil{m_i(\la)/2}}\Biggr)
P_{\la}\bigl(1,q,q^2,\dots;q^{2n+1}\bigr) \notag \\
&\qquad{}=\frac{\bigl(q^{2n+1};q^{4n+2}\bigr)_{\infty}(p;p)_{\infty}^n}
{\bigl(q;q^2\bigr)_{\infty}(q;q)_{\infty}^n}
\prod_{i=1}^n \theta\bigl(q^{k+i};p\bigr)\theta\bigl(q^{2i-1};p^2\bigr)
\prod_{1\leq i<j\leq n}\theta\bigl(q^{j-i},q^{i+j-1};p\bigr),
\label{Eq_GOW7}
\end{align}
where $m_0(\la):=\infty$ and the prime denotes the restriction that
the odd parts of $\la$ have even multiplicity.
\end{Corollary}

For $n=0$, this is the $i=0$ case of the somewhat unusual
Rogers--Ramanujan-type identity\footnote{In an earlier version
of this paper, only the $i=0$ and $i=k$ cases of \eqref{Eq_strange}
were stated.
Matthew Russell subsequently discovered the remaining cases, which are proved
in the appendix.}
\begin{align}
\sideset{}{'}\sum_{n_1\geq\cdots\geq n_{2k}\geq 0} &
\frac{q^{\frac{1}{2}(n_1^2+\dots+n_{2k}^2)+
\frac{1}{2}(n_1-n_2+\dots+n_{2k-1}-n_{2k})+
\frac{1}{2}(n_1+n_2+\dots+n_{2i})}}
{\bigl(q^2;q^2\bigr)_{\floor{(n_1-n_2)/2}}\cdots
\bigl(q^2;q^2\bigr)_{\floor{(n_{2k-1}-n_{2k})/2}}
(q;q)_{n_{2k}}} \notag \\
&=\frac{\bigl(-q^{2i+1},-q^{2i+1},q^{2i+1};q^{2i+1}\bigr)_{\infty}}
{\bigl(q^2;q^2\bigr)_{\infty}},
\label{Eq_strange}
\end{align}
where $0\leq i\leq k$ and the prime denotes the restriction
that $n_{2j-1}-n_{2j}$ is even for all $1\leq j\leq k$.

\begin{proof}
We take $\sigma=0$ and $t=q^{2n+1}$ in Theorem~\ref{Thm_qtRR4} and apply
\begin{align*}
&\frac{(-p;p)_{\infty}(p;p)_{\infty}^k}{(q;q)_{\infty}^k}
\prod_{i=1}^k \theta\bigl(q^{2i};p^2\bigr)
\prod_{1\leq i<j\leq k}\theta\bigl(q^{j-i},q^{i+j};p\bigr) \\
&\qquad=\frac{(-q;q)_{\infty}(p;p)_{\infty}^n}
{(q;q)_{\infty}^n}
\prod_{i=1}^n \theta\bigl(q^{k+i};p\bigr)\theta\bigl(q^{2i-1};p^2\bigr)
\prod_{1\leq i<j\leq n}\theta\bigl(q^{j-i},q^{i+j-1};p\bigr)
\end{align*}
for $p=q^{2k+2n+1}$.
Parametrising the level-$\ell$ dominant integral weight of
\smash{$\mathrm{A}_{2n}^{(2)}$} as
\[
\La=(\ell-2\la_1)\La_0+(\la_1-\la_2)\La_1+\dots+
(\la_{n-1}-\la_n)\La_{n-1}+\la_n\La_n,
\]
where $(\la_1,\dots,\la_n)$ is a partition such that
$\la_1\leq\floor{\ell/2}$, the \smash{$\mathrm{A}_{2n}^{(2)}$} case
of \eqref{Eq_Lepowsky} takes the form
\begin{align*}
\chi_{\La}(q,q;q)
&=\frac{\bigl(q^{2n+1};q^{4n+2}\bigr)_{\infty}(p;p)_{\infty}^n}
{\bigl(q;q^2\bigr)_{\infty}(q;q)_{\infty}^n}
\prod_{i=1}^n \theta\bigl(q^{\la_i+n-i+1};p\bigr)
\theta\bigl(q^{\ell-2\la_i+2i-1};p^2\bigr) \\
&\quad\times\prod_{1\leq i<j\leq n}
\theta\bigl(q^{\la_i-\la_j-i+j},q^{\la_i+\la_j+2n-i-j+2};p\bigr),
\end{align*}
where \smash{$p=q^{\ell+h^{\vee}}=q^{\ell+2n+1}$}.
For $\ell=2k$ and $\la_1=\dots=\la_n=k$, this is gives
the product on the right of \eqref{Eq_GOW7}.
\end{proof}

\begin{Corollary}[\smash{$\textrm{B}_n^{(1)}$} Rogers--Ramanujan identities]
\label{Cor_16}
For $k$, $n$ positive integers, let $p\!=\!q^{2k+2n-1}$.
Then
\begin{align}
&\chi_{2k\La_n}\bigl(q^2,-1;q\bigr) \notag\\
&\qquad{}=\sideset{}{'}\sum_{\substack{\la \\[1pt] \la_1\leq 2k}}
q^{\abs{\la}+(n-1/2)l(\la^{\textup{o}})}
\Biggl(\prod_{i=0}^{2k-1}
\bigl(q^{2n-1};q^{4n-2}\bigr)_{\ceil{m_i(\la)/2}}\Biggr)
P_{\la}\bigl(1,q,q^2,\dots;q^{2n-1}\bigr) \notag \\
&\qquad{}=\frac{\bigl(q^{2n-1};q^{4n-2}\bigr)_{\infty}(p;p)_{\infty}^n}
{2\bigl(q^2;q^2\bigr)_{\infty}(q;q)_{\infty}^{n-1}}
\prod_{i=1}^n \theta\bigl({-}q^{i-1};p\bigr)
\prod_{1\leq i<j\leq n}\theta\bigl(q^{j-i},q^{i+j-2};p\bigr),
\label{Eq_GOW9}
\end{align}
where $m_0(\la):=\infty$ and the prime denotes the restriction that
the odd parts of $\la$ have even multiplicity.
\end{Corollary}

For $n=1$, this is \eqref{Eq_strange} for $i=k$.

\begin{proof}
We take $\sigma=1$ and $t=q^{2n+1}$ in Theorem~\ref{Thm_qtRR4} and apply
\begin{align*}
&\frac{(-p;p)_{\infty}(p;p)_{\infty}^k}
{(q;q)_{\infty}^k}
\prod_{i=1}^k \theta\bigl(q^{2i-1};p^2\bigr)
\prod_{1\leq i<j\leq k}\theta\bigl(q^{j-i},q^{i+j-1};p\bigr) \\
&\qquad=\frac{(p;p)_{\infty}^n}
{2\bigl(q^2;q^2\bigr)_{\infty}(q;q)_{\infty}^{n-1}}
\prod_{i=1}^n \theta\bigl({-}q^{i-1};p\bigr)
\prod_{1\leq i<j\leq n}\theta\bigl(q^{j-i},q^{i+j};p\bigr)
\end{align*}
for $p=q^{2k+2n-1}$.
The stated specialisation is somewhat unusual in that it does
not lead to a~product form for weights other than $\ell\La_n$, for
which it follows from the Weyl--Kac formula and the \smash{$\mathrm{B}_n^{(1)}$}
Macdonald identity that
\begin{align*}
\begin{split}
& \chi_{\ell\La_n}\bigl(q^2,-1;q\bigr) \\
&\qquad=\begin{cases}\displaystyle
\frac{\bigl(q^{2n-1};q^{4n-2}\bigr)_{\infty}
(p;p)_{\infty}^n}
{2\bigl(q^2;q^2\bigr)_{\infty}(q;q)_{\infty}^{n-1}}
\prod_{i=1}^n \theta\bigl({-}q^{i-1};p\bigr)
\prod_{1\leq i<j\leq n}
\theta\bigl(q^{j-i},q^{i+j-2};p\bigr)
&\text{if $\ell$ is even}, \\[3mm] 0 &\text{if $\ell$ is odd},
\end{cases}
\end{split}
\end{align*}
where \smash{$p=q^{\ell+h^{\vee}}=q^{\ell+2n-1}$}.
The product on the right of \eqref{Eq_GOW9} corresponds to $\ell=2k$.
\end{proof}

\begin{Corollary}[\smash{$\textrm{A}_{2n-1}^{(2)}$} Rogers--Ramanujan identities
{\cite[Theorem 5.17]{RW21}}]\label{Cor_17}
For $k$, $n$ positive integers, let $p=q^{2k+2n}$.
Then
\begin{align*}
&\chi_{k\La_n}\bigl(q^2,q;q\bigr) \\
&\qquad=\sideset{}{'}\sum_{\substack{\la \\[1pt] \la_1\leq 2k}}
q^{\abs{\la}/2+n\,l(\la^{\textup{o}})}
\Biggl(\prod_{i=0}^{2k-1}\bigl(q^{2n};q^{4n}\bigr)_{\ceil{m_i(\la)/2}}\Biggr)
P_{\la}\bigl(1,q,q^2,\dots;q^{2n}\bigr) \\
&\qquad=\frac{\bigl(q^{2n};q^{4n}\bigr)_{\infty}\bigl(-p^{1/2};p\bigr)_{\infty}(p;p)_{\infty}^n}
{2(q;q)_{\infty}^n}
\prod_{i=1}^n \theta\bigl({-}q^{i-1},p^{1/2}q^{i-1};p\bigr)
\prod_{1\leq i<j\leq n}\theta\bigl(q^{j-i},q^{i+j-2};p\bigr),
\end{align*}
where $m_0(\la):=\infty$ and prime denotes the restriction that
the odd parts of $\la$ have even multiplicity.
\end{Corollary}

\begin{proof}
We set $\sigma=0$ and $t=q^{2n}$ in Theorem~\ref{Thm_qtRR4}.
The claim then follows from
\begin{align*}
&\frac{(-p;p)_{\infty}(p;p)_{\infty}^k}
{(q;q)_{\infty}^k}
\prod_{i=1}^k \theta\bigl(q^{2i};p^2\bigr)
\prod_{1\leq i<j\leq k}\theta\bigl(q^{j-i},q^{i+j};p\bigr) \\
&\qquad=
\frac{\bigl(-p^{1/2};p\bigr)_{\infty}(p;p)_{\infty}^n}{2(q;q)_{\infty}^n}
\prod_{i=1}^n \theta\bigl({-}q^{i-1},p^{1/2}q^{i-1};p\bigr)
\prod_{1\leq i<j\leq n}\theta\bigl(q^{j-i},q^{i+j-2};p\bigr)
\end{align*}
for $p=q^{2k+2n}$.
\end{proof}

\begin{Corollary}[\smash{$\textrm{D}_{n+1}^{(2)}$} Rogers--Ramanujan identities]
\label{Cor_18}
For $k$, $n$ positive integers, let $p=q^{2k+2n}$.
Then
\begin{align}
\chi_{k\La_0}(-1,q;q)&=
\sideset{}{'}\sum_{\substack{\la \\[1pt] \la_1\leq 2k}}
q^{\abs{\la}+n\,l(\la^{\textup{o}})}
\Biggl(\prod_{i=0}^{2k-1}\bigl(q^{2n};q^{4n}\bigr)_{\ceil{m_i(\la)/2}}\Biggr)
P_{\la}\bigl(1,q,q^2,\dots;q^{2n}\bigr) \label{Eq_GOW10}\\
&=\frac{\bigl(q^{2n};q^{4n}\bigr)_{\infty}\bigl(p^2;p^2\bigr)_{\infty}(p;p)_{\infty}^{n-1}}
{(q;q)_{\infty}^n} \prod_{i=1}^n \theta\bigl(q^{2i-1};p^2\bigr)
\prod_{1\leq i<j\leq n}\theta\bigl(q^{j-i},q^{i+j-1};p\bigr),
\notag
\end{align}
where $m_0(\la):=\infty$ and the prime denotes the restriction that
the odd parts of $\la$ have even multiplicity.
\end{Corollary}

\begin{proof}
Let
\[
\Theta_{k,n}(q)=\frac{(p;p)_{\infty}^k}{(q;q)_{\infty}^k}
\prod_{i=1}^k \theta\bigl(q^{2i-1};p^2\bigr)
\prod_{1\leq i<j\leq k}\theta\bigl(q^{j-i},q^{i+j-1};p\bigr),
\]
where $p=q^{2k+2n}$.
Taking $\sigma=1$ and $t=q^{2n}$ in Theorem~\ref{Thm_qtRR4} and applying
\[
\Theta_{k,n}(q)=\Theta_{n,k}(q)
\]
yields the second equality in \eqref{Eq_GOW10}.
To obtain the first equality we parametrise the level-$\ell$
dominant integral weights of \smash{$\mathrm{D}_{n+1}^{(2)}$} as
\[
\La=(\ell-2\la_1)\La_0+(\la_1-\la_2)\La_1+\dots+
(\la_{n-1}-\la_n)\La_{n-1}+2\la_n\La_n,
\]
where $(\la_1,\dots,\la_n)$ is a partition or half-partition
such that $\la_1\leq\floor{\ell/2}$.
By specialising the Weyl--Kac character formula
for the \smash{$\mathrm{D}_n^{(1)}$}-module $L(\La)$ according to
$\phi_{-1,q;q}$, it follows from the~\smash{$\mathrm{A}_{2n-1}^{(2)}$} Macdonald identity that
\begin{align*}
\begin{split}
&\chi_{\La}(-1,q;q) \\
&\qquad{}=\begin{cases}\displaystyle
\frac{\bigl(q^{2n};q^{4n}\bigr)_{\infty}\bigl(p^2;p^2\bigr)_{\infty}
(p;p)_{\infty}^{n-1}}{(q;q)_{\infty}^n}
\prod_{i=1}^n \theta\bigl(q^{2\la_i+2n-2i+1};p^2\bigr) & \\[4mm]
\displaystyle \qquad\times
\prod_{1\leq i<j\leq n}
\theta\bigl(q^{\la_i-\la_j-i+j},q^{\la_i+\la_j+2n-i-j+1};p\bigr)
&\text{if $\la$ is a partition}, \\[3mm]
0 &\text{if $\la$ is a half-partition},
\end{cases}
\end{split}
\end{align*}
where \smash{$p=q^{\ell+h^{\vee}}=q^{\ell+2n}$}.
For $k=\ell$ and $\la_1=\dots=\la_n$ this is the product on the
right of~\eqref{Eq_GOW10}.
\end{proof}

\subsection{Proofs of Theorems~\ref{Thm_qtRR}--\ref{Thm_qtRR4}}
\label{Sec_Ismail}

So far, we have only given conditional proofs of the
$q,t$-Rogers--Ramanujan identities, based on the conjectural
Jacobi--Trudi identities.
For proofs that are unconditional, we apply Ismail's analytic
argument \cite{Ismail77}.
Because this works in the exact same manner for all four theorems,
we only present the details for Theorem~\ref{Thm_qtRR}.

\begin{proof}
Fix $q\in\mathbb{C}$ such that $\abs{q}<1$ and denote the difference
between the left- and right-hand sides of \eqref{Eq_qt-one} by
$f(t)$.
Let $D$ denote the open unit disk in $\mathbb{C}$.
Then $f\colon D\to\mathbb{C}$ is analytic in an open neighbourhood $U\subset D$
of $0$.
By \eqref{Eq_GOW1}, \eqref{Eq_GOW2}, \eqref{Eq_GOW3} and \eqref{Eq_GOW4},
which were previously proved in \cite[Theorems~1.1--1.3]{GOW16},
$f\bigl(q^N\bigr)=0$ for all positive integers $N$.
Since the (geometric) sequence \smash{$\bigl(q^N\bigr)_{N\geq 1}$} has limit point $0\in U$,
$f(t)=0$ for all $t\in D$ by the identity theorem for analytic functions.
\end{proof}

There is one problem with the above proof method in that it assumes the
truth of all of the Rogers--Ramanujan identities stated in
Section~\ref{Sec_standard}.
Eight of these are new and therefore require independent proofs.
It is not difficult however to adapt the method of proof from
\cite{GOW16,RW21} to obtain the missing proofs.
In each case, the required calculations are quite lengthy and hence
we will give the full details of the most complicated to prove case
(Corollary~\ref{Cor_7}) and for the remaining cases we will only
briefly indicate the key steps.

\begin{proof}[Proof of Corollary~\ref{Cor_7}]
We require the special case of \cite[Proposition~5.10]{RW21} obtained
by setting $(t_2,t_3)=(0,-1)$ and letting $N$ tend to infinity.
In this case, we may replace the nonnegative integer $m$ in the
proposition by the nonnegative integer or half-integer $k/2$, as is
explained just above~\cite[equation~(5.2.32)]{RW21}.
Then
\begin{align}
&\sum_{\substack{\la \\[1pt] \la_1\leq k}} t^{\abs{\la}/2} P'_{\la}(x;t)
=\prod_{i=1}^n \frac{\bigl(-t^{1/2}x_i;t\bigr)_{\infty}}{(tx_i^2;t)_{\infty}}
\prod_{1\leq i<j\leq n} \frac{1}{(tx_ix_j;t)_{\infty}} \notag \\
&\qquad\times \sum_{r_1,\dots,r_n\geq 0}\,
\frac{\Delta_{\mathrm{C}}(xt^r)}{\Delta_{\mathrm{C}}(x)}
\prod_{i=1}^n x_i^{(k+1)r_i} t^{\frac{1}{2}(k+1)r_i^2}
\prod_{i,j=1}^n \biggl({-}\frac{x_i}{x_j}\biggr)^{r_i} t^{\binom{r_i}{2}}
\frac{(x_ix_j;t)_{r_i}}{(tx_i/x_j;t)_{r_i}}, \label{Eq_RW}
\end{align}
where, for $x=(x_1,\dots,x_n)$,
\[
\Delta_{\mathrm{C}}(x)=\prod_{i=1}^n \bigl(1-x_i^2\bigr)
\prod_{1\leq i<j\leq n} (x_i-x_j)(x_ix_j-1)
=\det_{1\leq i,j\leq n}\bigl(x_i^{j-1}-x_i^{2n-j+1}\bigr).
\]
We note that the simultaneous substitutions
$\bigl(t^{1/2},x_1,\dots,x_n\bigr)\mapsto \bigl(-t^{1/2},-x_1,\dots,-x_n\bigr)$
leave \eqref{Eq_RW} invariant.
As our next step we observe that the series on the right of
\eqref{Eq_RW} is the function
\smash{$L_{\text{--},(\infty^n)}^{(0)}(x_1,\dots,x_n)$} defined in
\cite[equation~(A.1b)]{BW15} with $(m,q)\mapsto (k,t)$ and
${b_1,\dots,b_{k+1}\to\infty}$, \smash{$c_1=\dots=c_{k+1}=-t^{1/2}$}.
This puts us in a position to iteratively apply \cite[Lemma A.1]{BW15}.
In~particular, in \eqref{Eq_RW} we first replace
$n\mapsto 2n-2$ (for $n\geq 2$) and then make the substitution
\[
(x_1,\dots,x_{2n-2})\mapsto
(x_1,x_2,y_2,x_3,y_3,\dots,x_{n-1},y_{n-1},x_n).
\]
Taking the limits $y_i\to x_i^{-1}$ for all $2\leq i\leq n-1$ using
\cite[Lemma A.1]{BW15}, and then taking the further limits
$x_1\to t^{1/2}$ and $x_n\to 1$, yields
\begin{align*}
&\sum_{\substack{\la \\[1pt] \la_1\leq k}} t^{\abs{\la}/2}
P'_{\la}\bigl(x_1,x_2^{\pm},\dots,x_{n-1}^{\pm},x_n;t\bigr)
=\frac{1}{(t;t)_{\infty}^n}
\prod_{i=1}^n \frac{1}{\theta(-x_i;t)}
\prod_{1\leq i<j\leq n} \frac{1}{(-x_j)\theta\bigl(x_ix_j^{\pm};t\bigr)} \\
&\qquad\times \sum_{\substack{r_1,r_n\geq 0 \\[1pt]
r_2,\dots,r_{n-1}\in\mathbb{Z}}}
 (1+\chi(r_n>0) )
\Delta_{\mathrm{B}}(-xt^r) \prod_{i=1}^n x_i^{(k+2n-1)r_i}
t^{\frac{1}{2}(k+2n-1)r_i^2-(n-\frac{1}{2})r_i},
\end{align*}
where $x_1:=t^{1/2}$, $x_n:=1$, and
\[
\Delta_{\mathrm{B}}(x)=\prod_{i=1}^n (1-x_i)
\prod_{1\leq i<j\leq n} (x_i-x_j)(x_ix_j-1)
=\det_{1\leq i,j\leq n}\bigl(x_i^{j-1}-x_i^{2n-j}\bigr).
\]
Denoting the summand on the right by $S_r$, we have
\[
S_{(r_1,r_2,\dots,r_{n-1},r_n)}=
S_{(-1-r_1,r_2,\dots,r_{n-1},r_n)}=
S_{(r_1,r_2,\dots,r_{n-1},-r_n)}.
\]
These symmetries imply that the above sum over $r_1,\dots,r_n$
can be simplified to
\[
\sum_{\substack{r\in\mathbb{Z}^n\\[1pt] \abs{r} \text{ even}}}
\Delta_{\mathrm{B}}(-xt^r) \prod_{i=1}^n x_i^{(k+2n-1)r_i}
t^{\frac{1}{2}(k+2n-1)r_i^2-(n-\frac{1}{2})r_i},
\]
although in the following we find it slightly more convenient to
drop the condition that $\abs{r}$ is even and divide the series by
two.
We now specialise $t\mapsto q^{2n-2}$ and $x_i\mapsto q^{n-i}$
\big(which is consistent with $x_1=t^{1/2}$ and $x_n=1$\big) and use the
homogeneity of the Hall--Littlewood polynomials, \eqref{Eq_PPp-spec}~and multilinearity, to obtain
\begin{equation}\label{Eq_PN}
\sum_{\substack{\la \\[1pt] \la_1\leq k}} q^{\abs{\la}}
P_{\la}\bigl(1,q,q^2,\dots;q^{2n-2}\bigr)
=\frac{\mathcal{N}}{4(q;q)_{\infty}^n},
\end{equation}
where
\[
\mathcal{N}=
\det_{1\leq i,j\leq n} \biggl(\, \sum_{r\in\mathbb{Z}}
\bigl(y_j^{i-j-2(n-1)r}+y_j^{2n-i-j+2(n-1)r}\bigr)
p^{(n-1)r^2+(n-i)r}\biggr)
\]
for $y_i=q^{n-i+1/2}$ and $p=q^{k+2n-1}$.
Next we swap $i$ and $j$ in the determinant, replace $r$ by $-r$ in the
term of the summand on the right and then again appeal to multilinearity.
As a result,
\begin{align*}
\mathcal{N}&=\sum_{r\in\mathbb{Z}^n}
\prod_{i=1}^n y_i^{2(n-1)r_i-i+1} p^{2(n-1)\binom{r_i}{2}}
\det_{1\leq i,j\leq n}\bigl((y_ip^{r_i})^{j-1}+(y_ip^{r_i})^{2n-j-1}\bigr)
\\[2mm]
&=4(p;p)_{\infty}^n\prod_{1\leq i<j\leq n} \theta\bigl(y_iy_j^{\pm};p\bigr) \\
&=4(p;p)_{\infty}^n\prod_{1\leq i<j\leq n} \theta\bigl(q^{j-i},q^{2n-i-j+1};p\bigr) \\
&=4(p;p)_{\infty}^n\prod_{1\leq i<j\leq n} \theta\bigl(q^{j-i},q^{i+j-1};p\bigr).
\end{align*}
Here the second equality (which holds for arbitrary $y_i$) follows from
the variant of the \smash{$\mathrm{D}_n^{(1)}$} Macdonald identity stated in
\cite[Appendix A.2]{KRTW24} and the third equality uses that
$y_i=q^{n-i+1/2}$.
Substituting the above in \eqref{Eq_PN} completes the proof.
\end{proof}

\begin{proof}[Sketch of proof of Corollary~\ref{Cor_5}]
As in the proof of Corollary~\ref{Cor_7}, the staring point is
\eqref{Eq_RW}, but this time we carry out the substitutions
$n\mapsto 2n-1$ followed by
\[
(x_1,\dots,x_{2n-1})\mapsto
(x_1,x_2,y_2,x_3,y_3,\dots,x_n,y_n),
\]
and then take the limits $y_i\to x_i^{-1}$ for all $2\leq i\leq n$
and $x_1\to t^{1/2}$.
Then carrying out manipulations similar (but simpler, since only the
sum over $r_1$ instead of $r_1$ and $r_n$ requires special attention)
to the proof of Corollary~\ref{Cor_7}, specialising
$t\mapsto q^{2n-1}$ and $x_i\mapsto q^{n-i+1/2}$ and using the
\smash{$\mathrm{B}_n^{(1)}$} Macdonald identity yields the identity \eqref{Eq_GOW5}.
\end{proof}

\begin{proof}[Sketch of proof of Corollary~\ref{Cor_9}]
This time the starting point is the special case of
\cite[Proposition~5.10]{RW21} obtained
by setting $(t_2,t_3)=\bigl(0,-t^{1/2}\bigr)$ and letting $N$ tend to infinity:
\begin{align}
&\sum_{\substack{\la \\[1pt] \la_1\leq 2k}}
t^{(\abs{\la}+l(\la^{\odd}))/2} P'_{\la}(x;t)
=\prod_{i=1}^n \frac{(-tx_i;t)_{\infty}}{\bigl(tx_i^2;t\bigr)_{\infty}}
\prod_{1\leq i<j\leq n} \frac{1}{(tx_ix_j;t)_{\infty}} \notag \\
&\qquad\times \sum_{r_1,\dots,r_n\geq 0}\,
\frac{\Delta_{\mathrm{B}}(xt^r)}{\Delta_{\mathrm{B}}(x)}
\prod_{i=1}^n x_i^{(2k+1)r_i} t^{\binom{r_i+1}{2}+kr_i^2}
\prod_{i,j=1}^n \biggl({-}\frac{x_i}{x_j}\biggr)^{r_i} t^{\binom{r_i}{2}}
\frac{(x_ix_j;t)_{r_i}}{(tx_i/x_j;t)_{r_i}} \label{Eq_RWb}.
\end{align}
The rest of the proof is the exact same as the proof of
Corollary~\ref{Cor_5}, including the use of the~\smash{$\mathrm{B}_n^{(1)}$}
Macdonald identity in the final step.
\end{proof}

\begin{proof}[Sketch of proof of Corollary~\ref{Cor_11}]
The identity follows by taking \eqref{Eq_RWb} and then following
the exact same steps as in the proof of Corollary~\ref{Cor_7}.
\end{proof}

\begin{proof}[Sketch of proof of Corollary~\ref{Cor_14}]
We begin with the special case of \cite[Proposition~5.10]{RW21} obtained
by setting $(t_2,t_3)=\bigl(-1,-t^{1/2}\bigr)$, letting $N$ tend to infinity
and replacing $m$ by $k/2$.
Then
\begin{align*}
&\sum_{\substack{\la \\[1pt] \la_1\leq k}}
t^{\abs{\la}/2}
\Biggl(\prod_{i=1}^{k-1} \bigl(-t^{1/2};t^{1/2}\bigr)_{m_i(\la)}\Biggr)P'_{\la}(x;t)
=\prod_{i=1}^n \frac{1}{\bigl(t^{1/2}x_i;t^{1/2}\bigr)_{\infty}}
\prod_{1\leq i<j\leq n} \frac{1}{(tx_ix_j;t)_{\infty}} \\
&\qquad\times \sum_{r_1,\dots,r_n\geq 0}\,
\frac{\Delta_{\mathrm{B}}(xt^r)}{\Delta_{\mathrm{B}}(x)}
\prod_{i=1}^n x_i^{kr_i} t^{\frac{1}{2}kr_i^2+\frac{1}{2}r_i}
\prod_{i,j=1}^n \biggl({-}\frac{x_i}{x_j}\biggr)^{r_i} t^{\binom{r_i}{2}}
\frac{(x_ix_j;t)_{r_i}}{(tx_i/x_j;t)_{r_i}}.
\end{align*}
The remainder follows the exact same steps as in the proof of
Corollary~\ref{Cor_7}.
\end{proof}

\begin{proof}[Sketch of proof of Corollary~\ref{Cor_15}]
In \cite[Proposition~5.10]{RW21}, we set $t_2=-t_3=-t^{1/2}$ and
let $N$ tend to infinity.
Then
\begin{align}
&\sideset{}{'}\sum_{\substack{\la \\[1pt] \la_1\leq 2k}}
t^{(\abs{\la}+l(\la^{\textup{o}}))/2}
\Biggl(\prod_{i=1}^{2k-1}(t;t^2)_{\ceil{m_i(\la)/2}}\Biggr)
P'_{\la}\bigl(x;t\bigr)
=\prod_{i=1}^n \frac{1}{\bigl(tx_i^2;t^2\bigr)_{\infty}}
\prod_{1\leq i<j\leq n} \frac{1}{(tx_ix_j;t)_{\infty}} \notag \\
&\qquad\times \sum_{r_1,\dots,r_n\geq 0}\,
\frac{\Delta_{\mathrm{D}}(xt^r)}{\Delta_{\mathrm{D}}(x)}
\prod_{i=1}^n (-1)^{r_i} x_i^{2kr_i} t^{kr_i^2+r_i}
\prod_{i,j=1}^n \biggl({-}\frac{x_i}{x_j}\biggr)^{r_i} t^{\binom{r_i}{2}}
\frac{(x_ix_j;t)_{r_i}}{(tx_i/x_j;t)_{r_i}}, \label{Eq_RWc}
\end{align}
where the prime in the final sum over $\la$ denotes the
restriction that the odd parts of $\la$ must have even multiplicity,
and
\[
\Delta_{\mathrm{D}}(x)=\prod_{1\leq i<j\leq n}(x_i-x_j)(x_ix_j-1)
=\frac{1}{2}\det_{1\leq i,j\leq n}\bigl(x_i^{j-1}+x_i^{2n-j-1}\bigr).
\]
We now replace $n\mapsto 2n+1$ followed by
\[
(x_1,\dots,x_{2n+1})\mapsto
(x_1,y_1,x_2,y_2,\dots,x_n,y_n,x_{n+1}),
\]
and then take the limits $y_i\to x_i^{-1}$ for all $1\leq i\leq n$
and $x_{n+1}\to 1$.
Carrying out the specialisation $t\mapsto q^{2n+1}$ and
$x_i\mapsto q^{n-i}$ and using \eqref{Eq_PPp-spec} on the left and
the \smash{$\mathrm{A}_{2n}^{(2)}$} Macdonald identity on the right
yields the result.
\end{proof}

\begin{proof}[Sketch of proof of Corollary~\ref{Cor_16}]
As in the previous proof, we take \eqref{Eq_RWc} but now carry out the
same steps as in the proof of Corollary~\ref{Cor_5}, including the
use of the \smash{$\mathrm{B}_n^{(1)}$} Macdonald identity to obtain the
product of theta functions.
\end{proof}

\begin{proof}[Sketch of proof of Corollary~\ref{Cor_18}]
Yet again we take \eqref{Eq_RWc} but then follow the proof of
Corollary~\ref{Cor_7}, except that in the final step we use the
\smash{$\mathrm{A}_{2n-1}^{(2)}$} Macdonald identity.
\end{proof}

\section{Open problems}\label{Sec_open}

To conclude the paper, we list a number of open problems pertaining to
affine Jacobi--Trudi and $q,t$-Rogers--Ramanujan identities,
beyond the obvious problem of proving
Conjectures~\ref{Con_JTB}--\ref{Con_JTBC}.

One such problem is to better understand algebraically for what values
of the parameters~$t$,~$s$ in \smash{$P_{(k^n)}^{\mathrm{B}_n}(x;t,s)$} and
\smash{$P_{(k^n)}^{\mathrm{C}_n}(x;t,s)$}, or
$t$, $s_1$, $s_2$ in~\smash{$P_{(k^n)}^{\mathrm{C}_n}(x;t,s_1,s_2)$}
one should expect further Jacobi--Trudi formulas.
Our list of results is definitely not complete, and for example we
conjecture that the additional \smash{$\mathrm{A}_{2k-1}^{(2)}$} identity
\[
P_{(k^n)}^{\mathrm{B}_n}(x;t,t)
=\frac{1}{2}\sum_{y\in\mathbb{Z}^k} \det_{1\leq i,j\leq k}\bigl(
(-1)^{y_i} t^{\frac{1}{2}Ky_i^2-(j-1)y_i} (
\ddot{e}_{n-i+j-Ky_i}(x)+\ddot{e}_{n+i+j-Ky_i-1}(x) )\bigr)
\]
holds, where $K=2k$.
What is unusual about this Jacobi--Trudi formula is that there does not
appear to be a companion for $\dot{e}_r$.
As a consequence, the $q,t$-Rogers--Ramanujan identity
\begin{align*}
&\sum_{\substack{\la \\[1pt] \la_1\leq 2k}}
q^{\abs{\la}}
\Biggl(\prod_{i=1}^{2k-1} \bigl(t;t^2\bigr)_{\ceil{m_i(\la)/2}}\Biggr)
 P_{\la}\bigl(1,q,q^2,\dots;t\bigr) \\
&\qquad=\frac{\bigl(p^2;p^2\bigr)_{\infty}(p;p)_{\infty}^{k-1}}
{(q;q)_{\infty}^k}
\prod_{i=1}^k \theta\bigl(pq^{2i-1};p^2\bigr)
\prod_{1\leq i<j\leq k}
\theta\bigl(q^{j-i},q^{i+j-1};p\bigr),
\end{align*}
where $p=tq^{2k}$, does not appear to have a counterpart for
\[
\sum_{\substack{\la \\[1pt] \la_1\leq 2k}}
q^{\abs{\la}/2}
\Biggl(\prod_{i=1}^{2k-1} \bigl(t;t^2\bigr)_{\ceil{m_i(\la)/2}}\Biggr)
P_{\la}\bigl(1,q,q^2,\dots;t\bigr).
\]
It remains unclear if this is an isolated case or that further
such results exist.
In particular, we note that the affine Lie algebras
\smash{$\mathrm{B}_k^{(1)}$} and \smash{$\mathrm{D}_k^{(1)}$} are currently
missing from our list of $q,t$-Rogers--Ramanujan identities.

It would also be interesting to consider the determinants in
Conjectures~\ref{Con_JTB}--\ref{Con_JTBC} with $K=2k+a$ (where~$a$
is one of $0$, $1$, $2$ depending on the conjecture) replaced by $K=2k+2\ell+a$
for~$\ell$ a~nonnegative integer.
For example, what are the \smash{$P_{\la}^{\mathrm{C}_n}(x;t,0)$} and
\smash{$(x_1\cdots x_n)^{-k} P_{\la}(x;t)$} expansions~of%
\begin{equation}\label{Eq_generalise}
\sum_{y\in\mathbb{Z}^k} \det_{1\leq i,j\leq k}\bigl(
t^{(k+\ell+1)y_i^2-j y_i} (
\dot{e}_{n-i+j-2(k+\ell+1)y_i}(x)-
\dot{e}_{n+i+j-2(k+\ell+1)y_i}(x) )\bigr)?
\end{equation}
The answers to these questions should imply further character identities.
For instance, in the simplest possible case, namely $k=\ell=1$, we have
\begin{align*}
\sum_{y\in\mathbb{Z}}
t^{3y^2-y} (\dot{e}_{n-6y}(x)-\dot{e}_{n-6y+2}(x) )
&=\sum_{s=0}^{\floor{n/2}} t^{s^2} P_{(1^{n-2s})}^{\mathrm{C}_n}(x;t,0) \\
&=(x_1\cdots x_n)^{-1}
\sum_{r,s\geq 0} t^{s^2} P_{(2^r,1^{2s})}(x;t).
\end{align*}
Here the equality between the expression on the left and the second
expression on the right follows from
\[
\sum_{y\in\mathbb{Z}} q^{3y^2-y}\biggl( \qbin{2n}{n-3y}_q-\qbin{2n}{n-3y+1}
\biggr) = t^{n^2}
\]
-- an identity equivalent to the Bailey pair A(5) in Slater's list --
whereas the equality between the two expressions on the right may
be established using the $k=1$ case of \eqref{Eq_Macdonald} and
\[
P_{(1^n)}^{\mathrm{B}_n}(x;t,0)=
\sum_{k=0}^n P_{(1^k)}^{\mathrm{C}_n}(x;t,0).
\]
Specialising $x_i=q^{i-1/2}$ for $1\leq i\leq n$, or first using
\eqref{Eq_doteddote2} and then specialising $x_i=q^i$ for $1\leq i\leq n$,
yields
\begin{align*}
&\sum_{r,s\geq 0} q^{(\sigma+1)(r+s)} t^{s^2}
P_{(2^r,1^{2s})}\bigl(1,q,\dots,q^{n-1};t\bigr) \\
&\qquad =
\sum_{y\in\mathbb{Z}} t^{3y^2-y}
\biggl( q^{3y(6y-\sigma)}\qbin{2n+\sigma}{n-6y+\sigma}_q
-q^{(3y-1)(6y+\sigma-2)}\qbin{2n+\sigma}{n-6y+2}_q \biggr)
\end{align*}
for $\sigma\in\{0,1\}$.
Taking the large-$n$ limit and applying Watson's quintuple product
identity \cite{Watson29} -- the \smash{$\mathrm{A}_{2k}^{(2)}$} Macdonald
identity \eqref{Eq_MacA2k2} for $k=1$ -- gives the further
$q,t$-Rogers--Ramanujan identity
\[
\sum_{r,s\geq 0} q^{(\sigma+1)(r+s)} t^{s^2}
P_{(2^r,1^{2s})}\bigl(1,q,q^2,\dots;t\bigr)=
\frac{(p;p)_{\infty}}{(q;q)_{\infty}}\,
\theta\bigl(q^{2-\sigma};p\bigr)\theta\bigl(pq^{4-2\sigma};p^2\bigr),
\]
where $p=t^2q^{12}$.
For $t=q$, this is
\[
\sum_{r,s\geq 0} \frac{q^{(r+s)^2+2s^2+\sigma(r+s)}}{(q;q)_r(q;q)_{2s}}=
\frac{\bigl(q^{14};q^{14}\bigr)_{\infty}}{(q;q)_{\infty}}\,
\theta\bigl(q^{2-\sigma};q^{14}\bigr)\theta\bigl(q^{10+2\sigma};q^{28}\bigr),
\]
which is a well known formula for the (normalised) character
\smash{$\chi_{1,2-\sigma}^{(3,7)}(q)$} of the Virasoro algebra
of central charge $c=-25/7$ and conformal dimension $-5(1-\sigma)/28$,
see, e.g.,~\cite{BM96}.

Recalling the notation \eqref{Eq_FiN}, \eqref{Eq_generalise} for $t=1$
is
\[
(x_1\cdots x_n)^{-k} \det_{1\leq i,j\leq k}
 ( F_{i-j,2(k+\ell+1)}(x)-F_{i+j,2(k+\ell+1)}(x) ),
\]
for which Huh et al.\ \cite[Theorem 3.3]{HKKO25} obtained a combinatorial
expression in terms of cylindric tableaux as follows:
\[
(x_1\cdots x_n)^{-k}
\biggl(\:\sideset{}{'}\sum_{\la\in\Par_{n,2k}^{2\ell+2}}
-\sideset{}{''}\sum_{\la\in\Par_{n,2k}^{2\ell+2}}\biggr)
\sum_{T\in\CSSYT_{n;2k,2\ell+2}(\la)} x^T.
\]
Here the prime in the first sum denotes the restriction
that $\la$ must be even and the double prime in the second sum denotes
the restrictions that $\la'_1-\la'_{2k}=2\ell+2$, the last $2\ell+2$
parts of $\la$ are odd and the other parts are even.
Is it possible to extend this result to \eqref{Eq_generalise}
by introducing an additional statistic on cylindric tableaux?

A final question is to find partition theoretic interpretations
for the Rogers--Ramanujan identities listed in Section~\ref{Sec_standard}.
From the recent work in \cite{BKMP24,CMPP22,DK22,KRTW24,Primc24,PT25,Russell23}
on the Capparelli--Meurman--Primc--Primc (CMPP) conjectures, we know that,
at least conjecturally, the GOW identities of
Corollaries~\ref{Cor_1}--\ref{Cor_3} admit interpretations as
identities for the generating function of restricted sets of coloured
partitions.
It is natural to suspect that such an interpretation is not limited to
the first three corollaries, and it would be extremely
interesting to find a partition theoretic interpretation in terms of coloured
partitions for the remaining 14 corollaries.

\appendix

\section{Proof of (\ref{Eq_strange})}

In this appendix, we prove the full set of Rogers--Ramanujan identities \eqref{Eq_strange},
which for $1\leq i\leq k-1$ were conjectured by Matthew Russell after reading an earlier
version of this paper.

\begin{Theorem}\label{Thm_MR}
For $k$ a positive integer and $i$ an integer such that $0\leq i\leq k$,
\begin{align}
&\sideset{}{'}\sum_{n_1\geq\cdots\geq n_{2k}\geq 0}
\frac{q^{\frac{1}{2}(n_1^2+\dots+n_{2k}^2)+
\frac{1}{2}(n_1-n_2+\dots+n_{2k-1}-n_{2k})+
\frac{1}{2}(n_1+n_2+\dots+n_{2i})}}
{\bigl(q^2;q^2\bigr)_{\floor{(n_1-n_2)/2}}\cdots
\bigl(q^2;q^2\bigr)_{\floor{(n_{2k-1}-n_{2k})/2}}
(q;q)_{n_{2k}}} \notag \\
&\qquad{}=\frac{\bigl(-q^{2i+1},-q^{2i+1},q^{2i+1};q^{2i+1}\bigr)_{\infty}}
{\bigl(q^2;q^2\bigr)_{\infty}}, \label{Eq_MR}
\end{align}
where the prime denotes the restriction
that $n_{2j-1}-n_{2j}$ is even for all $1\leq j\leq k$.
\end{Theorem}

Before proving this theorem, we prepare some preliminary results.
For nonnegative integers~$r$ and $s$, let ${_r\phi_s}$ denote the
basic hypergeometric series
\[
\qhyp{r}{s}{a_1,\dots,a_r}{b_1,\dots,b_s}{q,z}=
\sum_{k=0}^{\infty} \frac{(a_1,\dots,a_r;q)_k}{(q,b_1,\dots,b_s;q)_k}
\biggl((-1)^k q^{\binom{k}{2}}\biggr)^{s-r+1} z^k.
\]

\begin{Lemma}
For $n$ a nonnegative integer,
\begin{equation}\label{Eq_RS-sum}
\sum_{r=0}^n w^r q^{\binom{r+1}{2}} H_{n-r}(z;q)\qbin{n}{r}_q=
\qhyp{2}{0}{-wq/z,q^{-n}}{\text{--}}{q,zq^n},
\end{equation}
where $H_m$ is a Rogers--Szeg\H{o} polynomial, see \eqref{Eq_RS}.
\end{Lemma}

\begin{proof}
Denote the expression on the left of \eqref{Eq_RS-sum} by $f_n(w,z;q)$.
Then
\[
f_n(w,z;q)
=\sum_{r=0}^n \sum_{k=0}^{n-r}
w^r z^k q^{\binom{r+1}{2}} \qbin{n-r}{k}_q\qbin{n}{r}_q
=\sum_{k=0}^n z^k \qbin{n}{k}_q
\sum_{r=0}^{n-k} w^r q^{\binom{r+1}{2}} \qbin{n-k}{r}_q.
\]
The sum over $r$ evaluates to $(-wq;q)_{n-k}$ by
the $q$-binomial theorem \cite[equation~(II.4)]{GR04}
\[
\sum_{k=0}^n (-z)^k q^{\binom{k}{2}}\qbin{n}{k}_q=(z;q)_n.
\]
Thus
\[
f_n(w,z;q)=\sum_{k=0}^n z^k (-wq;q)_{n-k} \qbin{n}{k}_q
=(-wq;q)_n\;\qhyp{2}{1}{0,q^{-n}}{-q^{-n}/w}{q,-\frac{z}{w}}.
\]
Taking the $b\to 0$ limit in the transformation formula
\cite[equation~(III.8)]{GR04} yields
\[
\qhyp{2}{1}{0,q^{-n}}{c}{q,z}=\frac{1}{\bigl(q^{1-n}/c;q\bigr)_n}\;
\qhyp{2}{0}{q/z,q^{-n}}{\text{--}}{q,\frac{z}{c}}.
\]
Applying this to the above expression for $f_n(w,z;q)$, we obtain
the right-hand side of \eqref{Eq_RS-sum}.
\end{proof}

\begin{Corollary}\label{Cor_sumtoone}
For $n$ a nonnegative integer,
\[
\sum_{r=0}^n w^r q^{\binom{r+1}{2}} H_{n-r}(-wq;q)\qbin{n}{r}_q=1.
\]
\end{Corollary}

\begin{proof}
We specialise $z=-wq$ in \eqref{Eq_RS-sum}.
Since the ${_2\phi_0}$ series has a numerator parameter equal to $1$, the series
trivialises to $1$, resulting in the claim.
\end{proof}

\begin{Corollary}\label{Cor_rs-sum}
For $n$ a nonnegative integer,
\begin{equation*}
g_n(q):=\sum_{\substack{0\leq s\leq r\leq n\\[1pt] r-s \textup{ even}}}
\frac{q^{\binom{r+1}{2}+\binom{s}{2}}(q;q)_n}
{\bigl(q^2;q^2\bigr)_{\floor{(n-r)/2}}\bigl(q^2;q^2\bigr)_{(r-s)/2}
(q;q)_s}=1.
\end{equation*}
\end{Corollary}

Before proving this result, we note that in the $n\to\infty$ limit it
yields \eqref{Eq_MR} for $i=0$ and $k=1$.

\begin{proof}
We begin by noting that
\[
\frac{1}{\bigl(q^2;q^2\bigr)_{\floor{m/2}}}=\frac{\bigl(q;q^2\bigr)_{\ceil{m/2}}}{(q;q)_m}
=\frac{H_m(-q;q)}{(q;q)_m},
\]
where the final equality follows from \cite[equation~(1.10c)]{W06a}.
Using the above identity and replacing the summation index
$s\mapsto r-2s$, it follows that
\begin{align*}
g_n(q)&=\sum_{r=0}^n \sum_{s=0}^{\floor{r/2}}
\frac{q^{r^2-2rs+s(2s+1)} (q;q)_n H_{n-r}(-q;q)}
{(q;q)_{n-r}\bigl(q^2;q^2\bigr)_s(q;q)_{r-2s}} \\
&=\sum_{r=0}^n q^{r^2} H_{n-r}(-q;q) \qbin{n}{r}_q\,
\qhyp{2}{1}{q^{-r},q^{1-r}}{0}{q^2,q^2}.
\end{align*}
According to the $c\to 0$ limit of the $q$-Chu--Vandermonde summation
\cite[equation~(II.6)]{GR04}
\[
\qhyp{2}{1}{a,q^{-n}}{0}{q,q}=a^n.
\]
Applying this for $(q,a,n)\mapsto \bigl(q^2,q^{1-2\ceil{r/2}},\floor{r/2}\bigr)$,
we find
\[
g_n(q)=\sum_{r=0}^n q^{\binom{r+1}{2}} H_{n-r}(-q;q) \qbin{n}{r}_q.
\]
By Corollary~\ref{Cor_sumtoone} for $w=1$, we are done.
\end{proof}

We are now ready to give a proof of Theorem~\ref{Thm_MR}
assuming the truth of the result for~${i=k}$.
This case is the $n=1$ instance of Corollary~\ref{Cor_16}.

\begin{proof}[Proof of Theorem~\ref{Thm_MR}]
As remarked above, we may assume the theorem holds for $i=k$.
Combining this with the remark immediately following
Corollary~\ref{Cor_rs-sum}, there is nothing to prove for $k=1$.
In the following we thus assume that $k\geq 2$ and $0\leq i<k$.
Now denote the sum on the left-hand side of \eqref{Eq_MR}
by $S_{i,k}(q)$ and its summand by $S_{i,k;n_1,\dots,n_{2k}}(q)$.
Then
\begin{align*}
S_{i,k}(q)&=\sideset{}{'}\sum_{n_1\geq\cdots\geq n_{2k}\geq 0}
S_{i,k;n_1,\dots,n_{2k}}(q) =\sideset{}{'}\sum_{n_1\geq\cdots\geq n_{2k-2}\geq 0}
S_{i,k-2;n_1,\dots,n_{2k-2}}(q) g_{n_{2k-2}(q)} \\
&=\sideset{}{'}\sum_{n_1\geq\cdots\geq n_{2k-2}\geq 0}
S_{i,k-2;n_1,\dots,n_{2k-2}}(q) =S_{i,k-1}(q),
\end{align*}
where the second equality uses the definition of $g_n(q)$
and the third equality follows from Corollary~\ref{Cor_rs-sum}.
If $i\geq 1$, this may be iterated to give
\[
S_{i,k}(q)=S_{i,i}(q)
=\frac{\bigl(-q^{2i+1},-q^{2i+1},q^{2i+1};q^{2i+1}\bigr)_{\infty}}
{\bigl(q^2;q^2\bigr)_{\infty}}.
\]
If, on the other hand, $i=0$, iteration leads to
\[
S_{0,k}(q)=S_{0,1}(q)=
\frac{(-q,-q,q;q)_{\infty}}{\bigl(q^2;q^2\bigr)_{\infty}}.
\]
This completes the proof.
\end{proof}

\subsection*{Acknowledgements}
I am grateful to Anne Schilling for helpful discussions, to
Matthew Russell for pointing out the full set of identities \eqref{Eq_strange},
and to both referees for their careful reading of the paper.
This work was supported by the Australian Research Council
Discovery Project DP200102316.

\pdfbookmark[1]{References}{ref}
\LastPageEnding


\begin{thebibliography}{99}
\footnotesize\itemsep=0pt

\bibitem{Andrews74}
Andrews G.E., An~analytic generalization of the {R}ogers--{R}amanujan
 identities for odd moduli,
 \href{https://doi.org/10.1073/pnas.71.10.4082}{\textit{Proc. Nat. Acad. Sci.
 USA}} \textbf{71} (1974), 4082--4085.

\bibitem{Andrews76}
Andrews G.E., The theory of partitions, \textit{Encyclopedia Math. Appl.},
 Vol.~2, Addison-Wesley Publishing Co., Reading, Mass., 1976.

\bibitem{Andrews84}
Andrews G.E., Multiple series {R}ogers--{R}amanujan type identities,
 \href{https://doi.org/10.2140/pjm.1984.114.267}{\textit{Pacific~J.~Math.}}
 \textbf{114} (1984), 267--283.

\bibitem{Andrews10}
Andrews G.E., Parity in partition identities,
 \href{https://doi.org/10.1007/s11139-008-9150-0}{\textit{Ramanujan~J.}}
 \textbf{23} (2010), 45--90.

\bibitem{ASW99}
Andrews G.E., Schilling A., Warnaar S.O., An~{$A_2$} {B}ailey lemma and
 {R}ogers--{R}amanujan-type identities,
 \href{https://doi.org/10.1090/S0894-0347-99-00297-0}{\textit{J.~Amer. Math.
 Soc.}} \textbf{12} (1999), 677--702,
 \href{http://arxiv.org/abs/math.QA/9807125}{arXiv:math.QA/9807125}.

\bibitem{Bailey48}
Bailey W.N., Identities of the {R}ogers--{R}amanujan type,
 \href{https://doi.org/10.1112/plms/s2-50.1.1}{\textit{Proc. London Math.
 Soc.}} \textbf{50} (1948), 1--10.

\bibitem{BW15}
Bartlett N., Warnaar S.O., Hall--{L}ittlewood polynomials and characters of
 affine {L}ie algebras,
 \href{https://doi.org/10.1016/j.aim.2015.08.011}{\textit{Adv. Math.}}
 \textbf{285} (2015), 1066--1105,
 \href{http://arxiv.org/abs/1304.1602}{arXiv:1304.1602}.

\bibitem{BBL08}
Belbachir H., Boussicault A., Luque J.-G., Hankel hyperdeterminants, rectangular
 {J}ack polynomials and even powers of the {V}andermonde,
 \href{https://doi.org/10.1016/j.jalgebra.2008.06.015}{\textit{J.~Algebra}}
 \textbf{320} (2008), 3911--3925,
 \href{http://arxiv.org/abs/0709.3021}{arXiv:0709.3021}.

\bibitem{BM96}
Berkovich A., McCoy B.M., Continued fractions and fermionic representations for
 characters of~{$M(p,p')$} minimal models,
 \href{https://doi.org/10.1007/BF00400138}{\textit{Lett. Math. Phys.}}
 \textbf{37} (1996), 49--66,
 \href{http://arxiv.org/abs/hep-th/9412030}{arXiv:hep-th/9412030}.

\bibitem{BM94}
Bos M.K., Misra K.C., Level two representations of~\smash{$A^{(2)}_7$} and
 {R}ogers--{R}amanujan identities,
 \href{https://doi.org/10.1080/00927879408825059}{\textit{Comm. Algebra}}
 \textbf{22} (1994), 3965--3983.

\bibitem{BIS00}
Bressoud D., Ismail M.E.H., Stanton D., Change of base in {B}ailey pairs,
 \href{https://doi.org/10.1023/A:1009824218230}{\textit{Ramanujan~J.}}
 \textbf{4} (2000), 435--453,
 \href{http://arxiv.org/abs/math.CO/9909053}{arXiv:math.CO/9909053}.

\bibitem{Bressoud80}
Bressoud D.M., An~analytic generalization of the {R}ogers--{R}amanujan
 identities with interpretation,
 \href{https://doi.org/10.1093/qmath/31.4.385}{\textit{Quart.~J.~Math. Oxford
 Ser.~(2)}} \textbf{31} (1980), 385--399.

\bibitem{Bressoud80b}
Bressoud D.M., Analytic and combinatorial generalizations of the
 {R}ogers--{R}amanujan identities,
 \href{https://doi.org/10.1090/memo/0227}{\textit{Mem. Amer. Math. Soc.}}
 \textbf{24} (1980), 54~pages.

\bibitem{Bressoud81}
Bressoud D.M., Some identities for terminating {$q$}-series,
 \href{https://doi.org/10.1017/S0305004100058114}{\textit{Math. Proc.
 Cambridge Philos. Soc.}} \textbf{89} (1981), 211--223.

\bibitem{BKMP24}
Butorac M., Ko\v{z}i\'c S., Meurman A., Primc M., Lepowsky's and {W}akimoto's
 product formulas for the affine {L}ie algebras~\smash{$C_l^{(1)}$},
 \href{https://doi.org/10.1016/j.jalgebra.2024.06.040}{\textit{J.~Algebra}}
 \textbf{660} (2024), 147--189,
 \href{http://arxiv.org/abs/2403.05456}{arXiv:2403.05456}.

\bibitem{Capparelli93}
Capparelli S., On some representations of twisted affine {L}ie algebras and
 combinatorial identities,
 \href{https://doi.org/10.1006/jabr.1993.1017}{\textit{J.~Algebra}}
 \textbf{154} (1993), 335--355.

\bibitem{Capparelli96}
Capparelli S., A~construction of the level~{$3$} modules for the affine {L}ie
 algebra~\smash{$A^{(2)}_2$} and a~new combinatorial identity of the
 {R}ogers--{R}amanujan type,
 \href{https://doi.org/10.1090/S0002-9947-96-01535-8}{\textit{Trans. Amer.
 Math. Soc.}} \textbf{348} (1996), 481--501.

\bibitem{CLM06}
Capparelli S., Lepowsky J., Milas A., The {R}ogers--{S}elberg recursions, the
 {G}ordon--{A}ndrews identities and intertwining operators,
 \href{https://doi.org/10.1007/s11139-006-0150-7}{\textit{Ramanujan~J.}}
 \textbf{12} (2006), 379--397,
 \href{http://arxiv.org/abs/math.QA/0310080}{arXiv:math.QA/0310080}.

\bibitem{CMPP22}
Capparelli S., Meurman A., Primc A., Primc M., New partition identities
 from~{$C^{(1)}_\ell$}-modules,
 \href{https://doi.org/10.3336/gm.57.2.01}{\textit{Glas. Mat. Ser.~{III}}}
 \textbf{57(77)} (2022), 161--184,
 \href{http://arxiv.org/abs/2106.06262}{arXiv:2106.06262}.

\bibitem{CSXY22}
Chern S., Li Z., Stanton D., Xue T., Yee A.J., The {A}riki--{K}oike algebras
 and {R}ogers--{R}amanujan type partitions,
 \href{https://doi.org/10.1007/s10801-024-01340-z}{\textit{J.~Algebraic
 Combin.}} \textbf{60} (2024), 491--540,
 \href{http://arxiv.org/abs/2209.07713}{arXiv:2209.07713}.

\bibitem{DK22}
Dousse J., Konan I., Characters of level~$1$ standard modules of~\smash{$C_n^{(1)}$} as
 generating functions for generalised partitions,
 \href{http://arxiv.org/abs/2212.12728}{arXiv:2212.12728}.

\bibitem{FQ95}
Foda O., Quano Y.-H., Polynomial identities of the {R}ogers--{R}amanujan type,
 \href{https://doi.org/10.1142/S0217751X9500111X}{\textit{Internat. J.~Modern
 Phys.~A}} \textbf{10} (1995), 2291--2315,
 \href{http://arxiv.org/abs/hep-th/9407191}{arXiv:hep-th/9407191}.

\bibitem{Frenkel82}
Frenkel I.B., Representations of affine {L}ie algebras, {H}ecke modular forms
 and {K}orteweg--de {V}ries type equations, in Lie Algebras and Related Topics
 ({N}ew {B}runswick, {N}.{J}., 1981), \textit{Lecture Notes in Math.}, Vol.~933, \href{https://doi.org/10.1007/BFb0093354}{Springer}, Berlin, 1982,
 71--110.

\bibitem{FK97}
Fulmek M., Krattenthaler C., Lattice path proofs for determinantal formulas for
 symplectic and orthogonal characters,
 \href{https://doi.org/10.1006/jcta.1996.2711}{\textit{J.~Combin. Theory
 Ser.~A}} \textbf{77} (1997), 3--50.

\bibitem{GR04}
Gasper G., Rahman M., Basic hypergeometric series, 2nd ed., \textit{Encyclopedia Math.
 Appl.}, Vol.~96,
 \href{https://doi.org/10.1017/CBO9780511526251}{Cambridge University Press},
 Cambridge, 2004.

\bibitem{GOW16}
Griffin M.J., Ono K., Warnaar S.O., A~framework of {R}ogers--{R}amanujan
 identities and their arithmetic properties,
 \href{https://doi.org/10.1215/00127094-3449994}{\textit{Duke Math.~J.}}
 \textbf{165} (2016), 1475--1527,
 \href{http://arxiv.org/abs/1401.7718}{arXiv:1401.7718}.

\bibitem{HKKO25}
Huh J.S., Kim J.S., Krattenthaler C., Okada S., Bounded {L}ittlewood identities
 for cylindric {S}chur functions,
 \href{https://doi.org/10.1090/tran/9500}{\textit{Trans. Amer. Math. Soc.}}
 \textbf{378} (2025), 6765--6829,
 \href{http://arxiv.org/abs/2301.13117}{arXiv:2301.13117}.

\bibitem{Ismail77}
Ismail M.E.H., A~simple proof of {R}amanujan's~{$_{1}\psi _{1}$} sum,
 \href{https://doi.org/10.2307/2041093}{\textit{Proc. Amer. Math. Soc.}}
 \textbf{63} (1977), 185--186.

\bibitem{Kac90}
Kac V.G., Infinite-dimensional {L}ie algebras, 3rd ed.,
 \href{https://doi.org/10.1017/CBO9780511626234}{Cambridge University Press},
 Cambridge, 1990.

\bibitem{Kanade18}
Kanade S., Structure of certain level~$2$ standard modules for~\smash{$A_5^{(2)}$}
 and the {G}\"ollnitz--{G}ordon identities,
 \href{https://doi.org/10.1007/s11139-016-9875-0}{\textit{Ramanujan~J.}}
 \textbf{45} (2018), 873--893.

\bibitem{KR15}
Kanade S., Russell M.C., {I}dentity{F}inder and some new identities of
 {R}ogers--{R}amanujan type,
 \href{https://doi.org/10.1080/10586458.2015.1015186}{\textit{Exp. Math.}}
 \textbf{24} (2015), 419--423,
 \href{http://arxiv.org/abs/1411.5346}{arXiv:1411.5346}.

\bibitem{KR26}
Kanade S., Russell M.C., Tight cylindric partitions,
 \href{https://doi.org/10.1090/proc/17685}{\textit{Proc. Amer. Math. Soc.}}, {t}o appear,
 \href{http://arxiv.org/abs/2508.15113}{arXiv:2508.15113}.

\bibitem{KRTW24}
Kanade S., Russell M.C., Tsuchioka S., Warnaar S.O., Remarks on the conjectures
 of {C}apparelli, {M}eurman, {P}rimc and {P}rimc, \textit{Selecta Math.~(N.S.)}, to appear,
 \href{http://arxiv.org/abs/2404.03851}{arXiv:2404.03851}.

\bibitem{KY13}
Kim S., Yee A.J., The {R}ogers--{R}amanujan--{G}ordon identities, the
 generalized {G}\"ollnitz--{G}ordon identities, and parity questions,
 \href{https://doi.org/10.1016/j.jcta.2013.02.005}{\textit{J.~Combin. Theory
 Ser.~A}} \textbf{120} (2013), 1038--1056.

\bibitem{Kirillov00}
Kirillov A.N., New combinatorial formula for modified {H}all--{L}ittlewood
 polynomials, in {$q$}-{S}eries from a~{C}ontemporary {P}erspective ({S}outh
 {H}adley, {MA}, 1998), \textit{Contemp. Math.}, Vol.~254,
 \href{https://doi.org/10.1090/conm/254/03959}{American Mathematical Society},
 Providence, RI, 2000, 283--333,
 \href{http://arxiv.org/abs/math.QA/9803006}{arXiv:math.QA/9803006}.

\bibitem{Krattenthaler90}
Krattenthaler C., Generating functions for plane partitions of a~given shape,
 \href{https://doi.org/10.1007/BF02567918}{\textit{Manuscripta Math.}}
 \textbf{69} (1990), 173--201.

\bibitem{Krattenthaler01}
Krattenthaler C., Proof of a~summation formula for an~\smash{$\tilde{A}_n$} basic
 hypergeometric series conjectured by {W}arnaar, in {$q$}-{S}eries with
 Applications to Combinatorics, Number Theory, and Physics ({U}rbana, {IL},
 2000), \textit{Contemp. Math.}, Vol.~291,
 \href{https://doi.org/10.1090/conm/291/04899}{American Mathematical Society},
 Providence, RI, 2001, 153--161,
 \href{http://arxiv.org/abs/math.CA/0209272}{arXiv:math.CA/0209272}.

\bibitem{LS06}
Lassalle M., Schlosser M., Inversion of the {P}ieri formula for {M}acdonald
 polynomials, \href{https://doi.org/10.1016/j.aim.2005.03.009}{\textit{Adv.
 Math.}} \textbf{202} (2006), 289--325,
 \href{http://arxiv.org/abs/math.CO/0402127}{arXiv:math.CO/0402127}.

\bibitem{Lee19}
Lee S.J., Positivity of cylindric skew {S}chur functions,
 \href{https://doi.org/10.1016/j.jcta.2019.05.009}{\textit{J.~Combin. Theory
 Ser.~A}} \textbf{168} (2019), 26--49,
 \href{http://arxiv.org/abs/1706.04460}{arXiv:1706.04460}.

\bibitem{Lepowsky79}
Lepowsky J., Generalized {V}erma modules, loop space cohomology and
 {M}ac{D}onald-type identities,
 \href{https://doi.org/10.24033/asens.1365}{\textit{Ann. Sci. \'Ecole Norm.
 Sup.}} \textbf{12} (1979), 169--234.

\bibitem{Lepowsky82}
Lepowsky J., Affine {L}ie algebras and combinatorial identities, in Lie
 Algebras and Related Topics ({N}ew {B}runswick, {N}.{J}., 1981),
 \textit{Lecture Notes in Math.}, Vol.~933, \href{https://doi.org/10.1007/BFb0093358}{Springer}, Berlin, 1982, 130--156.

\bibitem{LM78b}
Lepowsky J., Milne S., Lie algebraic approaches to classical partition
 identities, \href{https://doi.org/10.1016/0001-8708(78)90004-X}{\textit{Adv.
 Math.}} \textbf{29} (1978), 15--59.

\bibitem{LM78a}
Lepowsky J., Milne S., Lie algebras and classical partition identities,
 \href{https://doi.org/10.1073/pnas.75.2.578}{\textit{Proc. Nat. Acad. Sci.
 USA}} \textbf{75} (1978), 578--579.

\bibitem{LP85}
Lepowsky J., Primc M., Structure of the standard modules for the affine {L}ie
 algebra \smash{$A^{(1)}_1$}, \textit{Contemp. Math.}, Vol.~46,
 \href{https://doi.org/10.1090/conm/046}{American Mathematical Society},
 Providence, RI, 1985.

\bibitem{LW81b}
Lepowsky J., Wilson R.L., A~new family of algebras underlying the
 {R}ogers--{R}amanujan identities and generalizations,
 \href{https://doi.org/10.1073/pnas.78.12.7254}{\textit{Proc. Nat. Acad. Sci.
 {USA}}} \textbf{78} (1981), 7254--7258.

\bibitem{LW81a}
Lepowsky J., Wilson R.L., The {R}ogers--{R}amanujan identities: {L}ie theoretic
 interpretation and proof,
 \href{https://doi.org/10.1073/pnas.78.2.699}{\textit{Proc. Nat. Acad. Sci.
 USA}} \textbf{78} (1981), 699--701.

\bibitem{LW82}
Lepowsky J., Wilson R.L., A~{L}ie theoretic interpretation and proof of the
 {R}ogers--{R}amanujan identities,
 \href{https://doi.org/10.1016/S0001-8708(82)80012-1}{\textit{Adv. Math.}}
 \textbf{45} (1982), 21--72.

\bibitem{LW84}
Lepowsky J., Wilson R.L., The structure of standard modules.~{I}. {U}niversal
 algebras and the {R}ogers--{R}amanujan identities,
 \href{https://doi.org/10.1007/BF01388447}{\textit{Invent. Math.}} \textbf{77}
 (1984), 199--290.

\bibitem{Littlewood50}
Littlewood D.E., The theory of group characters and matrix representations of
 groups, Oxford University Press, New York, 1940.

\bibitem{Macdonald72}
Macdonald I.G., Affine root systems and {D}edekind's {$\eta $}-function,
 \href{https://doi.org/10.1007/BF01418931}{\textit{Invent. Math.}} \textbf{15}
 (1972), 91--143.

\bibitem{Macdonald95}
Macdonald I.G., Symmetric functions and {H}all polynomials, 2nd ed., \textit{Oxford Math. Monogr.},
 \href{https://doi.org/10.1093/oso/9780198534891.001.0001}{Oxford University
 Press}, New York, 1995.

\bibitem{Macdonald00}
Macdonald I.G., Orthogonal polynomials associated with root systems,
 \textit{S\'em. Lothar. Combin.} \textbf{45} (2000), Art.~B45a, 40~pages,
 \href{http://arxiv.org/abs/math.QA/0011046}{arXiv:math.QA/0011046}.

\bibitem{Matsumoto08}
Matsumoto S., Hyperdeterminantal expressions for {J}ack functions of
 rectangular shapes,
 \href{https://doi.org/10.1016/j.jalgebra.2007.09.013}{\textit{J.~Algebra}}
 \textbf{320} (2008), 612--632,
 \href{http://arxiv.org/abs/math.CO/0603033}{arXiv:math.CO/0603033}.

\bibitem{McLaughlin18}
McLaughlin J., Topics and methods in {$q$}-series, \textit{Monogr. Number
 Theory}, Vol.~8, \href{https://doi.org/10.1142/10528}{World Scientific
 Publishing}, Hackensack, NJ, 2018.

\bibitem{McNamara06}
McNamara P., Cylindric skew {S}chur functions,
 \href{https://doi.org/10.1016/j.aim.2005.07.011}{\textit{Adv. Math.}}
 \textbf{205} (2006), 275--312,
 \href{http://arxiv.org/abs/math.CO/0410301}{arXiv:math.CO/0410301}.

\bibitem{MP87}
Meurman A., Primc M., Annihilating ideals of standard modules
 of~{$\mathfrak{sl}(2,\mathbb{C})^\sim$} and combinatorial identities,
 \href{https://doi.org/10.1016/0001-8708(87)90008-9}{\textit{Adv. Math.}}
 \textbf{64} (1987), 177--240.

\bibitem{MP99}
Meurman A., Primc M., Annihilating fields of standard modules of
 {$\mathfrak{sl}(2,\mathbb{C})^\sim$} and combinatorial identities,
 \href{https://doi.org/10.1090/memo/0652}{\textit{Mem. Amer. Math. Soc.}}
 \textbf{137} (1999), viii+89~pages,
 \href{http://arxiv.org/abs/math.QA/9806105}{arXiv:math.QA/9806105}.

\bibitem{Misra84}
Misra K.C., Structure of certain standard modules for~\smash{$A^{(1)}_{n}$} and the
 {R}ogers--{R}amanujan identities,
 \href{https://doi.org/10.1016/0021-8693(84)90098-X}{\textit{J.~Algebra}}
 \textbf{88} (1984), 196--227.

\bibitem{Nandi14}
Nandi D., Partition identities arising from the standard \smash{$A^{(2)}_2$}-modules of
 level~$4$, Ph.D.~Thesis, {T}he State University of New Jersey, 2014,
 available at \url{https://doi.org/doi:10.7282/T3154G81}.

\bibitem{Postnikov05}
Postnikov A., Affine approach to quantum {S}chubert calculus,
 \href{https://doi.org/10.1215/S0012-7094-04-12832-5}{\textit{Duke Math.~J.}}
 \textbf{128} (2005), 473--509,
 \href{http://arxiv.org/abs/math.CO/0205165}{arXiv:math.CO/0205165}.

\bibitem{Primc24}
Primc M., New partition identities for odd~{$W$} odd,
 \href{https://doi.org/10.21857/9e31lhzl8m}{\textit{Rad Hrvat. Akad. Znan.
 Umjet. Mat. Znan.}} \textbf{28(558)} (2024), 49--56.

\bibitem{PT25}
Primc M., Trup\v{c}evi\'c G., Linear independence for~\smash{$C_\ell^{(1)}$} by
 using~\smash{$C_{2\ell}^{(1)}$},
 \href{https://doi.org/10.1016/j.jalgebra.2024.08.003}{\textit{J.~Algebra}}
 \textbf{661} (2025), 341--356,
 \href{http://arxiv.org/abs/2403.06881}{arXiv:2403.06881}.

\bibitem{RW21}
Rains E.M., Warnaar S.O., Bounded {L}ittlewood identities,
 \href{https://doi.org/10.1090/memo/1317}{\textit{Mem. Amer. Math. Soc.}}
 \textbf{270} (2021), vii+115~pages,
 \href{http://arxiv.org/abs/1506.02755}{arXiv:1506.02755}.

\bibitem{Rogers94}
Rogers L.J., Second memoir on the expansion of certain infinite products,
 \href{https://doi.org/10.1112/plms/s1-25.1.318}{\textit{Proc. Lond. Math.
 Soc.}} \textbf{25} (1893), 318--343.

\bibitem{Rogers17}
Rogers L.J., On two theorems of combinatory analysis and some allied
 identities, \href{https://doi.org/10.1112/plms/s2-16.1.315}{\textit{Proc.
 London Math. Soc.}} \textbf{16} (1917), 315--336.

\bibitem{RR19}
Rogers L.J., Ramanujan S., Proof of certain identities in combinatory analysis,
 \textit{Proc. Cambridge Phil. Soc.} \textbf{19} (1919), 211--216.

\bibitem{Russell23}
Russell M.C., Companions to the {A}ndrews--{G}ordon and {A}ndrews--{B}ressoud
 identities and recent conjectures of {C}apparelli, {M}eurman, {P}rimc, and
 {P}rimc, \href{https://doi.org/10.3842/SIGMA.2026.046}{\textit{SIGMA}}
 \textbf{22} (2026), 046, 38~pages,
 \href{http://arxiv.org/abs/2306.16251}{arXiv:2306.16251}.

\bibitem{SS00}
Schilling A., Shimozono M., Bosonic formula for level-restricted paths, in
 Combinatorial Methods in Representation Theory ({K}yoto, 1998), \textit{Adv.
 Stud. Pure Math.}, Vol.~28,
 \href{https://doi.org/10.2969/aspm/02810305}{Kinokuniya}, Tokyo, 2000,
 305--325, \href{http://arxiv.org/abs/math.QA/9812106}{arXiv:math.QA/9812106}.

\bibitem{SS01}
Schilling A., Shimozono M., Fermionic formulas for level-restricted generalized
 {K}ostka polynomials and coset branching functions,
 \href{https://doi.org/10.1007/s002200100443}{\textit{Comm. Math. Phys.}}
 \textbf{220} (2001), 105--164,
 \href{http://arxiv.org/abs/math.QA/0001114}{arXiv:math.QA/0001114}.

\bibitem{SW99}
Schilling A., Warnaar S.O., Inhomogeneous lattice paths, generalized {K}ostka
 polynomials and~{$A_{n-1}$} supernomials,
 \href{https://doi.org/10.1007/s002200050586}{\textit{Comm. Math. Phys.}}
 \textbf{202} (1999), 359--401,
 \href{http://arxiv.org/abs/math.QA/9802111}{arXiv:math.QA/9802111}.

\bibitem{Sills03}
Sills A.V., Finite {R}ogers--{R}amanujan type identities,
 \href{https://doi.org/10.37236/1706}{\textit{Electron.~J.~Combin.}}
 \textbf{10} (2003), 13, 122~pages,
 \href{http://arxiv.org/abs/1901.02435}{arXiv:1901.02435}.

\bibitem{Sills18}
Sills A.V., An~invitation to the {R}ogers--{R}amanujan identities,
 \href{https://doi.org/10.1201/9781315151922}{CRC Press}, Boca Raton, FL,
 2018.

\bibitem{Slater51}
Slater L.J., A~new proof of {R}ogers's transformations of infinite series,
 \href{https://doi.org/10.1112/plms/s2-53.6.460}{\textit{Proc. London Math.
 Soc.}} \textbf{53} (1951), 460--475.

\bibitem{Stembridge90}
Stembridge J.R., Hall--{L}ittlewood functions, plane partitions, and the
 {R}ogers--{R}amanujan identities,
 \href{https://doi.org/10.2307/2001250}{\textit{Trans. Amer. Math. Soc.}}
 \textbf{319} (1990), 469--498.

\bibitem{Stembridge91}
Stembridge J.R., Nonintersecting paths, {P}faffians, and plane partitions,
 \href{https://doi.org/10.1016/0001-8708(90)90070-4}{\textit{Adv. Math.}}
 \textbf{83} (1990), 96--131.

\bibitem{Venkateswaran15}
Venkateswaran V., Symmetric and nonsymmetric {K}oornwinder polynomials in the
 {$q\rightarrow 0$} limit,
 \href{https://doi.org/10.1007/s10801-015-0583-4}{\textit{J.~Algebraic
 Combin.}} \textbf{42} (2015), 331--364.

\bibitem{W25}
Warnaar S.O., An~$A_2$ Bailey tree and \smash{$A_2^{(1)}$} {R}ogers--{R}amanujan-type
 identities, \href{https://doi.org/10.4171/JEMS/1627}{\textit{J.~Eur. Math.
 Soc.}}, {t}o appear, \href{http://arxiv.org/abs/2303.09069}{arXiv:2303.09069}.

\bibitem{W99}
Warnaar S.O., Supernomial coefficients, {B}ailey's lemma and
 {R}ogers--{R}amanujan-type identities. {A}~survey of results and open
 problems, \textit{S\'em. Lothar. Combin.} \textbf{42} (1999), Art. B42n,
 22~pages.

\bibitem{W01a}
Warnaar S.O., 50 years of {B}ailey's lemma, in Algebraic Combinatorics and
 Applications ({G}\"o\ss weinstein, 1999),
 \href{https://doi.org/10.1007/978-3-642-59448-9_23}{Springer}, Berlin, 2001,
 333--347, \href{http://arxiv.org/abs/0910.2062}{arXiv:0910.2062}.

\bibitem{W01}
Warnaar S.O., The generalized {B}orwein conjecture.~{I}. {T}he {B}urge
 transform, in {$q$}-Series with Applications to Combinatorics, Number Theory,
 and Physics ({U}rbana, {IL}, 2000), \textit{Contemp. Math.}, Vol.~291,
 \href{https://doi.org/10.1090/conm/291/04906}{American Mathematical Society},
 Providence, RI, 2001, 243--267,
 \href{http://arxiv.org/abs/math.CO/0011220}{arXiv:math.CO/0011220}.

\bibitem{W03}
Warnaar S.O., The generalized {B}orwein conjecture.~{II}. {R}efined
 {$q$}-trinomial coefficients,
 \href{https://doi.org/10.1016/S0012-365X(03)00047-5}{\textit{Discrete Math.}}
 \textbf{272} (2003), 215--258,
 \href{http://arxiv.org/abs/math.CO/0110307}{arXiv:math.CO/0110307}.

\bibitem{W06a}
Warnaar S.O., Rogers--{S}zeg\H{o} polynomials and {H}all--{L}ittlewood
 symmetric functions,
 \href{https://doi.org/10.1016/j.jalgebra.2006.04.026}{\textit{J.~Algebra}}
 \textbf{303} (2006), 810--830,
 \href{http://arxiv.org/abs/0708.3110}{arXiv:0708.3110}.

\bibitem{Watson29}
Watson G.N., Theorems stated by {R}amanujan~({VII}): {T}heorems on continued
 fractions, \href{https://doi.org/10.1112/jlms/s1-4.1.39}{\textit{J.~London
 Math. Soc.}} \textbf{4} (1929), 39--48.

\end{thebibliography}
\end{document}